\documentclass[twoside,11pt,reqno]{amsart}
\usepackage{natbib}
\usepackage{pstricks}
\usepackage{a4wide}
\usepackage[latin1]{inputenc}
\usepackage{hyperref}
\usepackage{graphicx}
\usepackage{amsmath}
\usepackage{color}
\usepackage{amsmath,amssymb,bbm,enumitem}
\usepackage{epic,eepic}
\newtheorem{theorem}{Theorem}[section]

\newtheorem{corollary}[theorem]{Corollary}
\newtheorem{definition}{Definition}[section]
\newtheorem{lemma}[theorem]{Lemma}
\newtheorem{proposition}[theorem]{Proposition}

\theoremstyle{definition}
\newtheorem{remark}[theorem]{Remark}
\newcommand{\rmd}{\mathrm{d}}
\newcommand{\rme}{\mathrm{e}}
\newcommand{\rmi}{\mathrm{i}}
\newcommand{\1}{\mathbbm{1}}

\newcommand{\ellset}{\mathcal{L}}
\newcommand{\K}{\mathbf{L}}

\newcommand{\bh}{\mathbf{h}}
\newcommand{\bW}{\mathbf{W}}
\newcommand{\bx}{\mathbf{x}}

\newcommand{\bZ}{\mathbf{Z}}
\newcommand{\bS}{\mathbf{S}}

\newcommand{\bof}{\mathbf{f}}

\newcommand{\bkappa}{\boldsymbol{\kappa}}

\title[Higher order chaotic limits of wavelets scalograms]{High order chaotic limits of wavelet
  scalograms under long--range dependence}

\author{M. Clausel}
\author{F. Roueff}
\author{M.~S. Taqqu}
\author{C. Tudor}

\address{Laboratoire Jean Kuntzmann\\
Université de Grenoble, CNRS\\
F38041 Grenoble Cedex 9}
\email{marianne.clausel@imag.fr}
\urladdr{\url{http://www-ljk.imag.fr/membres/Marianne.Clausel/}}
\address{Institut Mines--Telecom, Telecom ParisTech, CNRS LTCI, 46 rue Barrault\\
               75634 Paris Cedex 13, France}
\email{roueff@telecom-paristech.fr}
\urladdr{\url{http://perso.telecom-paristech.fr/~roueff/}}
\address{Departement of Mathematics and Statistics, Boston University,
  Boston, MA 02215, USA}
\email{murad@math.bu.edu}
\urladdr{\url{http://math.bu.edu/people/murad/}}
\address{Laboratoire Paul Painlev\'e, UMR 8524 du CNRS, Universit\'e
  Lille 1, 59655 Villeneuve d'Ascq, France. Associate member: Center for Applied Mathematics, Academy for Economical Studies, Bucharest, Romania}
\email{Ciprian.Tudor@math.univ-lille1.fr}
\urladdr{\url{https://sites.google.com/site/ciprianatudor/}}

\thanks{M. Clausel's research was partially supported by the PEPS project \emph{AGREE} and LabEx \emph{PERSYVAL-Lab} (ANR-11-LABX-0025-01) funded by the French program Investissement d'avenir.\\
 F. Roueff's research was
partially supported by the ANR project \emph{MATAIM} NT09 441552\\
Murad~S.Taqqu was supported in part by the
  NSF grants DMS--1007616 at Boston University.\\
  C. Tudor's research was
  partially supported by the ANR grant \emph{Masterie} BLAN 012103.}

\begin{document}

\begin{abstract}
  Let $G$ be a non--linear function of a Gaussian process
  $\{X_t\}_{t\in\mathbb{Z}}$ with long--range dependence. The resulting process
  $\{G(X_t)\}_{t\in\mathbb{Z}}$ is not Gaussian when $G$ is not linear. We
  consider random wavelet coefficients associated with
  $\{G(X_t)\}_{t\in\mathbb{Z}}$ and the corresponding wavelet scalogram which
  is the average of squares of wavelet coefficients over locations. We obtain
  the asymptotic behavior of the scalogram as the number of observations and
  the analyzing scale tend to infinity. It is known that when $G$ is a Hermite
  polynomial of any order, then the limit is either the Gaussian or the
  Rosenblatt distribution, that is, the limit can be represented by a multiple
  Wiener-Itô integral of order one or two. We show, however, that there are
  large classes of functions $G$ which yield a higher order Hermite
  distribution, that is, the limit can be represented by a a multiple
  Wiener-Itô integral of order greater than two. This happens for example if
  $G$ is a linear combination of a Hermite polynomial of order $1$ and a
  Hermite polynomial of order $q>3$. The limit in this case can be Gaussian but
  it can also be a Hermite distribution of order $q-1>2$. This depends not only
  on the relation between the number of observations and the scale size but
  also on whether $q$ is larger or smaller than a new critical index $q^*$. The
  convergence of the wavelet scalogram is therefore significantly more complex
  than the usual one.
 \end{abstract}
 \maketitle
\newpage
\tableofcontents

\section{Introduction}\label{s:intro}
Denote by $X=\{X_{t}\}_{t\in\mathbb{Z}}$ a centered stationary
Gaussian process with unit variance and spectral density
$f(\lambda), \lambda \in (-\pi , \pi)$. Such a stochastic process
is said to have {\it
  short memory} or {\it short--range dependence} if $f(\lambda)$ is bounded
around $\lambda=0$ and {\it long memory} or {\it long--range
dependence} if $f(\lambda)\to\infty$ as $\lambda\to0$. We
will suppose that $\{X_{t}\}_{t\in\mathbb{Z}}$ has long memory
with memory parameter $0<d<1/2$, that is,
\begin{equation}\label{e:sdf0}
f(\lambda)\sim |\lambda|^{-2d}f^*(\lambda)\mbox{ as }\lambda \to 0
\end{equation}
where $f^*(\lambda)$ is a bounded spectral density which is
continuous and positive at the origin. This hypothesis is
semi--parametric in nature because the function $f^*$ plays the
role of a ``nuisance function''. It is convenient to set
\begin{equation}\label{e:sdf}
f(\lambda)=|1-\rme^{-\rmi\lambda}|^{-2d}f^*(\lambda),\quad\lambda\in
(-\pi,\pi]\;.
\end{equation}
Since the process $X$ is defined only if $\int_{-\pi }^{\pi}
f(\lambda)d\lambda <\infty$, we need to require $d<\frac{1}{2}$.

Consider now a process $\{Y_{t}\}_{t\in\mathbb{Z}}$, such that
\begin{equation}\label{e:DefY}
\left(\Delta^K Y\right)_{t}=G(X_{t}),\quad t\in\mathbb{Z}\;,
\end{equation}
for $K\geq 0$, where $(\Delta Y)_{t}=Y_{t}-Y_{t-1}$,
$\{X_t\}_{t\in\mathbb{Z}}$ is Gaussian with spectral density $f$
satisfying~(\ref{e:sdf}) and where $G$ is a function such that
$\mathbb{E}[G(X_{t})]=0$ and $\mathbb{E}[G(X_{t})^2]<\infty$.
While the process $\{Y_t\}_{t\in\mathbb{Z}}$ is not necessarily
stationary, its $K$--th difference $\Delta^K Y_t$ is stationary
and is the output of a non--linear filter $G$ with Gaussian
input.

We shall study the asymptotic behavior of the wavelet scalogram of
$\{Y_t\}_{t\in\mathbb{Z}}$, that is, the average of squares of its wavelet
coefficients. As shown in~\cite{flandrin:1992}, \cite{abry:veitch:1998},
\cite{veitch:abry:1999} and~\cite{bardet:2000T} in a parametric context, the
normalized limit of scalogram can be used to estimate the long memory exponent
$d$ defined in~(\ref{e:sdf0}).

Empirical studies presented in~\cite{abry-helgason-pipiras-2011} consider the problem of estimating $d$
under various types of functions $G$. The argument, consistent with the one
in~\cite{clausel-roueff-taqqu-tudor-2011a}, suggests that at large scales the
wavelet coefficients behavior only depends on the ``Hermite rank'', which is
defined below, of $G$. Moreover the authors develop heuristical arguments to
deduce the asymptotic behavior of wavelet-based regression estimator of $d$.
We provide here a theoretical analysis in a semi--parametric setting for a large class of functions $G$. We will show that,
as $j$ goes to infinity, there is a delicate interplay between the scale
$\gamma_j$ (typically $2^j$) and the number of wavelet coefficients $n_j$ and that
the ``reduction theorem'' (see below) applies only when $\gamma_j$ is much greater than $n_j$.

In the semi--parametric context, the case where the function $G$ is linear was
firstly considered in \cite{moulines:roueff:taqqu:2007:jtsa} and the case where
$G$ is a Hermite polynomial of arbitrary order was studied
in~\cite{clausel-roueff-taqqu-tudor-2011b}. The case where $G(X_t)$ is the
so--called ``Rosenblatt process'' was studied by \cite{bardet:tudor:2010} (see also~\cite{tudor:2013}) and
is somewhat analogous to the one where $G$ is the second Hermite
polynomial. Our goal is to show that for more complicated functions $G$, one
can obtain new types of limits.

We have referred to Hermite polynomials a number of times. This is because they
form a basis for the space of functions $G$ and thus appear
naturally in our setting. Since the function $G$ satisfies $\mathbb{E}[G(X)]=0$
and $\mathbb{E}[G(X)^2]<\infty$ for $X\sim\mathcal{N}(0,1)$, $G(X)$ can be
expanded in Hermite polynomials, that is,
\begin{equation}\label{e:chaos-exp}
G(X)=\sum_{q=1}^{\infty}\frac{c_q}{q!} H_q(X)\;.
\end{equation}
One sometimes refer to~(\ref{e:chaos-exp}) as an expansion in Wiener chaos.
The convergence of the infinite sum~(\ref{e:chaos-exp}) is in
$L^2(\Omega)$,
\begin{equation}\label{e:cq}
c_q=\mathbb{E}[G(X)H_q(X)]\;,\quad q\geq 1\;,
\end{equation}
and
\begin{equation*}
H_{q}(x)= (-1)^{q}e^{\frac{x^{2}}{2}}\frac{d^{q}}{dx^{q}}\left(
e^{-\frac{x^{2}}{2}}\right)\;,
\end{equation*}
are the Hermite polynomials. These  Hermite polynomials satisfy
$H_0(x)=1,H_1(x)=x,H_2(x)=x^2-1$ and one has
\[
\mathbb{E}[H_q(X)H_{q'}(X)]=
\int_\mathbb{R}H_q(x)H_{q'}(x)\frac{1}{\sqrt{2\pi}}\rme^{-x^2/2}\rmd
x=q!\1_{\{q=q'\}}\;.
\]
Observe that the expansion~(\ref{e:chaos-exp}) starts at $q=1$,
since
\begin{equation}\label{e:c0}
c_0=\mathbb{E}[G(X)H_0(X)]=\mathbb{E}[G(X)]=0\;,
\end{equation}
by assumption. Denote by $q_0\geq 1$ the {\it Hermite rank} of
$G$, namely the index of the first non--zero coefficient in the
expansion~(\ref{e:chaos-exp}). Formally, $q_0$ is such that
\begin{equation}
  \label{e:hermiterank}
q_0=\min\{q\geq1,\,c_q\neq 0\}\;.
\end{equation}
One has then
\begin{equation}\label{e:summability}
\sum_{q=q_0}^{+\infty}\frac{c_q^2}{q!}=\mathbb{E}[G(X)^2]<\infty\;.
\end{equation}

We will focus on the wavelet coefficients of the sequence
$\{Y_t\}_{t\in\mathbb{Z}}$ in~(\ref{e:DefY}). Since
$\{Y_{t}\}_{t\in\mathbb{Z}}$ is random so will be its wavelet coefficients
which we denote by $\{W_{j,k},\,j\geq 0,\,k\in\mathbb{Z}\}$, where $j$
indicates the scale index and $k$ the location. These wavelet coefficients are
defined by
\begin{equation}\label{e:WC}
W_{j,k}=\sum_{t\in\mathbb{Z}}h_j(\gamma_j k-t)Y_{t}\;,
\end{equation}
where $\gamma_j\uparrow \infty$ as $j\uparrow \infty$ is a
sequence of non--negative decimation factors applied at scale index $j$, for
example $\gamma_j=2^j$ and $h_{j}$ is a filter whose properties
are listed in Appendix~\ref{s:appendixB}. We follow the
engineering convention where large values of $j$ correspond to
large scales. Our goal is to find the distribution of the
empirical quadratic mean of these wavelet coefficients at large
scales $j\to\infty$, that is, the asymptotic behavior of the
wavelet scalogram
\begin{equation}\label{e:defsnj}
S_{n_j,j}=\frac{1}{n_j}\sum_{k=0}^{n_j-1}W_{j,k}^2
\;,
\end{equation}
adequately centered and normalized as the scale $\gamma_j$ and the number of
wavelets coefficients $n_j$ available at scale index $j$ both tend to infinity.

The reduction theorem of~\cite{taqqu:1975} states that if $G(X_t)$
is long--range dependent then the limit in the sense of
finite--dimensional distributions of $\sum_{k=1}^{[nt]}G(X_k)$
adequately normalized, depends on the first term in the Hermite
expansion of $G$. In other words, there exist normalization
factors $a_n\to\infty$ as $n\to\infty$ such that
\[
\frac{1}{a_n}\sum_{k=1}^{[nt]}G(X_k)\qquad\mbox{ and
}\qquad\frac{1}{a_n}\sum_{k=1}^{[nt]}\frac{c_{q_0}}{q_0!}H_{q_0}(X_k)\;,
\]
have the same non--degenerate limit as $n\to\infty$.

We are interested here, however, in the asymptotic behavior of the wavelet
scalogram $S_{n_j,j}$ in~(\ref{e:defsnj}). We want to find exponents $\alpha>0$
and $\nu>0$ such that as the number of wavelet coefficients $n_j$ and the scale
$\gamma_j$ tend to $\infty$,
\begin{equation}\label{e:pbpose}
\{n_j^{\alpha}\gamma_j^{-\nu}S_{n_{j+u},j+u},u\in\mathbb{Z}\}\;,
\end{equation}
tends, after centering, to a limit in the sense of the finite--dimensional
distributions in the scale increment $u$.  This is a necessary and important step in
developing methods for estimating the underlying long memory parameter.

The limit of the sequence $S_{n_j,j}$ will be related to the so--called Hermite
process. The Hermite process is a self-similar stochastic process, with
stationary increments and long range dependence. The Hermite process of order
$q$ lives in the $q$th Wiener chaos, that is, it can be written as an iterated
multiple integral of order $q$ with respect to white noise. We refer to
Definition~\ref{d:HP} below for the precise representation.

We will see that, in the scalogram setting, the reduction theorem mentioned
above does not always apply. For example if $G(X_t)=H_1(X_t)+H_{q_1}(X_t)$,
$q_1\geq 3$ then the Hermite rank is $q_0=1$. But the limit of the normalized
scalogram is not necessarily the same as that of $H_1(X_t)=X_t$. This is
essentially due to the fact that the scalogram involves squares and, in
addition, depends on two parameters $j$ and $n_j$ which both tend to $\infty$.

In~\cite{clausel-roueff-taqqu-tudor-2011b}, the case
$$
G(X_t)=H_q(X_t), \quad  q\geq 2\;,
$$
was studied and it was shown that in this case the limit is a Rosenblatt
process (see Definition~\ref{d:HP}). In the present paper we study other
classes of functions $G$ for which different Hermite processes appear in the
limit. For example, for the process
$$
G(X_t)=H_1(X_t)+H_{q_1}(X_t), \quad q_1\geq 3\;,
$$
considered above, the limit of~(\ref{e:pbpose}) may be either Gaussian, a
Hermite process of order $q_1-1$ or a Rosenblatt process depending on the
specific circumstances. We will show the existence of a {\it critical index}
$q_1^*$ and of critical exponents $\nu,\nu'$ such that when $q_1<q_1^*$, then~:
\begin{itemize}
\item the limit is Gaussian if $n_j\ll\gamma_j^{\nu}$, \item the
limit is a Hermite process of order $q_1-1$ if $\gamma_j^{\nu}\ll
n_j\ll\gamma_j^{\nu'}$, \item the limit is a Rosenblatt process if
$\gamma_j^{\nu'}\ll n_j$,
\end{itemize}
where $a_j\ll b_j$ means that $a_j=o(b_j)$ as
$j\to\infty$.

We will also study interesting cases where the function $G$ has a Hermite rank
greater than two.

The paper is organized as follows. Long range--dependence and the
multidimensional wavelet scalogram are introduced in
Section~\ref{s:LRD}. The main results are stated and illustrated
in Section~\ref{s:main}. The chaos decomposition of the scalogram
is given in Section~\ref{s:expansion}. The study of the leading
terms is done in Sections~\ref{s:preliminary}
and~\ref{sec:LTSigma}. The proofs of the main theorems are given
in Section~\ref{s:proofsmain} while Section~\ref{sec:techlemma}
contains some technical lemmas. Basic facts about the Wiener chaos
are gathered in Appendix~\ref{s:appendixA} and
Appendix~\ref{s:appendixB} lists the assumptions on the wavelet
filters.

\section{Long--range dependence and the multidimensional wavelet scalogram}\label{s:LRD}

The Gaussian sequence $X=\{X_t\}_{t\in\mathbb{Z}}$ with spectral
density~(\ref{e:sdf}) is long--range dependent because $d>0$ and
hence its spectrum explodes at $\lambda=0$. Whether
$\{H_q(X_t)\}_{t\in\mathbb{Z}}$ is also long--range dependent
depends on the respective values of $q$ and $d$.  We show
in~\cite{clausel-roueff-taqqu-tudor-2011a}, that the spectral density of
$\{H_q(X_t)\}_{t\in\mathbb{Z}}$ behaves like
$|\lambda|^{-2\delta_+(q)}$ as $\lambda\to 0$, where
\begin{equation}\label{e:ldparamq}
\delta_+(q)=\max(\delta(q),0)\quad\text{and}\quad
\delta(q)=qd-(q-1)/2\;.
\end{equation}
Hence $\delta_+(q)$ is the memory parameter of
$\{H_q(X_t)\}_{t\in\mathbb{Z}}$. Therefore, since $0<d<1/2$,
$\{H_q(X_t)\}_{t\in\mathbb{Z}}$, $q\geq 1$, is
long--range dependent if and only if
\begin{equation}\label{e:dq>0}
\delta(q)>0\Longleftrightarrow (1/2)(1-1/q)<d<1/2\;,
\end{equation}
that is, $d$ must be sufficiently close to $1/2$. Specifically,
for long--range dependence,
\begin{equation}
  \label{eq:qd}
q=1\Rightarrow d>0,\quad q=2\Rightarrow d>1/4,\quad q=3\Rightarrow
d>1/3,\quad q=4\Rightarrow d>3/8\;.
\end{equation}
From another perspective,
\begin{equation}\label{e:dqq>0}
\delta(q)>0\Longleftrightarrow 1\leq q<1/(1-2d)\;,
\end{equation}
and thus $\{H_q(X_t)\}_{t\in\mathbb{Z}}$ is short--range dependent
if $q\geq 1/(1-2d)$.

We shall suppose that the Hermite rank of $G$
is $q_0\geq1$, that is the expansion of $G(X_{t})$ starts at $q_0$. We
always assume that $\{H_{q_0}(X_t)\}_{t\in\mathbb{Z}}$ has long
memory, that is,
\begin{equation}\label{e:longmemorycondition}
q_0 < 1/(1-2d) \;.
\end{equation}

The condition~(\ref{e:longmemorycondition}), with $q_0$ defined as the Hermite
rank~(\ref{e:hermiterank}), ensures such that
$\{\Delta^KY\}_{t\in\mathbb{Z}}=\{G(X_t)\}_{t\in\mathbb{Z}}$ is long-range
dependent (see~\cite{clausel-roueff-taqqu-tudor-2011a}, Lemma 4.1). We are
mainly interested in the asymptotic behavior of the scalogram $S_{n_j,j}$,
defined by~(\ref{e:defsnj}) as $n_j\to\infty$ (large sample behavior) and
$j\to\infty$ (large scale behavior). More precisely, we will study the
asymptotic behavior of the sequence
\begin{equation}\label{e:snjm}
\overline{S}_{n_{j+u},j+u}=
S_{n_{j+u},j+u}-\mathbb{E}(S_{n_{j+u},j+u})=\frac{1}{n_{j+u}}\sum_{k=0}^{n_{j+u}-1}\left(W_{j+u,
k}^2-\mathbb{E}(W_{j+u ,k}^{2})\right)\;,
\end{equation}
adequately normalized as $j,n_j\to\infty$.

There are two perspectives. One can consider, as
in~\cite{clausel-roueff-taqqu-tudor-2011a}, that the wavelet coefficients
$W_{j+u,k}$ are processes indexed by $u$ taking a finite number of values. A
second perspective consists in replacing instead the filter $h_{j}$
in~(\ref{e:WC}) by a multidimensional filter $h_{\ell,j}, \ell=1,\cdots,m$ and
thus replacing $W_{j,k}$ in~(\ref{e:WC}) by
$$
W_{\ell,j,k}=\sum_{t\in\mathbb{Z}}h_{\ell,j}(\gamma_j k-t)Y_{t}\;.
$$
We adopted this second perspective in~\cite{clausel-roueff-taqqu-tudor-2011b}
and we also adopt it here since it allows us to compare our results to those
obtained in~\cite{roueff-taqqu-2009} in the Gaussian case.

We use bold faced symbols $\bW_{j,k}$ and $\bh_j$ to
emphasize the multivariate setting and let
\begin{eqnarray*}
\bh_j=\{h_{\ell,j},\,\ell=1,\cdots,m\},\qquad
\bW_{j,k}=\{W_{\ell,j,k},\,\ell=1,\cdots,m\}\;,
\end{eqnarray*}
with
\begin{equation}\label{e:Wjkbold}
\bW_{j,k}=\sum_{t\in\mathbb{Z}}\bh_j(\gamma_j
k-t)Y_t=\sum_{t\in\mathbb{Z}}\bh_j(\gamma_j
k-t)\Delta^{-K}G(X_t),\,j\geq 0,k\in\mathbb{Z}\;.
\end{equation}
We then will study the asymptotic behavior of the sequence
\begin{equation}\label{e:snjbold}
\overline{\mathbf{S}}_{n_j,j}=\frac{1}{n_j}\sum_{k=0}^{n_j-1}\left(\bW_{j,
k}^2-\mathbb{E}[\bW_{j, k}^{2}]\right)\;,
\end{equation}
adequately normalized as $j\to\infty$, where, by convention, in this paper,
\begin{equation}
  \label{eq:conv-square-vector}
\bW_{j, k}^2=
\{W_{\ell,j,k}^2,\,\ell=1,\cdots,m\}\;.
\end{equation}
The squared Euclidean norm of a vector $\bx=[x_1,\dots,x_m]^T$ will be denoted
by $|\bx|^2=x_1^2+\dots+x_m^2$.

It turns out that the asymptotic behavior of $\overline{\bS}_{n_j,j}$ depends
on how the subsequence of Hermite coefficients $c_q,\,q\geq 1$ which are
\emph{non-vanishing} is distributed. We denote this
subsequence by $\{c_{q_\ell}\}_{\ell\in\ellset}$ where $\ellset$ is a sequence
of consecutive integers starting at 0,
\begin{equation}
  \label{eq:L}
\ellset\subseteq\{0,1,2,\dots\}\;,
\end{equation}
with same cardinality as the set of non-vanishing coefficients,
and $(q_\ell)_{\ell\in\ellset}$ is a
(finite of infinite) increasing sequence of integers such that
\begin{eqnarray*}
q_0&=&\mbox{ index of the first non--zero coefficient }c_q,\\
q_{\ell}&=&\mbox{ index of the }(\ell+1)\mbox{th non--zero
coefficient}\,,\quad \ell\geq 1\;.
\end{eqnarray*}

\medskip
\noindent\textbf{Examples}

  \begin{enumerate}[label=\arabic*)]
  \item If
$$
G(X_t)=c_1 H_1(X_t) + \frac{c_3}{3!}H_3(X_t)\;,
$$
where $c_1\neq 0$, $c_2=0,\,c_3\neq 0,\,c_q=0$ for
$q\geq 4$, then $q_0=1,\,q_1=3$ and $\ellset=\{0,1\}$.
\item If
$$
G(X_t)=\frac{c_2}{2!} H_2(X_t) + \frac{c_3}{3!}H_3(X_t)\;
 + \frac{c_4}{4!}H_4(X_t)\;,
$$
where $c_1=0,\,c_2\neq0,\,c_3\neq 0,\,c_4\neq
0,\,c_q=0$ for $q\geq 5$, then
$q_0=2,\,q_1=3,\,q_2=4$ and $\ellset=\{0,1,2\}$.
\item If
$$
G(X_t)=\sum_{q=1}^\infty\frac{c_q}{q!} H_q(X_t)\;,
$$
where
$c_q\neq 0$ for $q\geq 1$ then $q_0=1,\,q_1=2,\dots,$ and $\ellset=\{0,1,2,\cdots\}$.
\item  If
$$
G(X_t)=\frac{c_{q_0}}{q_0!} H_{q_0}(X_t)\;,
$$
where $c_{q_0}\neq 0$ and $c_q=0$ for $q\neq q_0$, then
$\ellset=\{0\}$.
  \end{enumerate}

While $c_0$ is always equal to $0$
(see~(\ref{e:c0})), the assumption~(\ref{e:hermiterank}) ensures
that $c_{q_0}\neq 0$ and hence that $\ellset$ always contains the
index $0$, so that $\ellset$ is never empty. In particular, we may
write
\begin{equation}\label{eq:assLD}
(\Delta^KY)_t=G(X_t)
=\sum_{\ell\in\ellset}\frac{c_{q_\ell}}{q_\ell!}H_{q_\ell}(X_t),\quad
t\in\mathbb{Z}\;,
\end{equation}
where, if $\ellset$ is infinite, the sum converges in the $L^2$
sense.

We set
\begin{align}
  \label{eq:defJ}
  &I=\{\ell\in\ellset~:\ell+1\in\ellset, q_{\ell+1}-q_{\ell}=1\} \;,\\
  \label{eq:defI}
  &J=\{(\ell_1,\ell_2)\in\ellset^2~:~\ell_1<\ell_2,\,q_{\ell_1}\neq 1\mbox{ and }q_{\ell_2}-q_{\ell_1}\geq
  2\}\;,
\end{align}
that is, $q_{\ell}$ and $q_{\ell+1}$ take consecutive values when
$\ell\in I$ and $q_{\ell_1}$ and $q_{\ell_2}$ differ by two or
more when $(\ell_1,\ell_2)\in J$. The structure of these two sets
is particulary important. The set $I$ could be empty (there are no
consecutive values of $q_\ell$) or not empty. Then we set
\begin{equation}\label{e:l0}
\ell_0=  \begin{cases}
\min(I)\geq 0\;,&\text{when $I$ is not empty}\;,    \\
\infty\;,&\text{when $I$ is  empty}\;.
  \end{cases}
\end{equation}
When $\ell_0$ is finite (that is, $I$ is not empty), $q_{\ell_0}$ is the
smallest index $q$ such that two Hermite coefficients $c_q$, $c_{q+1}$ are
non--zero. It will be involved in the normalization. We define, in addition,
\begin{equation}\label{e:m0}
m_0=\min(\{\ell\in\ellset,\,q_{\ell}\geq 3\})\geq 0\;.
\end{equation}
Thus $q_{m_0}$ is the smallest index $q$ such that $c_q$ is
non--zero with $q\geq 3$.

\newpage

\noindent\textbf{Examples}
  \begin{enumerate}[label=\arabic*)]
  \item If
$$
G(X_t)=c_1 H_1(X_t) + \frac{c_2}{2!} H_2(X_t) + \frac{c_4}{4!}H_4(X_t)\;,
$$
where $c_1\neq 0,\,c_2\neq 0,\,c_3= 0,c_4\neq0, c_{q}=0$
for $q\geq 5$
then $\ellset=\{0,1,2\}$,
$I=\{1\}$, $\ell_0=1$, $m_0=4$ and $J=\{(2,4)\}$.
\item If
$$
G(X_t)=\frac{c_2}{2!} H_2(X_t) + \frac{c_3}{3!}H_3(X_t)\;
 + \frac{c_4}{4!}H_4(X_t)\;,
$$
where $c_1=0,\,c_2\neq0,\,c_3\neq 0,\,c_4\neq
0,\,c_q=0$ for $q\geq 5$,
then  $\ellset=\{0,1,2\}$, $I=\{2,3\}$, $\ell_0=2$, $m_0=3$ and $J=\{(2,4)\}$.
\item If
$$
G(X_t)=c_1 H_1(X_t)=c_1 X_t\;,
$$
where $c_1\neq 0$ and
$c_q=0$ for $q\geq 2$, then  $\ellset=\{0\}$ and both $I$ and $J$ are empty.
\end{enumerate}

We are interested in the asymptotic behavior of the normalized scalogram
$\overline{\bS}_{n_j,j}$ defined in~(\ref{e:snjbold}). This behavior depends on
the sets $J$ and $I$. These sets affect both the rate of convergence and the
limit distribution of the rescaled sequence. The limit (see
Section~\ref{s:main}) will be expressed in terms of the Hermite processes which
are defined as follows~:
\begin{definition}\label{d:HP}
The Hermite process of order $q$ and index
\begin{equation}\label{e:qd}
(1/2)(1-1/q)<d<1/2\;,
\end{equation}
is the continuous time process
\begin{equation}\label{e:harmros}
Z_{q,d}(t)= \int_{\mathbb{R}^{q}}^{\prime\prime} \frac{\rme^{\rmi
(u_1+\cdots+u_q)\,t}- 1} {\rmi(u_1+\cdots+u_q)}|u_1\cdots
u_q|^{-d}\; \rmd\widehat{W}(u_1)\cdots
\rmd\widehat{W}(u_q),\,t\in\mathbb{R}\;.
\end{equation}
It is Gaussian and called Fractional Brownian Motion when $q=1$
and $0<d<1/2$. It is non Gaussian and called Rosenblatt process
when $q=2$ and $1/4<d<1/2$. The marginal distribution of
$Z_{q,d}(t)$ at $t=1$ is called the Hermite distribution of index $q$. It is
called a Rosenblatt distribution when $q=2$.
\end{definition}
The multiple integral~(\ref{e:harmros}) is defined in
Appendix~\ref{s:appendixA}. The symbol
$\int_{\mathbb{R}^{q}}^{\prime\prime}$ indicates that one does not
integrate on the diagonal $u_i=u_j$, $j\neq i$. The integral is
well-defined when~(\ref{e:qd}) holds or equivalently when,
\[
1\leq q<1/(1-2d)\;,
\]
because then it has finite $L^2$ norm. This process is
self--similar with self-similarity parameter
\[
H=dq+1-q/2=\delta(q)+1/2\in (1/2,1),
\]
that is for all $a>0$, $\{Z_{q,d}(at)\}_{t\in\mathbb{R}}$ and
$\{a^H Z_{q,d}(t)\}_{t\in\mathbb{R}}$ have the same finite
dimensional distributions, see~\cite{Taq79}.

\section{Main results}\label{s:main}
We shall now state the main results and discuss them. They are
proved in the following sections. We start with the assumptions

\bigskip

\noindent{\bf Assumptions A}
$\{\textbf{W}_{j,k},\,j\geq1,k\in\mathbb{Z}\}$ are the
multidimensional wavelet coefficients defined by~(\ref{e:Wjkbold})
, where
\begin{enumerate}[label=(\roman*)]
\item $\{X_t\}_{t\in\mathbb{Z}}$ is a stationary Gaussian process
with mean $0$, variance $1$ and spectral density $f$
  satisfying~(\ref{e:sdf}).
\item\label{item:LMass} $G$ is a real-valued function whose Hermite
expansion~(\ref{e:chaos-exp})
  satisfies condition~(\ref{e:longmemorycondition}), namely $q_0<1/(1-2d)$, and whose coefficients in Hermite expansion satisfy the following condition : for any $\lambda>0$,
\begin{equation}\label{e:condcv}
c_q=O((q!)^{d} e^{-\lambda q})\quad\mbox{ as }q\to\infty\;.
\end{equation}
\item the wavelet filters $(\bh_j)_{j\geq1}$ and their asymptotic Fourier
  transform $\widehat{\bh}_\infty$ satisfy \ref{ass:w-a}--\ref{ass:w-c} with
  $M$ vanishing moments. See details in Appendix~\ref{s:appendixB}.
\end{enumerate}

We shall focus on the asymptotic behavior of the scalogram for two
basic classes of functions $G$.
\begin{itemize}
\item The first class involves functions $G$ with Hermite rank
greater or equal to $2$ and with two consecutive terms in the
Hermite expansion, both of which having long--range dependence.
The result is stated in Theorem~\ref{th:LD1}.
\item The second
class involves functions $G$ with Hermite rank equal to $1$ with
no two consecutive terms with long--range dependence. The results
are stated in Theorems~\ref{th:LD2} and~\ref{th:LD3}.
\end{itemize}
Other classes are left for future work.

\subsection{$G$ has a Hermite rank greater or equal to
$2$}\label{s:main1} Consider functions $G$ of the form
\[
G(x)=\frac{c_2}{2!} H_2(x)+\cdots+\frac{c_{q_{\ell_0}}}{q_{\ell_0}!}
H_{q_{\ell_0}}(x)+\frac{c_{q_{\ell_0}+1}}{(q_{\ell_0}+1)!}H_{q_{\ell_0}+1}(x)
+\cdots\;.
\]
where $c_1=0$. Some of the $c_q$, $q\geq2$ may be zero as well.
More precisely assume that
\begin{equation}
  \label{eq:q0geq2}
q_0\geq 2\;,
\end{equation}
that is, that the Hermite rank of $G$ is 2 or more. Also assume that (a) there
exists two \emph{consecutive} terms and that (b) both are long range dependent.
Assumption (a) implies that the set $I$ in~(\ref{eq:defJ}) is not empty. Since
the index $q_{\ell_0}$ (see~(\ref{e:l0})) of the first of these two consecutive
terms could be $q_0\geq2$, we have $q_{\ell_0}\geq2$. The index of the second
of these consecutive terms is $q_{\ell_0}+1\geq3$. Assumption (b) will be
satisfied if this second term is long-range dependent, that is
\begin{equation}
  \label{eq:qell0plus1LRD}
q_{\ell_0}+1<1/(1-2d)\;,
\end{equation}
by~(\ref{e:dqq>0}). We note that this situation
implies the following boundaries for the parameter $d$:
$$
1/3<d<1/2 \;,
$$
as indicated in~(\ref{e:qd}).

Set
\begin{equation}\label{e:defnu}
\nu=2q_{\ell_0}+1-2q_0\;.
\end{equation}

The following theorem provides the limit of~(\ref{e:snjbold}) for two different
cases, depending on whether the limit of $n_j^{-1}\gamma_j^{\nu}$ when
$j\to+\infty$ is null or infinite.  It involves $K\geq 0$  defined
in~(\ref{e:DefY}), $q_0$ in~(\ref{e:hermiterank}), $\delta(q)$ is defined
in~(\ref{e:ldparamq}), $\ell_0$ in~(\ref{e:l0}). The integer $M$ is the number
of vanishing moments of the wavelet filters and appears in
Appendix~\ref{s:appendixB}.

\begin{theorem}\label{th:LD1}
  Suppose that Assumptions A hold with  $M\geq K+\delta(q_0)$.
  Suppose moreover that the Hermite expansion of $G$ satisfies~(\ref{eq:q0geq2})
  and~(\ref{eq:qell0plus1LRD}).

  Then two limits in distribution of the centered multidimensional scalogram
  $\overline{\bS}_{n,j}$ in~(\ref{e:snjbold}), suitably normalized, are
  possible. They involve the Hermite processes in Definition~\ref{d:HP}
  evaluated at time $t=1$. The coefficients involve $\ell_0$ and the
  multidimensional deterministic vector $\K_q$, whose entries
  $[L_{q}(\widehat{h}_{\ell,\infty})]_{\ell=1,\cdots,m}$ are defined as
\begin{equation}\label{e:defKp}
L_{q}(\widehat{h}_{\ell,\infty})=
\int_{\mathbb{R}^q}\frac{|\widehat{h}_{\ell,\infty}(u_1+\cdots+u_{q})|^{2}}
{|u_1+\cdots+u_{q}|^{2K}}\; \prod_{i=1}^q |u_i|^{-2d}\; \rmd
u_1\cdots\rmd u_{q} \; ,
\end{equation}
which is finite for any $q<1/(1-2d)$. Then
\begin{enumerate}[label=(\alph*)]
\item\label{it:thm1a} If $n_j\ll\gamma_j^{\nu}$ then, as $j,n_j\to\infty$,
\begin{equation*}
%\label{cc1}
n_j^{1-2d}\gamma _{j}
^{-2(\delta(q_{0})+K)}\overline{\bS}_{n_j,j}
\overset{\tiny{(\mathcal{L})}}{\rightarrow}\frac{c_{q_0}^2}{(q_0-1)!}\left[f^*(0)^{q_0}\K_{q_0-1}\right]
Z_{2,d}(1)\;.
\end{equation*}
\item\label{it:thm1b}

If $\gamma_j^{\nu}\ll n_j$ then, as $j,n_j\to\infty$,
\begin{equation*}%\label{cc2}
n_j^{(1-2d)/2}\gamma_j^{-(\delta(q_{\ell_0})+\delta(q_{\ell_0}+1)+2K)}
\overline{\bS}_{n_j,j}
\overset{\tiny{(\mathcal{L})}}{\rightarrow}2\frac{c_{q_{\ell_0}}c_{q_{\ell_0}+1}}{q_{\ell_0}!}\left[f^*(0)^{q_{\ell_0}+1/2}\; \K_{q_{\ell_0}}
\; \right]Z_{1,d}(1)\;.
\end{equation*}
\end{enumerate}
\end{theorem}
\newpage
\noindent{\bf Remarks.}
%\begin{remark}\label{rem:thLD1}
\begin{enumerate}[label=\arabic*.]
\item Using~(\ref{EqMajoHinf}) with $M\geq K$ and $\alpha>1/2$, the integral
  in~(\ref{e:defKp}) is finite for any positive integer $q<1/(1-2d)$, see
  Lemma~5.1 in~\cite{clausel-roueff-taqqu-tudor-2011a}. Thus, under
  Conditions~(\ref{eq:q0geq2}) and~(\ref{eq:qell0plus1LRD}), the vectors
  $\K_{q_{0}-1}$ and $\K_{q_{\ell_0}}$ appearing in the limits of
  Cases~\ref{it:thm1a} and~\ref{it:thm1b} have finite entries.
\item In case~\ref{it:thm1a}, the limit is a deterministic vector times the
non-Gaussian Rosenblatt random variable $Z_{2,d}(1)$, that is, the Rosenblatt
process  $Z_{2,d}(t)$ defined in~(\ref{e:harmros}) and evaluated at time $t=1$.
In case~\ref{it:thm1b}, the limit is a deterministic vector times the Gaussian
random variable  $Z_{1,d}(1)$, that is, Fractional Brownian motion
  $Z_{1,d}(t)$ defined in~(\ref{e:harmros}) and evaluated at time $t=1$.
\item  In the case where $n_j\sim C_0\gamma_j^{\nu}$ as $j\to\infty$ for
some $C_0>0$, the scalogram is asymptotically a linear combination
of a Rosenblatt and a Gaussian variable. Indeed, using the results
of Section~\ref{sec:LTSigma}, one can see that the scalogram is
the sum of two terms having the same order, both converging in the
$L^2$ sense respectively to a Rosenblatt and a Gaussian variable.
\end{enumerate}
\begin{proof}
This theorem is proved in Section~\ref{s:proofmain1}.
\end{proof}

In the framework of wavelet analysis as
in~\cite{moulines:roueff:taqqu:2007:jtsa}, we have $\gamma_j=2^j$ and the
number $n=n_j$
of wavelet coefficients available at scale index $j$, is related both to the number
$N$ of observations $Y_1,\cdots,Y_N$ of the time series $Y$ and to the length
$T$ of the support of the analyzing wavelet. More precisely, one has
(see~\cite{moulines:roueff:taqqu:2007:jtsa} for more details),
\begin{equation}\label{e:nj}
n_j=[2^{-j}(N-T+1)-T+1]= 2^{-j}N+O(1)\;,
\end{equation}
where $[x]$ denotes the integer part of $x$ for any real $x$.
Note that the assumption $n_j\to\infty$ when $j\to\infty$
is equivalent to $N\to\infty$ faster than $2^j$. Moreover, for any $\nu>0$,
\begin{equation}\label{e:cjnMRA}
n_j\ll 2^{j\nu}\Longleftrightarrow\, 2^{-j}N\ll
2^{j\nu}\Longleftrightarrow\,N\ll 2^{j(\nu+1)}\mbox{ when
}N\to\infty\;.
\end{equation}
\noindent{\bf Examples.} We now illustrate Theorem~\ref{th:LD1}
through three examples :
\begin{enumerate}[label=(\roman*)]
\item\label{item:exple1-thm31}
 $G=H_{q_0}$ with $q_0\geq 2$.
\item \label{item:exple2-thm31}
$G=H_{q_0}+H_{q_0+1}$ with $q_0\geq 2$, $q_0+1<1/(1-2d)$.
\item \label{item:exple3-thm31}
$G=H_{q_0}+H_{q_0+1}+H_{q_1}$ with $q_0\geq 2$,
$q_0+1<1/(1- 2d)$
 and with $q_1-(q_0+1)\geq 2$, that is, $J=\{(q_0+1,q_1)\}$.
\end{enumerate}
In all cases, the integer $q_0$ denotes the
 Hermite rank of $G$.

Let us elaborate on the conditions on $d$ and the resulting limits for these
examples. For simplicity, we assume that the scalogram $S_{n_j,j}$ is
univariate.

\medskip

\noindent\textbf{Example~\ref{item:exple1-thm31}}.
When $G=H_{q_0}$ with $q_0\geq 2$, $I$ and $J$ are both empty. Since $I$ is
empty one can regard $\ell_0$ and consequently $q_{\ell_0}$ and $\nu$ as
infinity, which suggests that we are in case~\ref{it:thm1a}, independently of
the growths of $n_j$ versus $\gamma_j$ as $j\to\infty$.  The asymptotic
behavior of the scalogram of this example is treated by Theorem~3.1
in~\cite{clausel-roueff-taqqu-tudor-2011b} under the condition $q_0<1/(1-2d)$.
Indeed, the obtained rate of convergence is the same as in case (a) of
Theorem~\ref{th:LD1} and the limit is also Rosenblatt.  This also corresponds
to the limit obtained by Bardet and Tudor in the case where $Y$ itself is the
Rosenblatt process (see Theorem 4 of~\cite{bardet:tudor:2010}).

\medskip

\noindent\textbf{Example~\ref{item:exple2-thm31}}.
Suppose $G=H_{q_0}+H_{q_0+1}$, with $q_0\geq 2$ and $q_0+1<1/(1-2d)$. Then $J$
is empty and $I=\{q_0\}$. The Hermite rank of $G$ is $q_0$ and thus coincides
with $q_{\ell_0}$. As a consequence, by~(\ref{e:defnu}), $\nu=1$.
Let us use Eq.~(\ref{e:cjnMRA}) to relate the asymptotic
behavior to the number of observation $N$ and the analyzing scale index
$j$. Since $\nu=1$, we get that the asymptotic behavior of the scalogram
$S_{n_j,j}$ depends on whether, as $j,N\to\infty$,
\[
N\ll 2^{2j}\mbox{ or if }2^{2j}\ll N\;.
\]
Let us explain how these two regimes show up in the limit.  The wavelet
coefficients of $Y$ can be expanded as follows
\[
W_{j,k}=W_{j,k}^{(q_0)}+W_{j,k}^{(q_0+1)}\;,
\]
where $W_{j,k}^{(q_0)},W_{j,k}^{(q_0+1)}$ belong respectively to
the chaos of order $q_0$ and $q_0+1$. Then,
\[
W_{j,k}^2=\left([W_{j,k}^{(q_0)}]^2+[W_{j,k}^{(q_0+1)}]^2\right)+\left(2W_{j,k}^{(q_0)}W_{j,k}^{(q_0+1)}\right)\;.
\]
The term $[W_{j,k}^{(q_0)}]^2$ behaves as in the case $G=H_{q_0}$ and is
asymptotically Rosenblatt as proved
in~\cite{clausel-roueff-taqqu-tudor-2011b}. The term $[W_{j,k}^{(q_0+1)}]^2$ is
asymptotically negligible as proved in Proposition~\ref{pro:SD}. The term
$W_{j,k}^{(q_0)}W_{j,k}^{(q_0+1)}$, on the other hand, turns out to be
asymptotically Gaussian. The asymptotic behavior of the scalogram then depends
on whether the Rosenblatt term or the Gaussian term is leading. This depends on
the limit of $N/2^{2j}$.  Hence, both limits stated in Theorem~\ref{th:LD1}
may occur:
\begin{enumerate}[label=$\bullet$]
\item If $2^{-2j}N\to0$, the term corresponding
to $[W_{j,k}^{(q_0)}]^2$ is leading and the scalogram $S_{n_j,j}$
of $Y$ is asymptotically Rosenblatt.
\item If $2^{-2j}N\to\infty$, the terms
corresponding to $W_{j,k}^{(q_0)}W_{j,k}^{(q_0+1)}$ are leading
and the
scalogram $S_{n_j,j}$ of $Y$ is asymptotically Gaussian.
\end{enumerate}

\medskip

\noindent\textbf{Example~\ref{item:exple3-thm31}}.
Suppose $G=H_{q_0}+H_{q_0+1}+H_{q_1}$ with $q_0\geq 2$, $q_0+1<1/(1-2d)$ and
$q_1-(q_0+1)\geq 2$. Then $I=\{q_0\}$, $J=\{(q_0,q_1),(q_0+1,q_1)\}$. Observe
that in this case, $J$ is not involved in the limit of $S_{n_j,j}$ and the
behavior of the scalogram is similar to that of
Example~\ref{item:exple2-thm31}. Thus, the two limits of
Theorem~\ref{th:LD1} may occur.

\subsection{The Hermite rank of $G$ equals $1$}\label{s:main23}

 Here we assume that
 \begin{equation}
   \label{eq:hyp-hermite-rank1}
q_0=1,\; q_1<1/(1-2d)  \quad\text{and}\quad \ell_0=\infty\;.
 \end{equation}
 In particular, $\ell_0=\infty$ implies $q_1\geq q_0+2=3$ and thus
 this condition implies $d>1/3$, thus $d\in(1/3,1/2)$.
 By definition of $\ell_0$ in~(\ref{e:l0}), the last condition
 in~(\ref{eq:hyp-hermite-rank1}) means that there are no terms with consecutive
 indices in the Hermite expansion. Thus
\[
G=c_1
H_1+\frac{c_{q_1}}{q_1!}H_{q_1}+\frac{c_{q_2}}{q_2!}H_{q_2}+\cdots
\]
where for any $\ell\in\ellset$, $q_{\ell+1}-q_{\ell}\geq 2$. In this case the
following critical index plays an important role~:
\begin{equation}\label{e:q1}
q_1^*=2+\frac{1}{2(1-2d)}\;.
\end{equation}
It will also be useful to relate the number of available wavelet coefficients
$n=n_j$ to $\gamma_j^{\nu}$ where $\nu$ takes the following three values~:
\begin{equation}\label{e:defbeta}
\nu_1=\frac{(1-2d)(q_1-1)}{1-(1-2d)(q_1-1)},\,\nu_2=\frac{1-2d}{2d-1/2}(q_1-1),\,
\nu_3=
\begin{cases}
\frac{q_1-1}{q_1-3}&\text{if $q_1>3$}\\
\infty&\text{if $q_1=3$}\;.
\end{cases}
\end{equation}
As shown in the following lemma, the relations between
$\nu_1,\nu_2$ and $\nu_3$ depend on whether $q_1< q_1^*$ or
$q_1\geq q_1^*$~:
\begin{lemma}\label{lem:compexp}
  \
\begin{enumerate}[label=$\bullet$]
\item If $q_1< q_1^*$ then $\nu_1< \nu_2<\nu_3$.
\item If
$q_1\geq q_1^*$ then $\nu_3\leq \nu_2\leq\nu_1$ (in particular we have $\nu_3<\infty$).
\end{enumerate}
\end{lemma}
\begin{proof}
First observe that
\begin{eqnarray*}
\nu_1=\frac{(1-2d)(q_1-1)}{1-(1-2d)(q_1-1)}< \nu_2=\frac{(1-2d)(q_1-1)}{2d-1/2}
&\Longleftrightarrow& 2d-\frac{1}{2}< 1-(1-2d)(q_1-1)\\
&\Longleftrightarrow& q_1< 1+\frac{1-(2d-\frac{1}{2})}{1-2d}=q_1^*\;.
\end{eqnarray*}
Now, if $q_1>3$ then
\begin{eqnarray*}
\nu_2=\frac{(1-2d)(q_1-1)}{2d-1/2}< \nu_3=\frac{q_1-1}{q_1-3}
\Longleftrightarrow q_1-3< \frac{2d-\frac{1}{2}}{1-2d}
\Longleftrightarrow q_1< 3+\frac{2d-\frac{1}{2}}{1-2d}=q_1^*\;,
\end{eqnarray*}
and, since $q_1^*>3$, the case $q_1=3$ can only happen for $q_1< q_1^*$, and yields
$\nu_2<\infty=\nu_3$.
\end{proof}
The next theorems indicate the limits in the various cases. We first consider the case where $q_1$ is lower than the critical index $q_1^*$.
\begin{theorem}\label{th:LD2}
Suppose that Assumptions A hold with $M\geq K+d$.
  Suppose moreover that the Hermite expansion of $G$
  satisfies~(\ref{eq:hyp-hermite-rank1}) and assume that $q_1< q_{1}^*$, where
  $q_{1}^*$ is defined in~(\ref{e:q1}).

Then three limits of  the
multidimensional scalogram $\overline{\bS}_{n,j}$
in~(\ref{e:snjbold}), suitably normalized, are possible~:
\begin{enumerate}[label=(\alph*)]
\item \label{item:LD23}
If $n_j\ll\gamma_j^{\nu_1}$ then as $j, n_j\to\infty$,
\begin{equation*}%\label{cc2}
n_j^{1/2}\gamma_j^{-(2d+2K)}
\overline{\bS}_{n_j,j}\overset{\tiny{(\mathcal{L})}}{\rightarrow}
\;c_1^2\;\mathcal{N}(0,\Gamma)\;,
\end{equation*}
where $\Gamma$ is defined as
\begin{equation}\label{e:Gamma}
\Gamma_{i,i'} = 4\pi(f^*(0))^2\int_{-\pi}^{\pi}\left|\sum_{p\in\mathbb{Z}}|\lambda+2p\pi|^{-2(K+d)}[\widehat{h}_{i,\infty}\overline{\widehat{h}}_{i',\infty}](\lambda+2p\pi)\right|^2\rmd \lambda,\quad 1\leq i,i'\leq m\;.
\end{equation}
\item  \label{item:LD22}
If $\gamma_j^{\nu_1}\ll n_j$ and either $\nu_3=\infty$ or $n_j\ll\gamma_j^{\nu_3}$ then as
$j,n_j\to\infty$
\begin{equation*}%\label{cc2}
  n_j^{(1-2\delta(q_1-1))/2}\gamma_j^{-(2\delta(\frac{q_1+1}{2})+2K)}\overline{\bS}_{n_j,j}
  \overset{\tiny{(\mathcal{L})}}{\rightarrow}\frac{2 c_1 c_{q_1}}{(q_1-1)!}[f^{*}(0)]^{(q_1+1)/2}\K_{1}Z_{q_1-1,d}(1)\;.
\end{equation*}
\item \label{item:LD21}
If $\nu_3<\infty$ and $\gamma_j^{\nu_3}\ll n_j$ then as
$j,n_j\to\infty$,
\begin{equation*}
%\label{cc1}
n_j^{1-2d} \gamma _{j}
^{-2(\delta(q_{1})+K)}\overline{\bS}_{n_j,j}
\overset{\tiny{(\mathcal{L})}}{\rightarrow}\frac{c_{q_1}^2}{(q_1-1)!}\left[f^*(0)\right]^{q_1}\;
\K_{q_1-1}\;
Z_{2,d}(1)\;.
\end{equation*}
\end{enumerate}
\end{theorem}
\begin{remark}\label{rem:thLD2}
In case~\ref{item:LD21}, the limit is a deterministic vector times the
non-Gaussian Rosenblatt random variable $Z_{2,d}(1)$. In case~\ref{item:LD22},
the limit is a deterministic vector times a Hermite random variable of order
$q_1-1>3-1=2$, which can be represented by a multiple Wiener integral of order
3 or more (see Definition~\ref{d:HP}).

In the case where $n_j\sim C_0\gamma_j^{\nu_3}$ as $j\to\infty$
for some $C_0>0$, the scalogram is asymptotically a linear
combination of a Rosenblatt and a Hermite random variable. This is
because, up to an equality in distribution, it is the sum of two terms both converging in $L^2$ after
normalization (see Section~\ref{sec:LTSigma}). On the other hand
if $n_j\sim C_0\gamma_j^{\nu_1}$ as $j\to\infty$ for some $C_0>0$,
the situation is complicated. This is because the scalogram is the
sum of two terms of same order, one converging in $L^2$ to a
Hermite random variable, the other converging only in law to a
Gaussian random variable.
\end{remark}
\begin{proof} This theorem is proved in Section~\ref{s:proofmain23}.
\end{proof}
We now consider the case where $q_1$ is greater than the critical
exponent $q_1^*$.
\begin{theorem}\label{th:LD3}
  Suppose that Assumptions A hold with $M\geq K+d$.  Suppose moreover that the
  Hermite expansion of $G$ satisfies~(\ref{eq:hyp-hermite-rank1}) and assume
  that $q_1\geq q_{1}^*$, where $q_{1}^*$ is defined
  in~(\ref{e:q1}).

  Then two limits of the multidimensional scalogram $\overline{\bS}_{n,j}$
  in~(\ref{e:snjbold}), suitably normalized, are possible~:
\begin{enumerate}[label=(\alph*)]
\item \label{item:LD31}
If $n_j\ll\gamma_j^{\nu_2}$ then as $j,n_j\to\infty$,
\begin{equation*}%\label{cc2}
n_j^{1/2}\gamma_j^{-(2d+2K)} \overline{\bS}_{n_j,j}
\overset{\tiny{(\mathcal{L})}}{\rightarrow}\;c_1^2\mathcal{N}(0,\Gamma)\;,
\end{equation*}
where $\Gamma$ is as in Theorem~\ref{th:LD2}~(a).
\item  \label{item:LD32}
If $\gamma_j^{\nu_2}\ll n_j$ then as $j,n_j\to\infty$,
\begin{equation*}
%\label{cc1}
 n_j^{1-2d} \gamma _{j} ^{-2(\delta(q_{1})+K)}  \overline{\bS}_{n,j}
\overset{\tiny{(\mathcal{L})}}{\rightarrow}\frac{c_{q_1}^2}{(q_1-1)!}\left[f^*(0)\right]^{q_1}
\K_{q_1-1}\;
Z_{2,d}(1)\;.
\end{equation*}
\end{enumerate}
\end{theorem}
\begin{remark}\label{rem:thLD3}
As in the case of Theorem~\ref{th:LD2}, the case where $n_j\sim
C_0\gamma_j^{\nu_2}$ as $j\to\infty$, seems quite complicated to
deal with.
\end{remark}
\begin{proof}
This theorem is proved in Section~\ref{s:proofmain23}.
\end{proof}
\noindent{\bf Example.} We now illustrate Theorem~\ref{th:LD2}
and~\ref{th:LD3}. Our setting is still that
of~\cite{moulines:roueff:taqqu:2007:jtsa} as above.

The memory parameter $d$ is assumed to belong to $(3/8,1/2)$.
Consider the case where $$G=H_1+H_{q_1}\;,$$ with $3<q_1<1/(1-2d)$. We
will prove in the sequel that the wavelet coefficients of $Y$ can
be expanded as
\[
W_{j,k}=W_{j,k}^{(1)}+W_{j,k}^{(q_1)}\;,
\]
where $W_{j,k}^{(1)}$ is Gaussian and $W_{j,k}^{(q_1)}$ belongs to
the chaos of order $q_1$. Then,
\[
W_{j,k}^2=[W_{j,k}^{(1)}]^2+[W_{j,k}^{(q_1)}]^2+2W_{j,k}^{(1)}W_{j,k}^{(q_1)}\;.
\]
The empirical mean of the terms $[W_{j,k}^{(1)}]^2$ behaves as in the
Gaussian case and is asymptotically Gaussian. The empirical mean of the
terms $[W_{j,k}^{(q_1)}]^2$ behaves as in the case $G=H_{q_1}$ with
$q_1\geq 2$ and is asymptotically Rosenblatt. Finally the
empirical mean of the terms $2W_{j,k}^{(1)}W_{j,k}^{(q_1)}$ belongs to
the chaos of order $q_1-1>2$. The asymptotic behavior of the
scalogram then depends on which of the three terms is leading.

To see what happens, let $N$ be as before the number of
observations and assume that $\gamma_j=2^j$. Let $n_j\sim N2^{-j}$
as $j\to\infty$ as in~(\ref{e:nj}). Distinguish two cases~:
$q_1<q_1^*$ and $q_1\geq q_1^*$ where $q_1^*$ is defined
in~(\ref{e:q1}).

If $q_1<q_1^*$, the three possibilities stated in
Theorem~\ref{th:LD2} can occur~:
\begin{enumerate}[label=$\bullet$]
\item if $2^{-j(\nu_1+1)}N\to 0$ as $N,j\to\infty$, then the term
corresponding to $[W_{j,k}^{(1)}]^2$ is leading and the scalogram
$S_{n_j,j}$ of the process $\{Y_t\}_{t\in\mathbb{Z}}$ is
asymptotically Gaussian (case~\ref{item:LD23}).
\item if $2^{-j(\nu_1+1)}N\to \infty$ and
$2^{-j(\nu_3+1)}N\to 0$ as $N,j\to\infty$, then the term
corresponding to $2W_{j,k}^{(1)}W_{j,k}^{(q_1)}$ is leading and
the scalogram $S_{n_j,j}$ of $\{Y_t\}$ belongs asymptotically to
the chaos of order $q_1-1>2$ (case~\ref{item:LD22}).
\item if $2^{-j(\nu_3+1)}N\to
\infty$ as $N,j\to\infty$, then the term corresponding to
$[W_{j,k}^{(q_1)}]^2$ with $q_1>3$ is leading and the scalogram
$S_{n_j,j}$ of $\{Y_t\}$ is asymptotically Rosenblatt (case~\ref{item:LD21}).
\end{enumerate}

If we now assume that $q_1\geq q_1^*$, we are in the setting of
Theorem~\ref{th:LD3} and the term corresponding to
$2W_{j,k}^{(1)}W_{j,k}^{(q_1)}$ is always negligible. Then only
two different situations can occur~:
\begin{enumerate}[label=$\bullet$]
\item if $2^{-j(\nu_2+1)}N\to 0$ as $N,j\to\infty$, then the term
corresponding to $[W_{j,k}^{(1)}]^2$ is leading and the scalogram
$S_{n_j,j}$ of $\{Y_t\}$ is asymptotically Gaussian (case~\ref{item:LD31}).
\item if $2^{-j(\nu_2+1)}N\to \infty$ as $N,j\to\infty$, then the term
  corresponding to $[W_{j,k}^{(q_1)}]^2$ is leading and the scalogram
  $S_{n_j,j}$ of $\{Y_t\}$ is asymptotically Rosenblatt (case~\ref{item:LD32}).
\end{enumerate}

\section{The basic decomposition}\label{s:expansion}
Our goal is to investigate the asymptotic behavior of $\overline{\bS}_{n_j,j}$
as defined in~(\ref{e:snjbold}) when $j\rightarrow +\infty$. As
in~\cite{clausel-roueff-taqqu-tudor-2011b}, our main tool will be the
Wiener-It\^o chaos expansion of $\overline{\bS}_{n_j,j}$ which involves
multiple stochastic integrals $\widehat{I}_q$, $q=1,2,\dots$.  These are
defined in Appendix~\ref{s:appendixA}. In this case, the situation is more
complex than in the case $G=H_{q_0}$ since as proved
in~\cite{clausel-roueff-taqqu-tudor-2011a}, the wavelet coefficients
$\bW_{j,k}$, defined in~(\ref{e:Wjkbold}), admit an expansion into Wiener chaos
as follows~:
\begin{equation}\label{e:Wjk}
\bW_{j,k}=
\sum_{q=1}^{\infty}\frac{c_q}{q!}\bW_{j,k}^{(q)}\;,
\end{equation}
where $\bW_{j,k}^{(q)}$ is a multiple integral of order $q$. Then, using the
same convention as in~(\ref{eq:conv-square-vector}), we have
\begin{equation}\label{e:sumOFprodwjk}
\bW_{j,k}^2=
\sum_{q=1}^{\infty}\left(\frac{c_q}{q!}\right)^2\;\left(\bW_{j,k}^{(q)}\right)^2
+2 \sum_{q'=2}^\infty\sum_{q=1}^{q'-1} \frac{c_q}{q!}
\frac{c_{q'}}{q'!}\bW_{j,k}^{(q)}\bW_{j,k}^{(q')}\;,
\end{equation}
where the convergence of the infinite sums hold in $L^1(\Omega)$ sense.

Each $\bW_{j,k}^{(q)}$ is a multiple integral of order $q$
of some multidimensional kernel $\mathbf{f}_{j,k} ^{(q)}$, that is
\begin{equation}\label{e:a2}
\bW_{j,k}^{(q)} =\widehat{I}_{q}
(\mathbf{f}_{j,k}^{(q)})\;.
\end{equation}
Now, using the product formula for multiple stochastic
integrals~(\ref{EqProdIntSto}), one gets, as shown in
Proposition~\ref{pro:decompSnjLD} that, for any
$(n,j)\in\mathbb{N}^2$,
\begin{align}
\nonumber \overline{\mathbf{S}}_{n,j}& =
\frac{1}{n}\sum_{k=0}^{n-1}\bW_{j,k}^2-\mathbb{E}[\bW_{j,0}^2]\\
\nonumber &=
\sum_{q=1}^{\infty}\left(\frac{c_q}{q!}\right)^2\;\;\sum_{p=0}^{q-1}p!
{{q}\choose{p}}^2\;\mathbf{S}_{n,j}^{(q,q,p)}\\
\label{e:decompSnjLD}&
\hspace{0.5cm}+2 \sum_{q'=2}^\infty\sum_{q=1}^{q'-1}\frac{c_q}{q!}
\frac{c_{q'}}{q'!} \sum_{p=0}^{q} \; p!\;
{{q}\choose{p}}{{q'}\choose{p}} \; \mathbf{S}_{n,j}^{(q,q',p)}\;,
\end{align}
where, for all $q,q'\geq 1$ and $0\leq p\leq \min(q,q')$, $\bS_{n,j}^{(q,q',p)}$ is of the form
\begin{equation}\label{e:Snjrrp1}
\bS_{n,j}^{(q,q',p)}=
\widehat{I}_{q+q'-2p}(\mathbf{g}_{n,j}^{(q,q',p)})\;.
\end{equation}
We call $q+q'-2p$ the {\it order} of the summand
$\mathbf{S}_{n,j}^{(q,q',p)}$. For any $n,j,q,q',p$, the function
$\mathbf{g}_{n,j}^{(q,q',p)}(\xi)$,
$\xi=(\xi_1,\dots,\xi_{q+q'-2p})\in\mathbb{R}^{q+q'-2p}$ is
defined for every $p,q,q'$ as
\begin{equation}\label{e:g-def}
\mathbf{g}_{n,j}^{(q,q',p)}(\xi)= \frac{1}{n}\sum_{k=0}^{n-1}
\left( \mathbf{f}_{j,k}^{(q)}\overline{\otimes}_{p}
\mathbf{f}_{j,k}^{(q')}\right)\;,
\end{equation}
where the operation $\overline{\otimes}_{p}$ is defined
in~(\ref{e:times-p}) for each entry. The expansion in Wiener chaos
of $\overline{\mathbf{S}}_{n,j}$ implies that
\begin{equation}\label{e:s3}
\overline{\mathbf{S}}_{n,j}=c_1^2
\mathbf{S}_{n,j}^{(1,1,0)}+\mathbf{\Sigma}_{n,j}^{(0)}+\mathbf{\Sigma}_{n,j}^{(1)}+
\mathbf{\Sigma}_{n,j}^{(2)}+\mathbf{\Sigma}_{n,j}^{(3)}\;,
\end{equation}
with
\begin{equation}\label{e:defSigma0}
\mathbf{\Sigma}_{n,j}^{(0)}=\sum_{\ell\in\mathcal{L},\,q_{\ell}\neq
1}\frac{c_{q_{\ell}}^2}{(q_{\ell}!)^2}\sum_{p=0}^{q_{\ell}-1} p!
{{q_{\ell}}\choose{p}}^2\,\mathbf{S}_{n,j}^{(q_{\ell},q_{\ell},p)},
\end{equation}
\begin{equation}\label{e:defSigma1}
\mathbf{\Sigma}_{n,j}^{(1)}=2\sum_{(\ell_1,\ell_2)\in
J}\frac{c_{q_{\ell_1}}}{q_{\ell_1}!}\frac{c_{q_{\ell_2}}}{q_{\ell_2}!}
\sum_{p=0}^{q_{\ell_1}}
p! {{q_{\ell_1}}\choose{p}} {{q_{\ell_2}}\choose{p}}
\,\mathbf{S}_{n,j}^{(q_{\ell_1},q_{\ell_2},p)},
\end{equation}
\begin{align}\nonumber
\mathbf{\Sigma}_{n,j}^{(2)}&
=2\sum_{\ell\in\ellset,\ell\geq m_0}\frac{c_1c_{q_{\ell}}}{q_{\ell}!}
\sum_{p=0}^{{1}}
p! {{{1}}\choose{p}} {{q_{\ell}}\choose{p}}
\,\mathbf{S}_{n,j}^{({1},q_{\ell},p)},
\\
\label{e:defSigma2}
&=2\sum_{\ell\in\ellset,\ell\geq
m_0}\left(\frac{c_{1}c_{q_{\ell}}}{q_{\ell}!}
\mathbf{S}_{n,j}^{(1,q_{\ell},0)}+\frac{ c_{1}c_{q_{\ell}}}{(q_{\ell}-1)!}
\mathbf{S}_{n,j}^{(1,q_{\ell},1)}\right)\;,
\end{align}
\begin{equation}\label{e:defSigma3}
\mathbf{\Sigma}_{n,j}^{(3)}=2\sum_{\ell\in
I}\frac{c_{q_{\ell}}}{q_{\ell}!}\frac{c_{q_{\ell}+1}}{(q_{\ell}+1)!}\sum_{p=0}^{q_{\ell}}
p!{{q_{\ell}}\choose{p}}
{{q_{\ell}+1}\choose{p}} \,
\mathbf{S}_{n,j}^{(q_{\ell},q_{\ell}+1,p)}\;.
\end{equation}

The sets $\ellset,\,I$ and $J$ are defined in~(\ref{eq:L}),~(\ref{eq:defJ})
and~(\ref{eq:defI}) respectively and the index $m_0$, defined in~(\ref{e:m0}),
is such that $q_{m_0}\geq3$.

Let us comment on the decomposition (\ref{e:s3}). The sum
$\mathbf{\Sigma}_{n,j}^{(0)}$ contains terms of the form
$\mathbf{S}_{n,j}^{(q,q,p)}$ that is multiple integrals of order
$2(q-p)$. Then this sum, after subtracting its expectation, has
only summands of order
$2,4,6,\dots$ in the Wiener chaos.

The sum $\mathbf{\Sigma}_{n,j}^{(1)}$ contains multiple integrals of orders
$q+q'-2p$ with $q\neq 1$, $q'\neq 1$, $p\leq q\wedge q'$ and $|q-q'|\geq
2$. That means that all the summands in $\mathbf{\Sigma}_{n,j}^{(1)}$ are of
order greater than or equal to $2$.

The sum $\mathbf{\Sigma}_{n,j}^{(2)}$ contains multiple integrals of orders
$q+q'-2p$ with $q=1 $, $q'\geq q_{m_0}\geq 3$ and $p=0$ or $1$. All the
summands in $\mathbf{\Sigma}_{n,j}^{(2)}$ are then of order greater than or
equal to $q_{m_0}-1\geq 2$.

The last sum $\mathbf{\Sigma}_{n,j}^{(3)} $ contains terms of the form
$\mathbf{S}_{n,j}^{(q,q+1,p)}$, that is multiple integrals of order
$q+(q+1)-2p=2q+1-2p$. When $p=q$, $q+1+q-2q=1$, thus one can have components in
the first Wiener chaos, that is Gaussian terms.

We will see that
$\mathbf{\Sigma}_{n,j}^{(0)}+\mathbf{\Sigma}_{n,j}^{(1)}$ will
converge to a non-Gaussian limit, more precisely to a random
variable in the second Wiener chaos. The sum
$\mathbf{\Sigma}_{n,j}^{(2)}$ will also converge to a non-Gaussian
limit, more precisely to a random variable in the Wiener chaos of
order $q_{m_0}-1$. Finally
$\mathbf{\Sigma}_{n,j}^{(3)}$ will tend to a Gaussian limit.

\begin{remark}
  It is the presence of
$\mathbf{\Sigma}_{n,j}^{(2)}$ which creates the possibility of
having as limit a multiple integral of order greater than $2$.
Thus, starting with a process
\[
G(X_t)=H_1(X_t)+H_{q_1}(X_t)\;,
\]
with $q_1\geq 4$, then $q_{m_0}=q_1$ and one may obtain as limit of the
scalogram a Hermite process of order $q_1-1\geq 3$.
\end{remark}

Let us formalize the above decomposition of $\mathbf{S}_{n,j}$ and
give a more explicit expression for the function
$\mathbf{g}_{n,j}^{(q,q',p)}$ in~(\ref{e:g-def}). The next
proposition is a generalization of Proposition~6.1
of~\cite{clausel-roueff-taqqu-tudor-2011b}.
\begin{proposition}\label{pro:decompSnjLD}
For all $j$, $\{\bW_{j,k}\}_{k\in\mathbb{Z}}$ is a weakly
stationary sequence. Moreover, for any $(n,j)\in\mathbb{N}^2$,
$\overline{\mathbf{S}}_{n,j}$ can be expressed as~(\ref{e:decompSnjLD})
where the infinite sums converge in the $L^1(\Omega)$ sense.
The function $\mathbf{g}_{n,j}^{(q,q',p)}(\xi)$,
  $\xi=(\xi_1,\dots,\xi_{q+q'-2p})\in\mathbb{R}^{q+q'-2p}$,
  in~(\ref{e:Snjrrp1}), equals
\begin{equation}\label{e:bg}
\begin{array}{lll}\mathbf{g}_{n,j}^{(q,q',p)}(\xi)&=& D_n(\gamma_j\{\xi_1+\dots+\xi_{q+q'-2p}\})\times \prod_{i=1}^{q+q'-2p}[\sqrt{f(\xi_i)}\1_{(-\pi,\pi)}(\xi_i)]\\
&&\times\;\mathbf{\widehat{\kappa}}_j^{(p)}(\xi_1+\dots+\xi_{q-p},\xi_{q-p+1}+\dots+\xi_{q+q'-2p})
\;.
\end{array}
\end{equation}
Here $f$ denotes the spectral density of the underlying Gaussian
process $X$ and
\begin{equation}\label{eq:dirichlet}
D_n(u)=\frac{1}{n}\sum_{k=0}^{n-1}\rme^{\rmi k u}=
\frac{1-\rme^{\rmi n u}}{n(1-\rme^{\rmi  u})}\;,
\end{equation}
denotes the normalized Dirichlet kernel. Finally, for
$\xi_1,\xi_2\in\mathbb{R}$, if $p\neq0$,
\begin{equation}\label{e:intrepKjp1}
{\widehat{\bkappa}}_{j}^{(p)}(\xi_1,\xi_2)=\int_{(-\pi,\pi)^{p}}\left(\prod_{i=1}^{p}f(\lambda_i)\right)\;
\mathbf{\widehat{h}}_j^{(K)}(\lambda_1+\dots+\lambda_p+\xi_1)\overline{\mathbf{\widehat{h}}_{j}^{(K)}
(\lambda_1+\dots+\lambda_p-\xi_2)}\;\rmd^p\lambda\;,
\end{equation}
and, if $p=0$,
\begin{equation}\label{e:intrepKjp1bis}
{\widehat{\bkappa}}_{j}^{(p)}(\xi_1,\xi_2)=
\mathbf{\widehat{h}}_j^{(K)}(\xi_1)\overline{\mathbf{\widehat{h}}_j^{(K)}(\xi_2)}\;.
\end{equation}
\end{proposition}

\noindent{\bf Notation.} To simplify the notation, for any integer
$p$ and $q_1,\dots,q_p\in\mathbb{Z}_+$ we shall denote by
$\Sigma_{q_1,\dots,q_p}$, the $\mathbb{C}^{q_1+\dots+q_p}\to\mathbb{C}^p$
function defined, for all
$y=(y_1,\dots,y_{q_1+\dots+q_p})\in\mathbb{C}^{q_1+\dots+q_p}$ by
\begin{equation}\label{e:DefPartialSums}
\Sigma_{q_1,\dots,q_p}(y)=\left(\sum_{i=1}^{q_1}
y_i,\sum_{i=q_1+1}^{q_1+q_2} y_i,\dots,
\sum_{i=q_1+\dots+q_{p-1}+1}^{q_1+\dots+q_p} y_i\right)\;.
  \end{equation}
Note that, for $p=1$, one simply has $\Sigma_q(y)=y_1+\dots+y_q$.

With this notation, (\ref{e:Snjrrp1}) and~(\ref{e:intrepKjp1})
become respectively
\begin{align}\label{e:Snjrrp}
&\mathbf{S}_{n,j}^{(q,q',p)}=\widehat{I}_{q+q'-2p}\left(
D_n\circ\Sigma_{q+q'-2p}(\gamma_j\times\cdot)\times
[\sqrt{f}\1_{(-\pi,\pi)}]^{\otimes(q+q'-2p)}
\times {\widehat{\bkappa}}_j^{(p)}\circ\Sigma_{q-p,q'-p}\right)\;,\\
\label{e:intrepKjp}
&{\widehat{\bkappa}}_{j}^{(p)}(\xi_1,\xi_2)=\int_{(-\pi,\pi)^{p}}
f^{\otimes p}(\lambda)\;
\mathbf{\widehat{h}}_j^{(K)}(\Sigma_p(\lambda)+\xi_1)\overline{\mathbf{\widehat{h}}_{j}^{(K)}(\Sigma_p(\lambda)-\xi_2)}\;\rmd^p\lambda\;,\mbox{
if }p\neq 0,
\end{align}
where $\circ$ denotes the composition of functions,
$\lambda=(\lambda_1,\cdots,\lambda_p)$ and $f^{\otimes
  p}(\lambda)=f(\lambda_1)\cdots f(\lambda_p)$ is written as a tensor
product.

\begin{proof}[Proof of Proposition~\ref{pro:decompSnjLD}]
  For sake of simplicity we can assume that $\bW$ is a vector of length $m=1$
  since the case $m\geq 2$ can be deduced by applying the case $m=1$ to each
  entries. We must give an expansion in Wiener chaos of the one dimensional
  scalogram $S_{n,j}-\mathbb{E}(S_{n,j})$ in our setting.
  Using~(\ref{e:defsnj}),~(\ref{e:sumOFprodwjk}),~(\ref{e:a2}) and the product
  formula~(\ref{EqProdIntSto}) of Proposition~\ref{pro:intproduct}, we have as
  in Proposition~6.1 of~\cite{clausel-roueff-taqqu-tudor-2011b}
\begin{equation}\label{e:a4}
S_{n,j}=\sum_{q,q'=1}^\infty \frac{c_q
c_{q'}}{q!q'!}\sum_{p=0}^{q\wedge q'}
p!{{q}\choose{p}}{{q'}\choose{p}}\widehat{I}_{q+q'-2p}\left(g_{n,j}^{(q,q',p)}\right)\;,
\end{equation}
where
\[
g_{n,j}^{(q,q',p)}=\frac{1}{n}\sum_{k=0}^{n-1}f_{j,k}^{(q)}\overline{\otimes}_p
f_{j,k}^{(q')}\;.
\]
By~(\ref{e:fjk_representation1}),
\begin{equation}\label{e:fjk_representation2}
f_{j,k}^{(q)}(\xi)= \exp\circ\Sigma_q(\rmi
k\gamma_j\xi)\left(\widehat{h}_{j}^{(K)}\circ\Sigma_q(\xi)\right)\left(f^{\otimes
q}(\xi)\right)^{1/2} \1_{(-\pi,\pi)}^{\otimes q}(\xi)
\;,\xi\in\mathbb{R}^q\;.
\end{equation}
If $q+q'-2p\neq 0$, let $\xi=(\xi_1,\cdots,\xi_{q+q'-2p})$. As in
\cite{clausel-roueff-taqqu-tudor-2011b} using~(\ref{e:times-p}), we get that
$g_{n,j}^{(q,q',p)}$ is a function with $q+q'-2p $ variables given by
\begin{equation*}
g_{n,j}^{(q,q',p)}(\xi) = \frac{1}{n}
\sum_{k=0}^{n-1}\exp\circ\Sigma_{q+q'-2p}(\rmi k\gamma_j\xi)
\times [\sqrt{f}\1_{(-\pi,\pi)}]^{\otimes q+q'-2p}(\xi) \times
\widehat{\kappa}_j^{(p)}\circ\Sigma_{q-p,q'-p}(\xi)\;.
\end{equation*}
The Dirichlet kernel $D_{n}$ appears when one computes the sum
$\frac{1}{n} \sum_{k=0}^{n-1}\exp\circ\Sigma_{q+q'-2p}(\rmi
k\gamma_j\xi)$. This implies the  formula (\ref{e:bg}).

In addition, the chaos of order zero appears in the expression~(\ref{e:a4}) of
$S_{n,j}$ in the terms with $p=q=q'$ since
$\widehat{I}_{q+q'-2p}=\widehat{I}_0$. In this case, a similar argument as in
\cite{clausel-roueff-taqqu-tudor-2011b} leads to
\[
\frac{1}{n}\sum_{q=1}^{\infty}\frac{c_q^2}{(q!)^2}\sum_{k=1}^{n}\mathbb{E}(|W_{j,k}^{(q)}|^2)
=\sum_{q=1}^{\infty}\frac{c_q^2}{(q!)^2}\times
\mathbb{E}(|W_{j,0}^{(q)}|^2)=\mathbb{E}(|W_{j,0}|^2) =
\mathbb{E}(S_{n,j})\;,
\]
by~(\ref{e:Wjk}) and~(\ref{e:sumOFprodwjk}). Therefore, in the univariate case $m=1$,
$\overline{\bS}_{n,j}=S_{n,j}-\mathbb{E}(S_{n,j})$ can be
expressed as stated in~(\ref{e:decompSnjLD}). The generalization
to the case $m\geq 2$ is straightforward.
\end{proof}

We prove in Section~\ref{sec:LTSigma} that
\begin{enumerate}[label=$\bullet$]
\item The leading term of
$\mathbf{\Sigma}_{n,j}^{(0)}+\mathbf{\Sigma}_{n,j}^{(1)}$ is
$\frac{c_{q_0^*}^2}{(q_0^*-1)!}\mathbf{S}_{n,j}^{(q_0^*,q_0^*,q_0^*-1)}$ (see
Propositions~\ref{pro:intSigma0} and \ref{pro:Sigma0}) where
\begin{equation}\label{e:q0star}
q_0^*=\left\{
\begin{array}{l}
q_1\mbox{ if }q_0=1,\\
q_0\mbox{ otherwise}.
\end{array}
\right.
\end{equation}
Note that $\mathbf{S}_{n,j}^{(q_0^*,q_0^*,q_0^*-1)}$ always is in the 2nd Wiener
chaos.
\item The leading term of $\mathbf{\Sigma}_{n,j}^{(2)}$ is $\frac{c_1
    c_{q_{m_0}}}{(q_{m_0}-1)!}\mathbf{S}_{n,j}^{(1,q_{m_0},1)}$ (see
  Propositions~\ref{pro:intSigma2} and \ref{pro:Sigma2}), which is in the
  $(q_{m_0}-1)$-th Wiener chaos.
\item The leading term of $\mathbf{\Sigma}_{n,j}^{(3)}$ is
$2\frac{c_{q_{\ell_0}} c_{q_{\ell_0}+1}}{q_{\ell_0}}\mathbf{S}_{n,j}^{(q_{\ell_0},q_{\ell_0}+1,q_{\ell_0})}$ (see
Propositions~\ref{pro:intSigma3} and \ref{pro:Sigma3}), which is Gaussian.
\end{enumerate}
Hence, for  the two classes of functions considered in
Sections~\ref{s:main1} and~\ref{s:main23}, we have to compare at
most four terms : $\mathbf{S}_{n,j}^{(1,1,0)}$,
$\mathbf{S}_{n,j}^{(q_0^*,q_0^*,q_0^*-1)}$,
$\mathbf{S}_{n,j}^{(q_{\ell_0},q_{\ell_0}+1,p)}$,
$\mathbf{S}_{n,j}^{(1,q_{m_0},1)}$ which are respectively
asymptotically Gaussian, Rosenblatt, Gaussian and in the chaos of
order $q_{m_0}-1$.

Our three theorems are based on the study of the asymptotic
behavior of each sum (see Section~\ref{sec:LTSigma} below). We
first establish some preliminary results.
\section{Preliminary results}\label{s:preliminary}
\subsection{$L^2$ bounds}
To identify the leading terms, using the same approach than
in~\cite{clausel-roueff-taqqu-tudor-2011b} we will give an upper bound for the
$L^2$ norm of the multidimensional terms $\mathbf{S}_{n,j}^{(q,q',p)}$,
$q,q',p$ defined in~(\ref{e:Snjrrp1}) and~(\ref{e:Snjrrp}). Here, the main
difficulty is that unlike the case where $G=H_{q_0}$, we have to deal with an
infinity of terms. We have also to obtain more precise bounds than
in~\cite{clausel-roueff-taqqu-tudor-2011b}. In the following, for any
random vector
$\bZ$, the $L^2(\Omega)$ norm of $\bZ$ is denoted by
\begin{equation}\label{e:L2norm}
\|\bZ\|_2=\left(\mathbb{E}\left[|\bZ|^2\right]\right)^{1/2}\;.
\end{equation}
(Recall that $|\bZ|$ denotes the Euclidean norm of $\bZ$.)  Our goal in this
section is to specify how $\|\mathbf{S}_{n,j}^{(q,q',p)}\|_2$ depends on $q,q'$
and $p$. The difficulty is that the sum $\overline{\mathbf{S}}_{n,j}$
contains long--memory and short--memory terms having then different
normalization factors. To recover all the cases, we shall use not only
$\delta_+(q)$ and $\delta(q)$ defined in~(\ref{e:ldparamq}) but also
\begin{equation}\label{e:dplus}
\delta_-(q)=\max(-\delta(q),0)\;,\quad q\geq 0\;,
\end{equation}
so that
$\delta=\delta_+-\delta_-$ and $\delta_+,\delta_-$ are nonnegative. In
particular, $\delta(0)=\delta_+(0)=1/2$ and $\delta_-(0)=0$.

As in~\cite{clausel-roueff-taqqu-tudor-2011b}, the expression~(\ref{e:Snjrrp})
of $\mathbf{S}_{n,j}^{(q,q',p)}$ involves the kernel
${\widehat{\bkappa}^{(p)}_j}$ defined in~(\ref{e:intrepKjp}) and we have to
distinguish the two cases $p\neq 0$ and $p=0$. The following notations will be
used in the sequel. For any $s\in\mathbb{Z}_+$ and $d\in (0,1/2)$, set
\begin{equation}\label{e:KC}
\Lambda_s(a)= \prod_{i=1}^{s}(a_i!)^{1-2d},\quad\forall
a=(a_{1},\cdots,a_{s})\in\mathbb{N}^s\;.
\end{equation}
For any $q,q',p\geq 0$, set
\begin{align}\label{e:alpha}
\alpha(q,q',p)&=
\begin{cases}
\min\left(1-\delta_+(q-p)-\delta_+(q'-p),1/2\right)
&\text{if }p\neq 0\;,\\
1/2&\text{if }p=0\;,
\end{cases}\\
\label{e:beta}
\beta(q,p)&=\max\left(\delta_+(p)+\delta_+(q-p)-1/2,0\right)\\
\label{e:betap}
\beta'(q,q',p)&=\max\left(2\delta_+(p)+\delta_+(q-p)+\delta_+(q'-p)-1,-1/2\right)\;.
\end{align}
Notice  that for any $q\geq 0$, $\beta(q,0)=\delta_+(q)$. Define
the function $\varepsilon$ on $\mathbb{Z}_+$ as
\begin{equation}
\label{e:varepsilontilde}
\varepsilon(p) =
\begin{cases}
0 &\text{if for any }s\in\{1,\cdots,p\},\,s(1-2d)\neq 1\;,\\
1&\text{if for some
}s\in\{1,\cdots,p\},\,s(1-2d)=1\;.
\end{cases}
\end{equation}
The index $K$ is defined in~(\ref{e:DefY}) and the index $M$ is defined
in~(\ref{e:majoHj}), and, as noted in Appendix~\ref{s:appendixB}, the filter
$\bh_j(t)$ has null moments of order $0,1,\dots,M-1$.
\begin{proposition}\label{pro:SD}
Under Assumptions~A, the following bounds hold:
\begin{enumerate}[label=(\roman*)]
\item\label{item:pro:SD1i}
There exists some $C>0$ such that, for all $n,\gamma_j\geq 2$ and $1\leq q\leq
q'$ and $1\leq p\leq \min(q,q'-1)$,
\begin{equation}\label{e:boundL2P51a}
\begin{array}{lll}
\|\mathbf{S}_{n,j}^{(q,q',p)}\|_{2}&\leq&
C^{\frac{q+q'}{2}}\Lambda_2(q-p,p)^{1/2}\Lambda_2(q'-p,p)^{1/2}\gamma_j^{2K}\\
&&\times [n^{-\alpha(q,q',p)}\gamma_j^{\beta'(q,q',p)}+n^{-1/2}\gamma_j^{\beta(q,p)+\beta(q',p)}]\\
&&\times\left(\log n\right)^{\varepsilon(q+q'-2p)}(\log\gamma_j)^{3\varepsilon(q')}\;.
\end{array}
\end{equation}
\item\label{item:pro:SD1ii}
If $M\geq K+\max(\delta_+(q),\delta_+(q'))$, there exists some $C>0$ such that
 $n,\gamma_j\geq 2$ and $1\leq q\leq q'$,
\begin{equation}\label{e:boundL2P51b}
\|\mathbf{S}_{n,j}^{(q,q',0)}\|_{2}\leq
C^{\frac{q+q'}{2}}\Lambda_1(q)^{1/2}\Lambda_1(q')^{1/2}
n^{-1/2}\gamma_j^{2K+\delta_+(q)+\delta_+(q')}
\left(\log\gamma_j\right)^{\varepsilon(q')}\;.
\end{equation}
\end{enumerate}
\end{proposition}
\begin{proof}
  Proposition~\ref{pro:SD} extends Proposition~7.1 in
  \cite{clausel-roueff-taqqu-tudor-2011b}.  Its proof follows the same lines,
  see Appendix~\ref{sec:postponed-proof} for details.
\end{proof}
The following result will be sufficient to find the leading term in
Section~\ref{sec:LTSigma}.
\begin{corollary}\label{cor:SD}
  Under Assumptions~A, if $M\geq K+\max(\delta_+(q),\delta_+(q'))$, then here exists some $C>0$ whose value depends only on
$d$ and $f^*$ such that for all $n,j\geq 2$, $1\leq q\leq q'$ and $0\leq p\leq \min(q,q'-1)$,
\begin{multline}\label{e:boundL2P51ab}
\|\mathbf{S}_{n,j}^{(q,q',p)}\|_{2}\leq
C^{\frac{q+q'}{2}}\Lambda_2(q-p,p)^{1/2}\Lambda_2(q'-p,p)^{1/2}
n^{-\alpha(q,q',p)}\left(\log n\right)^{\varepsilon(q+q'-2p)}\\
\times\gamma_j^{2K+\beta(q,p)+\beta(q',p)}
(\log\gamma_j)^{3\varepsilon(q')}\;.
\end{multline}
\end{corollary}
\begin{proof}
  We observe that, for all $1\leq p\leq q\leq q'$,
  $\beta'(q,q',p)\leq\beta(q,p)+\beta(q',p)$ and $\alpha(q,q',p)\leq1/2$. Hence
  the term between brackets in the right-hand side of~(\ref{e:boundL2P51a}) is
  bounded by
  $n^{-\alpha(q,q',p)}\times\gamma_j^{2K+\beta(q,p)+\beta(q',p)}$. This
  gives~(\ref{e:boundL2P51ab}) in the case $p\geq1$.  The case $p=0$ is
  obtained by using~(\ref{e:boundL2P51b}) and computing
  $\Lambda_2(q,0)=\Lambda_1(q)$, $\alpha(q,q',0)=1/2$, $\beta(q,0)=\delta_+(q)$
  and $\beta(q',0)=\delta_+(q')$.
\end{proof}
\subsection{Asymptotic behavior of the leading terms}

We now investigate the exact asymptotic behavior of the terms that will turn
out to be leading in the sum~(\ref{e:decompSnjLD}).

Let us first suppose that the bounds in Proposition~\ref{pro:SD} are sharp
enough to determine which terms are leading. Since $\gamma_j\to\infty$ and
$n=n_j\to\infty$, those for which the bounds have the largest exponents
$\beta(q,p)$ and $\beta(q',p)$ and the lowest exponent $\alpha(q,q',p)$ are
more likely to dominate, in particular, if $\delta_+(p)$, $\delta_+(q-p)$,
$\delta_+(q'-p),\delta_+(p)+\delta_+(q-p)-1/2$,
$\delta_+(p)+\delta_+(q'-p)-1/2$ and $1/2-(1-\delta_+(q-p)-\delta_+(q'-p))$ are
all positive. Using~(\ref{e:dqq>0}), if $p>0$, this happens for
$0<p,q-p,q'-p,q,q',q-p+q'-p<1/(1-2d)$, that is (taking $q\leq q'$ without loss
of generality),
\begin{equation}\label{e:idLT}
0<p\leq q\leq q'<1/(1-2d)\quad\text{and}\quad 0<q+q'-2p<1/(1-2d)\;.
\end{equation}
In particular, for such a triplet $(p,q,q')$, we have
$\varepsilon(q')=\varepsilon(q+q'-2p)=0$ so that bounds
in~(\ref{e:boundL2P51a}) and~(\ref{e:boundL2P51b}) involving logarithms will
not appear in these terms.  We shall check
afterwards (in Section~\ref{sec:LTSigma}) that indeed, in all the cases
we consider, either such a term is leading in the sum~(\ref{e:decompSnjLD}), or
the leading term is  $\mathbf{S}_{n_j,j}^{(1,1,0)}$ ($q=q'=1$ and $p=0$). The bounds established in
Proposition~\ref{pro:SD} will be sharp enough for this goal.

This is why, in the following, we shall only determine the asymptotic
behaviors of $\mathbf{S}_{n_j,j}^{(1,1,0)}$ and of
$\mathbf{S}_{n_j,j}^{(q,q',p)}$ under Condition~(\ref{e:idLT}),
when $j,n_j\to\infty$.

\begin{proposition}\label{pro:Snj110}
Suppose that Assumptions A hold with $M\geq K+\delta(1)=K+d$ and that $\gamma_j$
is even for all $j$. Let
$(n_j)$ be any diverging sequence of integers. Then as
$j\to\infty$,
\begin{equation}\label{e:normgauss}
n_j^{1/2}\gamma_j^{-2(d+K)}\mathbf{S}_{n_j,j}^{(1,1,0)}\overset{\tiny{(\mathcal{L})}}{\rightarrow}\mathcal{N}(0,\Gamma)\;,
\end{equation}
where $\Gamma$ is defined by~(\ref{e:Gamma}).
\end{proposition}
\begin{proof}
  This is a direct application of Theorem~3.1 case (a)
  in~\cite{clausel-roueff-taqqu-tudor-2011b}.
\end{proof}
We now consider the case where Condition~(\ref{e:idLT}) is satisfied.
\begin{proposition}\label{pro:lrd}
Let $q,q'$ and $p$ be non-negative integers such that~(\ref{e:idLT}) holds.
Assume that Assumptions A hold with $M\geq K$ and let $(n_j)$ be any diverging sequence of integers.
Then, as $j\to\infty$,
\begin{equation}\label{e:norma}
(n_j\gamma_j)^{1-\delta(q-p)-\delta(q'-p)}\gamma_j^{-2(K+\delta(p))}
\mathbf{S}_{n_j,j}^{(q,q',p)}
\overset{\tiny{(\mathcal{L})}}{\rightarrow} [f^*(0)]^{(q+q')/2}\;
\K_p \; \; Z_{q+q'-2p,d}(1)\;,
\end{equation}
where $Z_{q+q'-2p,d}$ is the Hermite process defined in~(\ref{e:harmros}) and
$\K_p$ is defined in~(\ref{e:defKp}).
\end{proposition}
\begin{proof}
  The proof follows the same line as the proof of Proposition~8.1
  in~\cite{clausel-roueff-taqqu-tudor-2011b}. Therefore we only explain how to
  adapt this proof to our setting. Set $r=q+q'-2p$.
 Using~(\ref{e:Snjrrp}) and that, for all
  $g\in L^2(\mathbb{R}^r)$,
  $$
\widehat{I}_r(g)\overset{d}{=} (n\gamma_j)^{-r/2}
\widehat{I}_r(g(\cdot/(n\gamma_j)))\;,
$$
we have
\begin{equation}
  \label{eq:Ito-int-rep}
\mathbf{S}_{n,j}^{(q,q',p)}\overset{d}{=}
(n\gamma_j)^{-r/2}
\widehat{I}_{r}\left(
D_n\circ\Sigma_{q+q'-2p}(\cdot/n)\times
[\1_{(-\gamma_j\pi,\gamma_j\pi)}]^{\otimes r}(\cdot/n)
\times\bof_j\right)\;,
\end{equation}
where, for all $\xi\in\mathbb{R}^r$,
$$
\bof_j(n\gamma_j\xi)=
 \sqrt{f}^{\otimes r}(\xi) \times
 {\widehat{\bkappa}}_j^{(p)}\circ\Sigma_{q-p,q'-p}(\xi)\;.
$$
The rest of the proof consists in proving the $L^2$ convergence of the Itô
integral in~(\ref{eq:Ito-int-rep}), adequately normalized. This is done in the
proof of Proposition~8.1 in~\cite{clausel-roueff-taqqu-tudor-2011b} with
$q-p=q'-p=1$ (hence $r=2$).  The same proof applies in our setting but results
in a multiple integral of order $r$ with $r\geq2$. In particular, if $r>2$ the
asymptotic limit is not Rosenblatt but an $r$-order Hermite process.
\end{proof}
\section{Leading terms}\label{sec:LTSigma}
Recall the decomposition~(\ref{e:s3}) of $\overline{\bS}_{n,j}$ using sums
$\mathbf{\Sigma}_{n,j}^{(0)}$, $\mathbf{\Sigma}_{n,j}^{(1)}$,
$\mathbf{\Sigma}_{n,j}^{(2)}$ and $\mathbf{\Sigma}_{n,j}^{(3)}$. The aim of
this section is to identify the leading terms of the three following sums~:
$\mathbf{\Sigma}_{n_j,j}^{(0)}+\mathbf{\Sigma}_{n_j,j}^{(1)}$,
$\mathbf{\Sigma}_{n_j,j}^{(2)}$, $\mathbf{\Sigma}_{n_j,j}^{(3)}$ under the
conditions specified in Sections~\ref{s:main1} and~\ref{s:main23}.
\subsection{Leading term of $\mathbf{\Sigma}_{n_j,j}^{(0)}+\mathbf{\Sigma}_{n_j,j}^{(1)}$}\label{sec:LTSigma0}
Recall that the two sums
$\mathbf{\Sigma}_{n_j,j}^{(0)},\mathbf{\Sigma}_{n_j,j}^{(1)}$ are defined in
equations~(\ref{e:defSigma0}), (\ref{e:defSigma1}) and that $q_0^*$, defined
in~(\ref{e:q0star}), equals $q_1$ if $q_0=1$ and equals $q_0$
otherwise. Therefore $q_0\geq2$. We shall prove that, if $q_0^*<1/(1-2d)$, the
main term in $\mathbf{\Sigma}_{n_j,j}^{(0)}+\mathbf{\Sigma}_{n_j,j}^{(1)}$ is
$\frac{c_{q_0^*}^2}{(q_0^*-1)!}\mathbf{S}_{n,j}^{(q_0^*,q_0^*,q_0^*-1)}$,
and has rate $n_j^{-(1-2d)}\gamma_j^{2(\delta(q_0^*)+K)}$.  The following
proposition is used to show that the remainder terms are negligible.

\begin{proposition}\label{pro:intSigma0}
Suppose that Assumptions A hold with $M\geq K+\delta(q_0^*)$ and that
\[
q_0^*<1/(1-2d)\;,
\]
where $q_0^*$ is defined in~(\ref{e:q0star}). Let $(n_j)$ be a diverging
sequence.  Then, when $j\to\infty$,
\begin{gather}\label{e:supSigma0a}
n_j^{1-2d}\gamma_j^{-2(\delta(q_0^*)+K)}
\left(\sum_{p=0}^{q_0^*-2}\frac{c_{q_0^*}^2}{(q_0^*!)^2}
p!{{q_{0}^*}\choose{p}}^2 \,
\|\mathbf{S}_{n_j,j}^{(q_{0}^*,q_{0}^*,p)}\|_2\right)
\rightarrow 0\;,\\
\label{e:supSigma0b}
n_j^{1-2d}\gamma_j^{-2(\delta(q_0^*)+K)}
\left(\sum_{\ell\in\ellset,q_\ell>q_0^*}\frac{c_{q_\ell}^2}{(q_\ell
!)^2}\sum_{p=0}^{q_\ell-1} p!{{q_{\ell}}\choose{p}}^2\,
\|\mathbf{S}_{n_j,j}^{(q_{\ell},q_{\ell},p)}\|_2\right)\rightarrow
0\;,\\
\label{e:supSigma0c}
n_j^{1-2d}\gamma_j^{-2(\delta(q_0^*)+K)}
\left(\sum_{(\ell_1,\ell_2)\in
J}\frac{c_{q_{\ell_1}}}{q_{\ell_1}!}\frac{c_{q_{\ell_2}}}{q_{\ell_2}!}\sum_{p=0}^{q_{\ell_1}}
p!{{q_{\ell_1}}\choose{p}}{{q_{\ell_2}}\choose{p}}\,
\|\mathbf{S}_{n_j,j}^{(q_{\ell_1},q_{\ell_2},p)}\|_2\right)\rightarrow
0\;.
\end{gather}
\end{proposition}
\begin{proof}
We first note that, since $q_0^*\geq2$ by definition and $q_0^*<1/(1-2d)$ by
assumption, we have $1/4<d<1/2$.

As in the proofs of Proposition~\ref{pro:decompSnjLD}, we prove the result in the case
where $m=1$ without loss of generality. We thus use non bold faced symbols.

We first prove~(\ref{e:supSigma0a}). Since there is a finite number of terms in
the sum appearing on the left-hand side of~(\ref{e:supSigma0a}), it is
sufficient to show that each term converges to 0. Let
$p\in\{0,\dots,q^*_0-2\}$. We apply Corollary~\ref{cor:SD} with
$q=q'=q_0^*$. Since $q_0^*<1/(1-2d)$ and thus $\varepsilon(q_0^*)=0$,
Inequality~(\ref{e:boundL2P51ab}) reads
\begin{equation*}%\label{e:Snjqqpspe}
\gamma_j^{-2(\delta(q_0^*)+K)}
\|S_{n,j}^{(q_0^*,q_0^*,p)}\|_2\leq
C^{q_0^*}\Lambda_2(q_0^*-p,p)
n^{-\alpha(q_0^*,q_0^*,p)}\left(\log n\right)^{\varepsilon(2(q_0^*-p))}
\gamma_j^{2(\beta(q_0^*,p)-\delta(q_0^*))}\;.
\end{equation*}
By~(\ref{e:expnaqq}) and~(\ref{e:expj}) in Lemma~\ref{lem:exp}, we have
$\alpha(q_0^*,q_0^*,p)\geq\min(2(1-2d),1/2)$ and
$\beta(q_0^*,p)\leq\delta_+(q_0^*)=\delta(q_0^*)$. Hence,
$$
\gamma_j^{-2(\delta(q_0^*)+K)}\|S_{n,j}^{(q_0^*,q_0^*,p)}\|_2
\leq C^{q_0^*}\Lambda_2(q_0^*-p,p)
n^{-\min(2(1-2d),1/2)}\left(\log n\right)^{\varepsilon(2(q_0^*-p))}\;.
$$
Since $d\in (1/4,1/2)$, we have
$n^{-\min(2(1-2d),1/2)}(\log n)^{\varepsilon(2(q_0^*-p))}=o(n^{2d-1})$.
Thus Inequality~(\ref{e:supSigma0a}) holds.

We now prove~(\ref{e:supSigma0b}). We apply~(\ref{e:boundL2P51ab}), in the cases
$q=q'=q_\ell$, $0\leq p\leq q_\ell-2$ and $q=q'=q_\ell$, $p=q_{\ell}-1$,
successively.  Then for some $C>0$, for any $\ell\in\ellset$ such that
$q_{\ell}>q_0^*$ and for any $0\leq p\leq q_\ell-2$ one has,
\[
\gamma_j^{-2(\delta(q_0^*)+K)}\|S_{n,j}^{(q_\ell,q_\ell,p)}\|_2\leq
C^{q_\ell}\Lambda_2(q_\ell-p,p)
n^{-\alpha(q_\ell,q_\ell,p)}\left(\log n\right)^{\varepsilon(2(q_{\ell}-p))}
\gamma_j^{2(\beta(q_\ell,p)-\delta(q_0^*))}(\log\gamma_j)^3\;.
\]
On the other hand, since $d>1/4$, $\varepsilon(2)=0$. Thus
in the case where $p=q_\ell-1$, the exponent of $\log n$ vanishes.
Moreover, in this case, by~(\ref{e:expnaqqbis}) in Lemma~\ref{lem:exp},
$\alpha(q_\ell,q_\ell,p)=\min(1-2d,1/2)=1-2d$. In the alternative case
$p<q_\ell-1$, we use~(\ref{e:expnaqq}) in Lemma~\ref{lem:exp}, which gives
$$
n^{-\alpha(q_\ell,q_\ell,p)}\left(\log n\right)^{\varepsilon(2(q_{\ell}-p))}
\leq n^{-\min(2(1-2d),1/2)}\left(\log n\right)^{\varepsilon(2(q_{\ell}-p))}
\leq n^{2d-1}\;,
$$
for $n$ large enough, since $2(1-2d)>1-2d$ and $1/2>1-2d$.
Hence in all the cases, the terms in $n$ can be bounded by $C'\,n^{2d-1}$.
As for the terms in $\gamma_j$, we use that, by ~(\ref{e:expj}) in
Lemma~\ref{lem:exp}, $\beta(q_\ell,p)\leq
\delta_+(q_\ell)\leq\delta_+(q_0^*+1)$, since $q_\ell\geq q^*_0+1$ and
$\delta_+$ is non-increasing. Hence we get that
$$
\gamma_j^{-2(\delta(q_0^*)+K)}\|S_{n,j}^{(q_\ell,q_\ell,p)}\|_2\leq
C^{q_\ell}\Lambda_2(q_\ell-p,p)n^{2d-1}\;
\gamma_j^{2(\delta_+(q_0^*+1)-\delta(q_0^*))}(\log\gamma_j)^3\;,
$$
where $C>0$ may have changed from the previous line. Hence, summing over
$\ell\in\ellset$ such that $q_\ell>q_0^*$ and all $p\in\{0,\dots,q_\ell-1\}$,
we get that
\begin{multline}\label{eq:cqlambda-Lambda2}
n^{1-2d}\gamma_j^{-2(\delta(q_0^*)+K)}
\left(\sum_{q_\ell>q_0^*}\sum_{p=0}^{q_\ell-1}\frac{c_{q_\ell}^2}{(q_\ell
!)^2} p!{{q_{\ell}}\choose{p}}^2 \,
\|S_{n,j}^{(q_{\ell},q_{\ell},p)}\|_2\right)\\
\leq
\gamma_j^{2(\delta_+(q_{0}^*+1)-\delta(q_0^*))}(\log\gamma_j)^3
\sum_{\ell=1}^{+\infty}\frac{c_{q_\ell}^2}{(q_\ell !)^2}\,C^{q_\ell}
\sum_{p=0}^{q_\ell-1} p!\,
{{q_{\ell}}\choose{p}}^2 \Lambda_2(q_\ell-p,p)\;.
\end{multline}
Observe now that
\[
\sum_{\ell=1}^{+\infty}\frac{c_{q_\ell}^2}{(q_\ell !)^2}\,C^{q_\ell}\sum_{p=0}^{q_\ell-1} p!\,
{{q_{\ell}}\choose{p}}^2 \Lambda_2(q_\ell-p,p)\leq\sum_{q,q'\geq 1}\sum_{p=0}^{q\wedge q'}\frac{|c_q|}{q!}
\frac{|c_{q'}|}{q'!} \; p!\;
{{q}\choose{p}}{{q'}\choose{p}}C^{\frac{q+q'}{2}}\Lambda_2(q-p,p)^{1/2}\Lambda_2(q'-p,p)^{1/2}
\]
Since Condition~(\ref{e:condcv}) holds, Lemma~\ref{lem:c_q-cond} implies that~:
\begin{equation}\label{e:condconv0b}
\sum_{\ell=1}^{+\infty}\frac{c_{q_\ell}^2}{(q_\ell !)^2}\,C^{q_\ell}\sum_{p=0}^{q_\ell-1} p!\,
{{q_{\ell}}\choose{p}}^2 \Lambda_2(q_\ell-p,p)<\infty\;.
\end{equation}
Finally we observe that, since
$\delta(q_{0}^*)>\delta(q_0^*+1)$ and $\delta(q_{0}^*)>0$, we have
\begin{equation}\label{e:conv0b}
\gamma_j^{2(\delta_+(q_{0}^*+1)-\delta(q_0^*))}(\log\gamma_j)^3\to0\;,
\end{equation}
as $\gamma_j\to\infty$. Hence~(\ref{e:condconv0b}) and~(\ref{e:conv0b}) imply that Inequality~(\ref{e:supSigma0b}) holds.

We finally prove that~(\ref{e:supSigma0c})
holds. Inequality~(\ref{e:boundL2P51ab}) for $(\ell_1,\ell_2)\in J$ with
$q=q_{\ell_1}, q'=q_{\ell_2}$ and $p\leq q_{\ell_1}$ implies that
\begin{multline}\label{eq:supsigma0c-inter1}
\gamma_j^{-2(\delta(q_0^*)+K)}\|S_{n,j}^{(q_{\ell_1},q_{\ell_2},p )}\|_2
\leq
C^{\frac{q_{\ell_1}+q_{\ell_2}}{2}}\Lambda_2(q_{\ell_1}-p,p)^{1/2}\Lambda_2(q_{\ell_2}-p,p)^{1/2}\\
\times n^{-\alpha(q_{\ell_1},q_{\ell_2},p)}\left(\log n\right)^{\varepsilon(q_{\ell_1}+q_{\ell_2}-2p)}
\gamma_j^{\beta(q_{\ell_1},p)+\beta(q_{\ell_2},p)-2\delta(q_0^*)}
(\log \gamma_j)^3\;.
\end{multline}
We first bound the terms that depend on $n$. First suppose that $p=q_{\ell_1}$
and $q_{\ell_2}=q_{\ell_1}+2$. In this case, the exponent of $\log n$ vanishes,
since $\varepsilon(2)=0$ for $d>1/4$, and by~(\ref{e:expna}) in
Lemma~\ref{lem:exp}, the exponent of $n$
$\alpha(q_{\ell_1},q_{\ell_2},p)\geq1-2d$. Hence, in this case, the terms in
$n$ are bounded by $n^{2d-1}$. Otherwise, if $p<q_{\ell_1}$ or
$q_{\ell_2}>q_{\ell_1}+2$, we observe that for $(\ell_1,\ell_2)\in J$, we have
$p\leq q_{\ell_2}-3$ and hence, by definition of $\alpha$ in~(\ref{e:alpha})
and since $\delta_+$ is non-increasing,
\[
\alpha(q_{\ell_1},q_{\ell_2},p)\geq 1/2-\delta^+(3)=
\min(1/2,1/2-(3d-1))>1-2d \;,
\]
since $1/4<d<1/2$. Whatever the exponent of $\log n$, we again obtain
that  the terms in $n$ are bounded by $n^{2d-1}$, up to a multiplicative
constant:
\begin{equation}\label{e:boundn}
\sup_{(\ell_1,\ell_2)\in J,0\leq p\leq q_{\ell_1}}
n^{-\alpha(q_{\ell_1},q_{\ell_2},p)}\left(\log n\right)^{\varepsilon(q_{\ell_1}+q_{\ell_2}-2p)}=O\left(n^{2d-1}\right)\;.
\end{equation}
We now bound the terms that depend on $\gamma_j$
in~(\ref{eq:supsigma0c-inter1}).
By~(\ref{e:expj}) in Lemma~\ref{lem:exp}, we have $\beta(q,p)\leq \delta_+(q)$
for $0\leq p\leq q$. Thus
$\beta(q_{\ell_1},p)+\beta(q_{\ell_2},p)-\delta(q_0^*)\leq
\delta_+(q_{\ell_1})+\delta_+(q_{\ell_2})-2\delta(q_0^*)$.
Since $\delta$ is non--increasing,  $q_{\ell_1}\geq q_0^*$ and $q_{\ell_2}\geq
q_{\ell_1}+2$ we deduce that $\delta_+(q_{\ell_1})\leq
\delta_+(q_0^*)=\delta(q_0^*)$ and $\delta_+(q_{\ell_2})\leq
\delta_+(q_0^*+2)<\delta(q_0^*)$. Hence the exponent of $\gamma_j$ is bounded
by a negative constant and
\begin{equation}\label{e:boundgammaj}
\sup_{(\ell_1,\ell_2)\in J,0\leq p\leq q_{\ell_1}}
\gamma_j^{\beta(q_{\ell_1},p)+\beta(q_{\ell_2},p)-2\delta(q_0^*)}(\log
\gamma_j)^3\to0 \quad\text{as $j\to\infty$}\;.
\end{equation}
In view of~(\ref{eq:supsigma0c-inter1}),~(\ref{e:boundn})
and~(\ref{e:boundgammaj}), the proof of~(\ref{e:supSigma0c}) follows from the
bound
\begin{equation}\label{e:supsigma0c-inter2}
\sum_{(\ell_1,\ell_2)\in
J}\frac{|c_{q_{\ell_1}}|}{q_{\ell_1}!}\frac{|c_{q_{\ell_2}}|}{q_{\ell_2}!}C^{\frac{q_{\ell_1}+q_{\ell_2}}{2}}
\sum_{p=0}^{q_{\ell_1}}p!\,(2\pi)^p\,
\prod_{i=1}^2{{q_{\ell_i}}\choose{p}}
[\Lambda_2(q_{\ell_i}-p,p)]^{1/2} < \infty \;,
\end{equation}
which follows from the inequality
\begin{eqnarray*}
&&\sum_{(\ell_1,\ell_2)\in
J}\frac{|c_{q_{\ell_1}}|}{q_{\ell_1}!}\frac{|c_{q_{\ell_2}}|}{q_{\ell_2}!}C^{\frac{q_{\ell_1}+q_{\ell_2}}{2}}
\sum_{p=0}^{q_{\ell_1}}p!\,(2\pi)^p\,
\prod_{i=1}^2{{q_{\ell_i}}\choose{p}}
[\Lambda_2(q_{\ell_i}-p,p)]^{1/2}\\
&\leq&\sum_{q,q'\geq 1}\sum_{p=0}^{q\wedge q'}\frac{|c_q|}{q!}
\frac{|c_{q'}|}{q'!} \; p!\;
{{q}\choose{p}}{{q'}\choose{p}}(2\pi C)^{\frac{q+q'}{2}}\Lambda_2(q-p,p)^{1/2}\Lambda_2(q'-p,p)^{1/2}
\end{eqnarray*}
and from Lemma~\ref{lem:c_q-cond} with Condition~(\ref{e:condcv}). This concludes the proof.
\end{proof}
We now focus on the leading term of the sum
$\mathbf{\Sigma}_{n_j,j}^{(0)}+\mathbf{\Sigma}_{n_j,j}^{(1)}$.
\begin{proposition}\label{pro:Sigma0}
  Under the same assumptions as Proposition~\ref{pro:intSigma0}, we have, as
  $j\to\infty$,
\begin{equation}\label{e:PrepSigmanj01result}
n_j^{1-2d}\gamma_j^{-2(\delta(q_0^*)+K)}\left(\mathbf{\Sigma}_{n_j,j}^{(0)}+\mathbf{\Sigma}_{n_j,j}^{(1)}\right)
\overset{\tiny{(\mathcal{L})}}{\longrightarrow}\left[\frac{c_{q_0^*}^2}{(q_0^*-1)!}f^*(0)^{q_0^*}\K_{q_0^*-1}\right]
Z_{2,d}(1)\;.
\end{equation}\end{proposition}
\begin{proof}
We apply Proposition~\ref{pro:lrd} with $q=q'=q_0^*$ and
$p=q_{0}^*-1$. Since
\[
2\delta(1)+2\delta(q_0^*-1)-1=2\delta(q_0^*)\;,
\]
we get that
\begin{equation}\label{e:PrepSigmanj01}
n_j^{1-2d}\gamma_j^{-2(\delta(q_0^*)+K)}
\frac{c_{q_0^*}^2\,(2\pi)^{q^*_0-1}}{(q_0^*-1)!}\mathbf{S}_{n,j}^{(q_0^*,q_0^*,q_0^*-1)}
\overset{\tiny{(\mathcal{L})}}{\rightarrow}
\left[\frac{c_{q_0^*}^2}{(q_0^*-1)!}f^*(0)^{q_0^*}
\K_{q_0^*-1}\right]Z_{2,d}(1)\;.
\end{equation}
The left-hand side in~(\ref{e:PrepSigmanj01}) corresponds to the term
$q_{\ell}=q^*_0$ and $p=q^*_0-1$ of $\mathbf{\Sigma}_{n_j,j}^{(0)}$
in~(\ref{e:defSigma0}). The terms of  $\mathbf{\Sigma}_{n_j,j}^{(0)}$ with
$q_{\ell}=q^*_0$ and $p<q^*_0-1$ are gathered in the left-hand side
of~(\ref{e:supSigma0a}). The terms  of  $\mathbf{\Sigma}_{n_j,j}^{(0)}$ with
$q_{\ell}>q^*_0$ are gathered in the left-hand side
of~(\ref{e:supSigma0b}). Finally the left-hand side of~(\ref{e:supSigma0c})
corresponds to $\mathbf{\Sigma}_{n_j,j}^{(1)}$ in~(\ref{e:defSigma0}).
Hence, by Proposition~\ref{pro:intSigma0}, all these terms are negligible
and~(\ref{e:PrepSigmanj01result}) holds.
\end{proof}
\subsection{Leading term of $\mathbf{\Sigma}_{n_j,j}^{(2)}$}\label{sec:LTSigma2}
In this section, we investigate the asymptotical behavior of the sum
$\mathbf{\Sigma}_{n_j,j}^{(2)}$ defined in~(\ref{e:defSigma2}). We shall prove
that, if $q_{m_0}<1/(1-2d)$, the leading term of this sum is $ \frac{c_1
  c_{q_{m_0}}}{(q_{m_0}-1)!}\mathbf{S}_{n,j}^{(1,q_{m_0},1)}$ and has rate
$n_j^{-(1-2\delta(q_{m_0}-1))/2}\gamma_j^{\delta(q_{m_0})+d+2K}$. To this end
we first show that the remainder terms are negligible.
\begin{proposition}\label{pro:intSigma2}
Assume that Assumptions A hold with $M\geq K+d$ and that
\[
q_{m_0}<1/(1-2d)\;,
\]
where $q_{m_0}$ is defined by~(\ref{e:m0}).

Let $(n_j)$ be a diverging sequence. Then, as
$j\to\infty$,
\begin{equation}\label{e:supSigma2a}
n_j^{1/2-\delta(q_{m_0}-1)}\gamma_j^{-\delta(q_{m_0})-d-2K}
\left(\sum_{\ell\geq m_0}\frac{c_1 c_{q_{\ell}}}{q_{\ell}!}\|\mathbf{S}_{n_j,j}^{(1,q_{\ell},0)}\|_2\right)\rightarrow
0\;,
\end{equation}
\begin{equation}\label{e:supSigma2b}
n_j^{1/2-\delta(q_{m_0}-1)}\gamma_j^{-\delta(q_{m_0})-d-2K}
\left(\sum_{\ell>m_0}
\frac{c_1c_{q_{\ell}}}{(q_{\ell}-1)!}
\|\mathbf{S}_{n_j,j}^{(1,q_{\ell},1)}\|_2\right)\rightarrow
0\;.
\end{equation}
\end{proposition}
\begin{proof}
  Observe that $\delta_+(1)=d$. We apply~(\ref{e:boundL2P51ab}) in
  Corollary~\ref{pro:SD} with $q=1$ and $q'=q_{\ell}$. Thus there exists some
  $C>0$ such that for any $\ell\geq m_0$
  \begin{equation}
    \label{eq:p0m01somme}
    \gamma_j^{-\delta(q_{m_0})-d-2K}
\|S_{n,j}^{(1,q_\ell,0)}\|_2
\leq C^{\frac{q_{\ell}+1}{2}}(q_{\ell}!)^{1/2-d}\,
n^{-1/2}\gamma_j^{\delta_+(q_{\ell})-\delta(q_{m_0})}
(\log\gamma_j)^{3\varepsilon(q_{\ell})}\;.
  \end{equation}
Since by assumption $q_{m_0}<1/(1-2d)$, we have $\varepsilon(q_{m_0})=0$ and
$\delta_+(q_{m_0})=\delta(q_{m_0})$.
Thus, if $\ell=m_0$, the terms involving $\gamma_j$ vanish in the right-hand side
of~(\ref{eq:p0m01somme}). If $\ell>m_0$, we have
$\delta_+(q_{\ell})<\delta(q_{m_0})$ and these terms are $o(1)$ as
$j\to\infty$. Hence, for $j$ large enough, and for any $\ell\geq m_0$,
\[
n_j^{1/2-\delta(q_{m_0}-1)}\gamma_j^{-\delta(q_{m_0})-d-2K}
\|S_{n_j,j}^{(1,q_\ell,0)}\|_2
\leq C^{\frac{q_{\ell}+1}{2}}\,(q_{\ell}!)^{1/2-d}\,
n_j^{-\delta(q_{m_0}-1)}\;.
\]
Using that
$\delta(q_{m_0}-1)>\delta(q_{m_0})>0$, and that, by Condition~(\ref{e:condcv}),
\[
\sum_{\ell=m_0}^{+\infty}C^{\frac{q_{\ell}+1}{2}}
\,|c_{q_{\ell}}|\,(q_{\ell}!)^{-1/2-d}<+\infty\;,
\]
we obtain the limit~(\ref{e:supSigma2a}).

We now show that~(\ref{e:supSigma2b}) holds. Applying~(\ref{e:boundL2P51ab})
with $q=1$, $q'=q_{\ell}$ and $p=1$, we
get that there exists some $C>0$ such that for any $\ell>m_0$,
\begin{equation}
  \label{eq:p1somme2}
\|\mathbf{S}_{n,j}^{(1,q_{\ell},1)}\|_2\leq C^{\frac{q_{\ell}+1}{2}}
\{(q_{\ell}-1)!\}^{1/2-d}
\,n^{-\alpha(1,q_{\ell},1)}(\log n)^{\varepsilon(q_{\ell}-1)}\gamma_j^{\beta(1,1)+\beta(q_{\ell},1)+2K}(\log\gamma_j)^{3}\;.
\end{equation}
The definition of $\alpha$ and $\beta$ by Equations~(\ref{e:alpha}) and~(\ref{e:beta}), implies that
\[
\alpha(1,q_{\ell},1)=1/2-\delta_+(q_{\ell}-1),\,\beta(1,1)=d,\,\beta(q_{\ell},1)=\max(d+\delta_+(q_{\ell}-1)-1/2,0)\;.
\]
Since $\ell>m_0$, one has $\delta_+(q_{\ell}-1)\leq
\delta_+(q_{m_0+1}-1)$. Thus
$$
n_j^{-\alpha(1,q_{\ell},1)}(\log n_j)^{\varepsilon(q_{\ell}-1)}
\leq n_j^{1/2-\delta_+(q_{m_0+1}-1)}\log n_j
=o\left(n_j^{1/2-\delta(q_{m_0}-1)}\right)\;.
$$
Observe now that for $\ell>m_0$, we have $q_{\ell}-1\geq q_{m_0}$ and thus
$$
\gamma_j^{\beta(1,1)+\beta(q_{\ell},1)+2K}(\log\gamma_j)^{3}
\leq \gamma_j^{d+2K+\max(d+\delta_+(q_{m_0})-1/2,0)}(\log\gamma_j)^{3}
=o\left(\gamma_j^{d+2K+\delta(q_{m_0})}\right) \;.
$$
Now, using the  last  two displayed equations,~(\ref{eq:p1somme2}) and
Condition~(\ref{e:condcv}), we obtain the limit~(\ref{e:supSigma2b}), which
concludes the proof.
\end{proof}
We now deduce the asymptotic behavior of
$\mathbf{\Sigma}_{n_j,j}^{(2)}$.
\begin{proposition}\label{pro:Sigma2}
  Under the same assumptions as Proposition~\ref{pro:intSigma2}, we have as
  $j\to\infty$
\begin{equation}\label{e:Sigma2}
n_j^{(1-2\delta(q_{m_0}-1))/2}\gamma_j^{-(\delta(q_{m_0})+d+2K)}\mathbf{\Sigma}_{n_j,j}^{(2)}
\overset{(\mathcal{L})}{\rightarrow}\frac{2 c_1 c_{q_{m_0}}}{(q_{m_0}-1)!}[f^*(0)]^{(q_{m_0}+1)/2}\K_1
Z_{q_{m_0}-1,d}(1)\;,
\end{equation}
where $\K_1$ is defined in~(\ref{e:defKp}) and
$Z_{q-1,d}$ is the Hermite process defined in~(\ref{e:harmros}).
\end{proposition}
\begin{proof}
We apply Proposition~\ref{pro:lrd} with $q=1$, $q'=q_{m_0}$ and $p=1$. For
these values, since $q_{m_0}<1/(1-2d)$, Condition~(\ref{e:idLT}) is satisfied.
The exponents of $n$ and $\gamma_j$ in the left-hand side of~(\ref{e:norma})
respectively read
$$
1-\delta(q-p)-\delta(q'-p)=1-\delta(0)-\delta(q_{m_0}-1)=1/2-\delta(q_{m_0}-1)
$$
and
$$
1-\delta(q-p)-\delta(q'-p)-2K-2\delta(p)=-\delta(q_{m_0})-d-2K\;.
$$
Hence we get that
\begin{equation}\label{e:PrepSigma2}
n_j^{(1-2\delta(q_{m_0}-1))/2}\gamma_j^{-(\delta(q_{m_0})+d+2K)}\mathbf{S}_{n_j,j}^{(1,q_{m_0},1)}
\overset{\tiny{(\mathcal{L})}}{\rightarrow}
[f^*(0)]^{(q_{m_0}+1)/2}\K_1 Z_{q_{m_0}-1,d}(1)\;.
\end{equation}
Finally we observe that this term corresponds to the second term of the summand
in~(\ref{e:defSigma2}) with index $\ell=q_{m_0}$, up to the multiplicative
constant $4\pi c_1 c_{q_{m_0}}/(q_{m_0}-1)!$. All the other terms are
negligible by Proposition~\ref{pro:intSigma2}. Thus the limit~(\ref{e:Sigma2})
holds.
\end{proof}
\subsection{Leading term of $\mathbf{\Sigma}_{n_j,j}^{(3)}$}\label{sec:LTSigma3}
In this section we investigate the asymptotic behavior of
$\mathbf{\Sigma}_{n,j}^{(3)}$ defined in~(\ref{e:defSigma3}). We first bound
the sum over indices $\ell=\ell_0$ and $p\neq q_{\ell_0}$ and the one over
indices $\ell>\ell_0$ and $p\in\{0,\dots,q_{\ell}\}$. The two sums will turn
out to be negligible.
\begin{proposition}\label{pro:intSigma3}
Assume that Assumptions A hold with $M\geq K+ \delta(q_{\ell_0})$ and
\begin{equation}
  \label{eq:condSogma3}
q_{\ell_0}+1<1/(1-2d)\;.
\end{equation}
Let $(n_j)$ be a diverging sequence. Then, as
$j\to\infty$,
\begin{equation}\label{e:supSigma3a}
n_j^{\frac{1-2d}{2}}\gamma_j^{-(\delta(q_{\ell_0})+\delta(q_{\ell_0}+1)+2K)}
\left(\sum_{p=0}^{q_{\ell_0}-1}\frac{c_{q_{\ell_0}}}{q_{\ell_0}!}\frac{c_{q_{\ell_0}+1}}{(q_{\ell_0}+1)!}
 p!{{q_{\ell_0}}\choose{p}}{{q_{\ell_0}+1}\choose{p}}\|\mathbf{S}_{n_j,j}^{(q_{\ell_0},q_{\ell_0}+1,p)}\|_2\right)\rightarrow
0\;,
\end{equation}
\begin{equation}\label{e:supSigma3b}
n_j^{\frac{1-2d}{2}}\gamma_j^{-(\delta(q_{\ell_0})+\delta(q_{\ell_0}+1)+2K)}
\left(\sum_{\ell\in
I\setminus\{\ell_0\}}\sum_{p=0}^{q_\ell}\frac{c_{q_\ell}}{q_{\ell}!}\frac{c_{q_\ell+1}}{(q_{\ell}+1)!}
 p!{{q_{\ell}}\choose{p}}{{q_{\ell}+1}\choose{p}}\|\mathbf{S}_{n_j,j}^{(q_{\ell},q_{\ell}+1,p)}\|_2\right)\rightarrow
0\;.
\end{equation}
\end{proposition}
\begin{proof}
  Observe that, since $q_{\ell_0}\geq 1$, the assumption
  $q_{\ell_0}+1<1/(1-2d)$ implies that $d\in (1/4,1/2)$.

  We first prove Inequality~(\ref{e:supSigma3a}). Since there is only a finite
  number of terms in the left hand side of Inequality~(\ref{e:supSigma3a}), we
  only have to prove that each term tends to $0$. We apply
  Corollary~\ref{cor:SD} with $q=q_{\ell_0}$, $q'=q_{\ell_0}+1$ and $p\leq
  q_{\ell_0}-1$.  For these values of $q,q'$ and $p$, under
  Condition~(\ref{eq:condSogma3}), we have $\varepsilon(q')=0$, and
  by~(\ref{e:expnaqqpun}) and~(\ref{e:expj}), we have
$\alpha(q,q',p)\geq \min ({3}(1/2-d),1/2)$, $\beta(q,p)\leq
\delta_+(q_{\ell_0})=\delta(q_{\ell_0})$ and $\beta(q',p)\leq
\delta_+(q_{\ell_0}+1)=\delta(q_{\ell_0}+1)$.
Thus Equation~(\ref{e:boundL2P51ab}) yields
\[
n_j^{(1-2d)/2}\gamma_j^{-(\delta(q_{\ell_0})+\delta(q_{\ell_0}+1)+2K)}
\|S_{n_j,j}^{(q_{\ell_0},q_{\ell_0}+1,p)}\|_2=
O\left(n_j^{-\min(1-2d,d)}\log(n_j)\right)\;.
\]
Since $d\in (1/4,1/2)$, we obtain~(\ref{e:supSigma3a}).

We now prove~(\ref{e:supSigma3b}).  We apply Corollary~\ref{cor:SD} with
$q=q_{\ell}$, $q'=q_{\ell}+1$ and $p\leq q_{\ell}$ for some $\ell\in
I\setminus\{\ell_0\}$.  In this case Inequality~(\ref{e:boundL2P51ab}) reads
\begin{multline}\label{e:INEQboundSigma3b}
\gamma_j^{-(\delta(q_{\ell_0})+\delta(q_{\ell_0}+1)+2K)}
\|S_{n,j}^{(q_\ell,q_{\ell}+1,p)}\|_2\\
\leq C^{q_\ell
+\frac{1}{2}}\Lambda_2(q_{\ell}-p,p)^{\frac{1}{2}}\Lambda_2(q_{\ell}+1-p,p)^{\frac{1}{2}}
n^{-\alpha(q_\ell,q_{\ell}+1,p)}\log(n)^{\varepsilon(2q_{\ell}+1-2p)}\\
\times\gamma_j^{(\beta(q_\ell,p)-\delta(q_{\ell_0}))+(\beta(q_{\ell}+1,p)-\delta(q_{\ell_0}+1))}(\log\gamma_j)^3\;.
\end{multline}
We observe that for $n$ large enough,
\begin{equation}\label{e:boundnSigma3}
n^{-\alpha(q_\ell,q_{\ell}+1,p)}\log(n)^{\varepsilon(2q_{\ell}+1-2p)}\leq n^{-(1-2d)/2}\;.
\end{equation}
Indeed, on the one hand, if $p=q_{\ell}$, then
$\varepsilon(2q_{\ell}+1-2p)=\varepsilon(q_{\ell}+q_{\ell}+1-2q_{\ell})=\varepsilon(1)=0$
and $\alpha(q_{\ell},q_{\ell}+1,q_{\ell})\geq (1-2d)/2$~(\ref{e:expnb}). On the
other hand, if $p<q_\ell$, since $d>1/4$,~(\ref{e:expna}) implies that
$\alpha(q_{\ell},q_{\ell}+1,p)\geq 1-2d$.

In addition, by~(\ref{e:expj}) one has for any $p\leq q_{\ell}$,
$\beta(q_{\ell},p)\leq \delta_+(q_{\ell})$. Thus, for any
$\ell>\ell_0$ and any $p\leq q_{\ell}$,
\begin{multline}\label{e:boundgammajSigma3}
\gamma_j^{(\beta(q_\ell,p)-\delta(q_{\ell_0}))+(\beta(q_{\ell}+1,p)-\delta(q_{\ell_0}+1))}(\log\gamma_j)^3\leq
\gamma_j^{(\delta_+(q_{\ell})-\delta(q_{\ell_0}))+(\delta_+(q_{\ell}+1)-\delta(q_{\ell_0}+1))}(\log\gamma_j)^3\\
\leq \gamma_j^{(\delta_+(q_{\ell_0+1})-\delta(q_{\ell_0}))+(\delta_+(q_{\ell_0+1}+1)-\delta(q_{\ell_0}+1))}(\log\gamma_j)^3 =o(1)\;.
\end{multline}
As in the proof of Proposition~\ref{pro:intSigma0}, applying
Lemma~\ref{lem:c_q-cond} with Condition~(\ref{e:condcv}), we have
$$
\sum_{\ell\geq0}
\frac{C^{q_{\ell}+{1}/{2}}|c_{q_\ell}||c_{q_\ell+1}|}{q_{\ell}!(q_{\ell}+1)!}\sum_{p=0}^{q_\ell}
p!{{q_{\ell}}\choose{p}}{{q_{\ell}+1}\choose{p}}\Lambda_2(q_{\ell}-p,p)^{\frac{1}{2}}\Lambda_2(q_{\ell}+1-p,p)^{\frac{1}{2}}
< \infty \;.
$$
Applying this,~(\ref{e:INEQboundSigma3b}), ~(\ref{e:boundnSigma3})
and~(\ref{e:boundgammajSigma3}), we obtain~(\ref{e:supSigma3b}).
\end{proof}
The following result can now be established.
\begin{proposition}\label{pro:Sigma3}
Under the same assumptions as Proposition~\ref{pro:intSigma3}, we have as $j\to\infty$
\[
n_j^{(1-2d)/2}\gamma_j^{-(\delta(q_{\ell_0}+1)+\delta(q_{\ell_0})+2K)}
\mathbf{\Sigma}_{n_j,j}^{(3)}
\overset{(\mathcal{L})}{\rightarrow}2\frac{c_{q_{\ell_0}}c_{q_{\ell_0}+1}}{q_{\ell_0}!}[f^*(0)^{q_{\ell_0}+1/2}\K_{q_{\ell_0}}]Z_{1,d}(1)\;.
\]
\end{proposition}
\begin{proof}
We apply Proposition~\ref{pro:lrd} with $q=q'-1=q_{\ell_0}$ and
$p=q_{\ell_0}$. Indeed we have, under Condition~(\ref{eq:condSogma3}),
$0<q=q_{\ell_0}<q'=q_{\ell_0}+1<1/(1-2d)$ and
$q+q'-2p=q_{\ell_0}+q_{\ell_0}+1-2q_{\ell_0}=1<1/(1-2d)$. Thus
Condition~(\ref{e:idLT}) holds. We obtain that, as $j\to\infty$,
\begin{equation}\label{e:PrepTn2}
n_j^{(1-2d)/2}\gamma_j^{-(\delta(q_{\ell_0})+\delta(q_{\ell_0}+1)+2K)}
S_{n_j,j}^{(q_{\ell_0},q_{\ell_0}+1,q_{\ell_0})}
\overset{(\mathcal{L})}{\rightarrow}[f^*(0)^{q_{\ell_0}+1/2}\K_{q_{\ell_0}}]Z_{1,d}(1)\;.
\end{equation}
Using this limit, Proposition~\ref{pro:intSigma3} and the definition of
$\mathbf{\Sigma}_{n_j,j}^{(3)}$ in~(\ref{e:defSigma3}), we
conclude the proof.
\end{proof}
\section{Proofs of Theorems~\ref{th:LD1},~\ref{th:LD2} and
\ref{th:LD3}}\label{s:proofsmain}
\subsection{Proof of Theorem~\ref{th:LD1}}\label{s:proofmain1}
In the setting of Theorem~\ref{th:LD1}, one has $q_0\geq2$ and thus $c_1=0$ and
$q_0^*=q_0\geq 2$. Thus $\mathbf{\Sigma}_{n_j,j}^{(2)}$ and
$\mathbf{S}_{n_j,j}^{(1,1,0)}$, vanish in~(\ref{e:s3}) and the asymptotic
behavior of $\overline{S}_{n_j,j}$ results from
$\mathbf{\Sigma}_{n_j,j}^{(0)}+\mathbf{\Sigma}_{n_j,j}^{(1)}$ and
$\mathbf{\Sigma}_{n_j,j}^{(3)}$ given in Proposition~\ref{pro:Sigma0}
and~\ref{pro:Sigma3}, respectively. These propositions apply because we
assume~(\ref{eq:qell0plus1LRD}) and $M\geq K+\delta(q_0)$ in
Theorem~\ref{th:LD1}. Now the ratio of the convergence rates appearing in these
propositions reads
$$
n_j^{1/2-d}\gamma_{j}^{-(\delta(q_{\ell_{0}})+\delta(q_{\ell_{0}}+1)+K)}
=\left(n_j\gamma_j^{-\nu}\right)^{1/2-d}\;.
$$
Hence Case~\ref{it:thm1a} of Theorem~\ref{th:LD1} corresponds to
$$
\mathbf{\Sigma}_{n_j,j}^{(3)}=o_P\left(\mathbf{\Sigma}_{n_j,j}^{(0)}+\mathbf{\Sigma}_{n_j,j}^{(1)}\right)
$$
and Case~\ref{it:thm1b} to
$$
\mathbf{\Sigma}_{n_j,j}^{(0)}+\mathbf{\Sigma}_{n_j,j}^{(1)}
=o_P\left(\mathbf{\Sigma}_{n_j,j}^{(3)}\right)\;.
$$
The proof of Theorem~\ref{th:LD1} follows.\hfill$\Box$
\subsection{Proof of Theorems~\ref{th:LD2}
  and~\ref{th:LD3}}\label{s:proofmain23}

%%%START WITH THE CASE $\nu_3=\infty$

Here Condition~(\ref{eq:hyp-hermite-rank1}) holds, so that $q_0=1$,
$q_1<1/(1-2d)$ and $\ell_0=\infty$ (or equivalently $I$ is an empty set). In
particular $\mathbf{\Sigma}_{n_j,j}^{(3)}$ vanishes in~(\ref{e:s3}) and the
asymptotic behavior of $\overline{S}_{n_j,j}$ is obtained from those of
$\mathbf{S}_{n_j,j}^{(1,1,0)}$,
$\mathbf{\Sigma}_{n_j,j}^{(0)}+\mathbf{\Sigma}_{n_j,j}^{(1)}$ and
$\mathbf{\Sigma}_{n_j,j}^{(2)}$.  Since moreover $M>K+d$,
Proposition~\ref{pro:Snj110} applies.  Using the definition of $q_0^*$
in~(\ref{e:q0star}) we have $q_0^*=q_1$, and since $M>K+d\geq K+\delta(q_0^*)$
Propositions~\ref{pro:Sigma0} also applies.  Finally, observing that here $m_0$
defined in~(\ref{e:m0}) equals 1 and that $M>K+d$, Proposition~\ref{pro:Sigma2}
applies. Thus, using~(\ref{e:s3}), it only remains to compare the convergence
rates in these propositions.

\begin{figure}[h]
  \centering \setlength{\unitlength}{1.5mm}
\begin{picture}(100,50)(-40,-5)
\put(-32,0){\vector(1,0){80}}
\put(-30,-2){\vector(0,1){26}}
\put(35,-4){$\gamma_j$}
\put(-35,25){$n_j$}
%arc de gauche
\put(-7,30){$\gamma_j^{\nu_3}$}
\put(-7.5,22){$R$}
\put(-5,22){$H$}
\put(-30,25){\arc{50}{0.0}{1.57}}
%arc centre
\put(6,30){$\gamma_j^{\nu_2}$}
\put(-30,40){\arc{80}{0.35}{1.57}}
\put(3.5,22){$R$}
\put(6.5,22){$G$}
%arc droite
\put(22,30){$\gamma_j^{\nu_1}$}
\put(-30,65){\arc{130}{0.65}{1.57}}
\put(16,22){$H$}
\put(20,22){$G$}
\end{picture}
\caption{Pairwise comparisons of the rates of convergence of
$\mathbf{S}_{n_j,j}^{(1,1,0)}$ ($G$),
$\mathbf{\Sigma}_{n_j,j}^{(0)}+\mathbf{\Sigma}_{n_j,j}^{(1)}$ ($R$) and
$\mathbf{\Sigma}_{n_j,j}^{(2)}$ ($H$) in the plane $\gamma_j$ versus $n_j$.}
\label{fig1}
\end{figure}
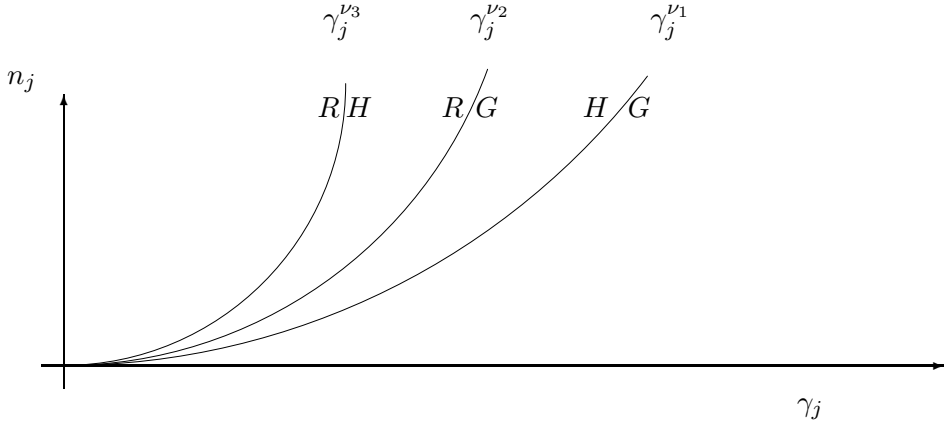

We first prove Theorem~\ref{th:LD2}. Recall that, by Lemma~\ref{lem:compexp},
since $q_1< q_1^*$, one has
\[
\nu_1 < \nu_2< \nu_3\;,
\]
where these three indices are defined in~(\ref{e:defbeta}).  In
Figure~\ref{fig1}, we provide pairwise comparisons of the rates of convergence
of $\mathbf{S}_{n_j,j}^{(1,1,0)}$,
$\mathbf{\Sigma}_{n_j,j}^{(0)}+\mathbf{\Sigma}_{n_j,j}^{(1)}$ and
$\mathbf{\Sigma}_{n_j,j}^{(2)}$. We obtain domains separated by the three
curves $n_j=\gamma_j^{\nu_1}$, $n_j=\gamma_j^{\nu_2}$ and
$n_j=\gamma_j^{\nu_3}$. Each curve is concerned with a pair of two terms among
the three and separates the plane $(\gamma_j,n_j)$ in two domains, where one of
the two terms dominates the other. We indicated the dominating term by $G$ for
the asymptotically Gaussian term $S_{n_j,j}^{(1,1,0)}$, $R$ for the
asymptotically Rosenblatt term
$\mathbf{\Sigma}_{n_j,j}^{(0)}+\mathbf{\Sigma}_{n_j,j}^{(1)}$ and $H$ for the
term $\mathbf{\Sigma}_{n_j,j}^{(3)}$ belonging asymptotically to a chaos of
order greater than $2$.

We begin with the case $q_1=3$. In this case, one has $\nu_3=\infty$. Further Propositions~\ref{pro:Sigma0} and \ref{pro:Sigma2} imply that $\Sigma_{n_j,j}^{(0)}+\Sigma_{n_j,j}^{(1)}=o_P(\Sigma_{n_j,j}^{(2)})$. One has then to compare the rates of convergence of $\mathbf{S}_{n_j,j}^{(1,1,0)}$ and $\Sigma_{n_j,j}^{(2)}$. Using the diagram, we then deduce that
\begin{enumerate}[label=$\bullet$]
\item  if $n_j\ll\gamma_j^{\nu_1}$, $G$ dominates $H$ and then we obtain Case~\ref{item:LD23} of
  Theorem~\ref{th:LD2} for $q_1=3$.
\item if $\gamma_j^{\nu_1}\ll n_j$, $H$ dominates $G$ and then we obtain Case~\ref{item:LD22} of
  Theorem~\ref{th:LD2} for $q_1=3$.
\end{enumerate}

If $q_1>3$, one has $\nu_3<\infty$ and the term $\mathbf{\Sigma}_{n_j,j}^{(0)}+\mathbf{\Sigma}_{n_j,j}^{(1)}$ is no more always negligible with respect to $\Sigma_{n_j,j}^{(2)}$. We then get three domains where one term dominates over the
other two:
\begin{enumerate}[label=$\bullet$]
\item $n_j\ll \gamma_j^{\nu_1}$: $G$ dominates $H$ and $R$, that is,
the two terms $\mathbf{\Sigma}_{n_j,j}^{(0)}+\mathbf{\Sigma}_{n_j,j}^{(1)}$ and
  $\mathbf{\Sigma}_{n_j,j}^{(2)}$ are both negligible with respect to
  $S_{n_j,j}^{(1,1,0)}$. By
  Proposition~\ref{pro:Snj110}, we obtain Case~\ref{item:LD23} of
  Theorem~\ref{th:LD2}.
\item $\gamma_j^{\nu_1}\ll n_j\ll \gamma_j^{\nu_3}$: since the domain lies both
  on the right-hand side of the curve $n_j=\gamma_j^{\nu_3}$ and on the
  left-hand side of the curve $n_j=\gamma_j^{\nu_1}$, $H$ dominates $R$ and $H$
  dominates $G$, hence $R$ dominates $R$ and $G$. That is, the two terms
  $S_{n_j,j}^{(1,1,0)}$ and
  $\mathbf{\Sigma}_{n_j,j}^{(0)}+\mathbf{\Sigma}_{n_j,j}^{(1)}$ are both
  negligible with respect to $\mathbf{\Sigma}_{n_j,j}^{(2)}$. By
  Proposition~\ref{pro:Sigma2}, we obtain Case~\ref{item:LD22} of
  Theorem~\ref{th:LD2}.
  \item $\gamma_j^{\nu_3}\ll n_j$: since the domain lies both on the left-hand
  side of the curve $n_j=\gamma_j^{\nu_3}$ and on the left-hand side of the
  curve $n_j=\gamma_j^{\nu_2}$, $R$ dominates $H$ and $R$ dominates $G$, hence
  $R$ dominates $H$ and $G$. That is, the two terms $S_{n_j,j}^{(1,1,0)}$ and
  $\mathbf{\Sigma}_{n_j,j}^{(2)}$ are both negligible with respect to
  $\mathbf{\Sigma}_{n_j,j}^{(0)}+\mathbf{\Sigma}_{n_j,j}^{(1)}$. By
  Proposition~\ref{pro:Sigma0}, we obtain Case~\ref{item:LD21} of
  Theorem~\ref{th:LD2}.
\end{enumerate}
This completes the proof of Theorem~\ref{th:LD2}.

The proof of Theorem~\ref{th:LD3} is similar except that the
assumption $q_1\geq q_1^*$ implies that
\[
\nu_3\leq \nu_2\leq \nu_1\;.
\]
The domains of convergence are now obtained from Figure~\ref{fig2}. \hfill$\Box$
\begin{figure}[h]
  \centering \setlength{\unitlength}{1.5mm}
\begin{picture}(100,50)(-40,-5)
\put(-32,0){\vector(1,0){80}}
\put(-30,-2){\vector(0,1){26}}
\put(35,-4){$\gamma_j$}
\put(-35,25){$n_j$}
%arc de gauche
\put(-7,30){$\gamma_j^{\nu_1}$}
\put(-7.7,22){$H$}
\put(-5,22){$G$}
\put(-30,25){\arc{50}{0.0}{1.57}}
%arc centre
\put(6,30){$\gamma_j^{\nu_2}$}
\put(-30,40){\arc{80}{0.35}{1.57}}
\put(3.5,22){$R$}
\put(6.5,22){$G$}
%arc droite
\put(22,30){$\gamma_j^{\nu_3}$}
\put(-30,65){\arc{130}{0.65}{1.57}}
\put(16,22){$R$}
\put(20,22){$H$}
\end{picture}
\caption{Domains of convergence for Theorem~\ref{th:LD3}}
\label{fig2}
\end{figure}
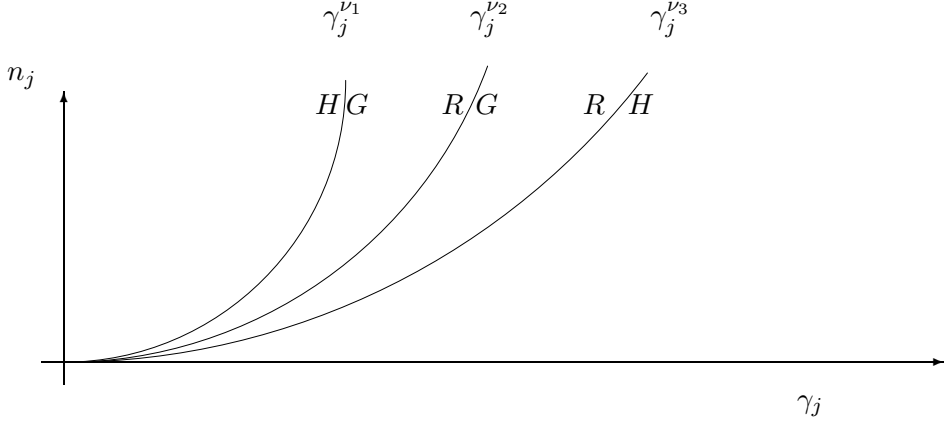
\section{Technical lemmas}\label{sec:techlemma}
The following lemma is used in the proof of Proposition~\ref{pro:SD} and in
that of Lemma~\ref{lem:suphjkp}.
\begin{lemma}\label{lem:Ja}
 Define, for all $a>0$ and $\beta_1\in(0,1)$,
\begin{equation}\label{e:J1a}
J_{1,a}(s_1;\beta_1)=|s_1|^{-\beta_1},\quad s_1\in\mathbb{R}\;,
\end{equation}
and, for any integer $p\geq2$ and
$\beta=(\beta_1,\cdots,\beta_p)\in
  (0,1)^p$,
\begin{equation}\label{e:Jma}
J_{p,a}(s_1;\beta)=\int_{s_2=-(p-1)a}^{(p-1)a}\dots\int_{s_p=-a}^{a}
\prod_{i=2}^p|s_{i-1}-s_i|^{-\beta_{i-1}}\; |s_p|^{-\beta_p}
\;\rmd s_p\dots \rmd s_2,\quad s_1\in\mathbb{R}\;.
\end{equation}
Then
\begin{enumerate}[label=(\roman*)]
\item if $\beta_1+\dots+\beta_p> p-1$, we have
$$
C_p(\beta):= \sup_{a>0}\;\sup_{|s_1|\leq p a}\left(
|s_1|^{-(p-1-(\beta_1+\dots+\beta_p))} J_{p,a}(s_1;\beta) \right)
< \infty \;,
$$
\item if $\beta_1+\dots+\beta_p=p-1$, we have
$$
C_p(\beta):= \sup_{a>0}\;\sup_{|s_1|\leq p a}\left(
\frac{1}{1+\log(pa/|s_1|)} J_{p,a}(s_1;\beta) \right)< \infty
\;,
$$
\item if there exists $q\in\{2,\dots,p-1\}$ such that
  $\beta_q+\dots+\beta_p=p-q$, we have
$$
C_p(\beta):= \sup_{a>0}\;\sup_{|s_1|\leq p a} \left(
  \frac1{1+\log(pa/|s_1|)}a^{-(q-1-(\beta_{1}+\dots+\beta_{q-1}))}J_{p,a}(s_1;\beta)  \right)< \infty \;,
$$
\item if  $\beta_1+\dots+\beta_p<p-1$ and for all
$q\in\{1,\dots,p-1\}$, we
  have  $\beta_q+\dots+\beta_p\neq p-q$, we have
$$
C_p(\beta):= \sup_{a>0}\;\sup_{|s_1|\leq p a}
\left(a^{-(p-1-(\beta_1+\dots+\beta_p))}J_{p,a}(s_1;\beta)
\right)< \infty \;.
$$
\end{enumerate}

Moreover, in the case where all the components of $\beta$ are
equal to $b\in (0,1)$, there exists a constant $c>0$ depending
only on $b$ such that
\begin{equation}
  \label{eq:uniform-CJbound}
\sup_{p\geq1}c^{-p}(p!)^{b-1}\;C_p(b1_p) <\infty \;,
\end{equation}
where $1_p$ denotes the $p$--dimensional vector with all entries
equal to $1$.
\end{lemma}
\begin{remark}
  As in~\cite{clausel-roueff-taqqu-tudor-2011b}, all the cases can be compactly
  written as
$$
C_p(\beta)=\sup_{a>0} \sup_{|s_1|\leq pa
}\left(\frac{a^{-(p-1-(\beta_1+\cdots+\beta_p))_+}
|s_1|^{(p-1-(\beta_1+\cdots+\beta_p))_-}}
{(1+\log(pa/|s_1|))^{\varepsilon}}J_{p,a}(s_1;\beta)\right)
$$
where $\varepsilon=1$ if there exists $q\in\{1,\cdots,p\}$ such
that $\beta_q+\cdots+\beta_p=p-q$ and $\varepsilon=0$ otherwise,
and $x_+=\max(x,0)$, $x_-=\max(-x,0)$.
Now, observing that
\[
(p-1-2 p d)_+=(p(1-2d)-1)_+=(-2\delta(p))_+=2\delta_-(p)\;,
\]
and, similarly, $(p-1-2 p d)_-=2\delta_+(p)$,
Inequality~(\ref{eq:uniform-CJbound}) with $b=2d\in(0,1)$ implies
there exists a constant $c>0$
depending only on $d$ such that for any $a>0$, $|s_1|\leq p a$
\begin{equation}\label{e:Jmaunified}
J_{p,a}(s_1;2d 1_p)\leq c^{p}(p!)^{1-2d}
a^{2\delta_-(p)}|s_1|^{-2\delta_+(p)}
(1+\log(pa/|s_1|))^{\varepsilon(p)}\;,
\end{equation}
where $\varepsilon(p)$ is here defined by~(\ref{e:varepsilontilde}), which
corresponds to the $\varepsilon$ above in the case $\beta_1=\dots=\beta_p=2d$.
\end{remark}
\begin{proof}
Observe first that for all $p\geq1$,
  \begin{equation}
    \label{eq:inductionJ}
J_{p,a}(s_1;\beta)=\int_{s_2=-(p-1)a}^{(p-1)a}|s_2-s_1|^{-\beta_1}\;
J_{p-1,a}(s_2;\beta')\;\rmd s_2\;,
  \end{equation}
  where $\beta'=(\beta_2,\dots,\beta_{p})$. The finiteness of the bounds
  $C_p(\beta)$ for any integer $p$ and any $\beta\in (0,1)^p$ is then proved by
  induction on $p$ in the different cases in Lemma~9.3
  of~\cite{clausel-roueff-taqqu-tudor-2011b}.

Finally we show the uniform bound~(\ref{eq:uniform-CJbound}), that
is, that $C_p(b,\dots,b)=O(c_1^p(p!)^{1-b})$ as $p\to\infty$ for
any fixed $b\in(0,1)$.  We provide a proof only in the case where
$1/(1-b)$ is not an integer (to avoid cases (ii) and (iii)). The
proof is similar in the other case.  Hence we use the induction
step described in Case 1 above. Observe that there exists some
integer $p_0$ depending only on $b$, such that for any $p\geq p_0$
we have $(p-1)b<p-2$, which corresponds above to
$\beta_2+\cdots+\beta_p<p-2$ (case (iv)). Hence using the
induction assumption~(\ref{eq:inductionJ}), the finiteness of
$C_p$ in case (iv) and the fact that $|s_1|\leq p a$, we get that
there exists some positive constant $c$ depending only on $b$ such
that,
\begin{eqnarray*}
J_{p,a}(s_1;b,\cdots,b)&\leq& C_{p-1}(b,\cdots,b)\;a^{p-2-(p-1)b}
\left(\int_{-(p-1)a}^{(p-1)a}|s_2-s_1|^{-b}\rmd s_2\right)\\
&\leq& C_{p-1}(b,\cdots,b)\;a^{p-2-(p-1)b}\times \left(c(b)
((2p-1)a)^{1-b}\right)\\
&=& \left(c(b)p^{1-b} \; C_{p-1}(b,\dots,b)\right)\;a^{p-1-pb}\;.
\end{eqnarray*}
This yields that for any $p\geq p_0(b)$, $C_{p}(b,\dots,b)\leq
c(b)p^{1-b} \; C_{p-1}(b,\dots,b)$. Since this holds for any
$p\geq p_0(b)$, the bound~(\ref{eq:uniform-CJbound}) follows by
induction.
\end{proof}

The following lemma provides bounds of $\alpha$ and $\beta$
defined in~(\ref{e:alpha}) and~(\ref{e:beta}). It is used in the
proofs of Propositions~\ref{pro:intSigma0},~\ref{pro:intSigma2}
and ~\ref{pro:intSigma3}.
\begin{lemma}\label{lem:exp}
One has
\begin{enumerate}
\item Assume that $d>1/4$. Then for any $(q,q')\in\mathbb{N}^2$
\begin{equation}\label{e:expna}
\inf_{0\leq p\leq \min(q\vee q'-2,q\wedge
q')}\left(\alpha(q,q',p)\right)\geq 1-2d\;,
\end{equation}
In any case,
\begin{equation}\label{e:expnb}
\alpha(q,q',\min(q\vee q'-1,q\wedge q'))\geq 1/2-d\;.
\end{equation}
\item For any $q\in\mathbb{N}$
\begin{equation}\label{e:expnaqq}
\inf_{0\leq p\leq q-2}\left(\alpha(q,q,p)\right)\geq
\min(2(1-2d),1/2)\;.
\end{equation}
Further,
\begin{equation}\label{e:expnaqqbis}
\alpha(q,q,q-1) = \min(1-2d,1/2)\;.
\end{equation}
\item For any $q\in\mathbb{N}$
\begin{equation}\label{e:expnaqqpun}
\inf_{0\leq p\leq q-1}\left(\alpha(q+1,q,p)\right)\geq
\min(3/2(1-2d),1/2)\;.
\end{equation}
\item For any $q\in\mathbb{N}$
\begin{equation}\label{e:expj}
\sup_{0\leq p\leq q}\left(\beta(q,p)\right)\leq \delta_+(q)\;.
\end{equation}
\end{enumerate}
\end{lemma}
\begin{proof}
\begin{enumerate}
\item Let us fix $(q,q')\in\mathbb{N}^2$ and assume that $q'\leq
q$. Since the map
\[
m\mapsto \delta_+(m)=\max(dm-(m-1)/2,0)\;,
\]
is non--increasing
with range in $[0,1/2]$, one has for $0\leq p\leq \min(q-2,q')$
\[
\alpha(q,q',p)=\min(1-\delta_+(q-p)-\delta_+(q'-p),1/2)\geq
\min(1-\delta_+(2)-1/2,1/2)\;.
\]
If $d>1/4$, $\delta_+(2)= 2d-1/2$ and thus
\[
\alpha(q,q',p)\geq \min(1-2d,1/2)=1-2d\;,
\]
which proves~(\ref{e:expna}).
Finally, if $p= q-1$ and $p\leq q'$,
\begin{align*}
  \alpha(q,q',p)=\min(1-\delta_+(q-p)-\delta_+(q'-p),1/2) &\geq
\min(1-\delta_+(1)-1/2,1/2)\\
&= \min(1/2-d,1/2)=1/2-d\;,
\end{align*}
which proves~(\ref{e:expnb}).
\item Let us fix $q\in\mathbb{N}$, then for any $p\leq q-2$,
\[
\alpha(q,q,p)=\min(1-\delta_+(q-p)-\delta_+(q-p),1/2)\geq
\min(1-2\delta_+(2),1/2)\;.
\]
If $d\leq 1/4$, $\delta_+(2)=0$ and we get $\alpha(q,q,p)\geq
1/2\geq \min(2(1-2d),1/2)$. If $d>1/4$,
$2\delta_+(2)=2\delta(2)=4d-1$ and
\[
\alpha(q,q,p)\geq \min(1-(4d-1),1/2)=\min(2(1-2d),1/2)\;,
\]
which gives~(\ref{e:expnaqq}). To prove~(\ref{e:expnaqqbis}), we observe that if $p=q-1$,
\[
\alpha(q,q,p)=\min(1-2\delta_+(1),1/2)= \min(1-2d,1/2)\;.
\]
\item Let us fix $q\in\mathbb{N}$, then for any $p\leq q-1$,
\begin{align*}
\alpha(q+1,q,p)=\min(1-\delta_+(q+1-p)-\delta_+(q-p),1/2)&\geq
\min(1-\delta_+(2)-\delta_+(1),1/2)\\
&=\min(1-d-\delta_+(2),1/2)\;.
\end{align*}
If $d\leq 1/4$, $\delta_+(2)=0$ and
$\alpha(q+1,q,p)\geq\min(1-d,1/2)= 1/2$. If $d>1/4$,
$\delta_+(2)=2d-1/2$ and~(\ref{e:expnaqqpun}) follows from
\[
\alpha(q+1,q,p)\geq \min(1-d-(2d-1/2),1/2)=\min(3(1-2d)/2,1/2)\;.
\]
\item If $\beta(q,p)=0$, then $\beta(q,p)\leq \delta_+(q)$. Now
consider the case where
\[
\beta(q,p)=\max(\delta_+(p)+\delta_+(q-p)-1/2,0)>0\;,
\]
that is, $\delta_+(p)+\delta_+(q-p)-1/2>0$. In this case, $\delta_+(p)$ and
$\delta_+(q-p)$ are both positive (since
$0\leq\delta_+(\cdot)<1/2$) and they respectively equal $\delta(p)$
and $\delta(q-p)$. Then we obtain
\[
\max(\delta_+(p)+\delta_+(q-p)-1/2,0)=\delta(p)+\delta(q-p)-1/2=
\delta(q)\;,
\]
which again implies~(\ref{e:expj}).
\end{enumerate}
\end{proof}
The following result provides a bound of $\widehat{\kappa}_{j}^{(p)}$ defined
in~(\ref{e:intrepKjp1}), in the case where $p>0$. It is a refinement of
Lemma~10.1 of \cite{clausel-roueff-taqqu-tudor-2011b}. It is used in the proof
of Proposition~\ref{pro:lrd}.

\begin{lemma}\label{lem:suphjkp}
  Suppose that Assumptions A hold and let $p$ be a positive integer.  Then
  there exists some $C>0$ neither depending on $p$ nor $j$ such that for any
  $(\xi_1,\xi_2)\in \mathbb{R}^2$,
\begin{enumerate}[label=(\roman*)]
\item\label{item:lem:suphjkp1} if for any $s\in\{1,\cdots,p\}$, $s(1-2d)\neq 1$ then,
\begin{equation}\label{e:majohjkp}
|\widehat{\kappa}_{j}^{(p)}(\xi_1,\xi_2)| \leq C^p (p!)^{1-2d}\;
\frac{\gamma_j^{2(\delta_+(p)+K)}}{(1+\gamma_j|\{\xi_1\}|)^{\delta_+(p)}
(1+\gamma_j|\{\xi_2\}|)^{\delta_+(p)}}\;.
\end{equation}
\item\label{item:lem:suphjkp2}  if there exists $s\in\{1,\cdots,p\}$ such that
$s(1-2d)=1$, then,
\begin{equation}\label{e:majohjkp2}
|\widehat{\kappa}_{j}^{(p)}(\xi_1,\xi_2)| \leq C^p
(p!)^{1-2d}\;\gamma_j^{2K}\log(\gamma_j)\;.
\end{equation}
\end{enumerate}
\end{lemma}
\begin{remark}
  In Case~\ref{item:lem:suphjkp2} of Lemma~\ref{lem:suphjkp}, we have
  $p>1/(1-2d)$, hence $\delta_+(p)=0$. Equations~(\ref{e:majohjkp2})
  and~(\ref{e:majohjkp}) can thus be written as a single bound, namely,
  \begin{equation}
    \label{eq:majo-kappa-generale}
|\widehat{\kappa}_{j}^{(p)}(\xi_1,\xi_2)| \leq C^p (p!)^{1-2d}\;
\frac{\gamma_j^{2(\delta_+(p)+K)}}{(1+\gamma_j|\{\xi_1\}|)^{\delta_+(p)}
(1+\gamma_j|\{\xi_2\}|)^{\delta_+(p)}}\;(\log\gamma_j)^{\varepsilon(p)}\;,
  \end{equation}
where $\varepsilon(p)$ is defined by~(\ref{e:varepsilontilde}).
\end{remark}
\begin{proof}
  By $(2\pi)$-periodicity of $\widehat{\kappa}_{j}^{(p)}(\xi_1,\xi_2)$ along
  both variables $\xi_1$ and $\xi_2$, we may take $\xi_1,\xi_2\in[-\pi,\pi]$.
  The remainder of the proof shows that~(\ref{eq:majo-kappa-generale}) holds
  for such $(\xi_1,\xi_2)$.

Note that by assumption,
$$
f(\lambda)\leq C \;|\lambda|^{-2d}\;,
$$
where $C>0$ only depends on
$f^*$. Using~(\ref{eq:M_replaced_by_K}),~(\ref{e:intrepKjp1}) and~(\ref{e:Jma})
with
$$
\mu_i=\gamma_j(\lambda_i+\dots+\lambda_p)\;,
$$
we get
$$
|\widehat{\kappa}_j^{(p)}(\xi_1,-\xi_2)| \leq C^p
\gamma_j^{2(K+\delta(p))}\int_{-p\gamma_j\pi}^{p\gamma_j\pi}\frac{J_{p,\gamma_j\pi}(\mu_1;2d\,1_p)\rmd \mu_1}{\prod_{i=1}^2\left(1+\gamma_j\left|\left\{{\mu_1}/{\gamma_j}+\xi_i\right\}\right|\right)^{K+\alpha}}\;.
$$
Then, by~(\ref{e:Jmaunified}), there exists $C>0$ not depending on $j,p$ such
that, for all $(\xi_1,\xi_2)\in [-\pi,\pi]^2$,
\[
|\widehat{\kappa}_j^{(p)}(\xi_1,-\xi_2)| \leq C^p (p!)^{1-2d}
\gamma_j^{2K+2(\delta(p)+\delta_-(p))}\int_{-p\gamma_j
\pi}^{p\gamma_j \pi}\frac
{|\mu_1|^{-2\delta_+(p)} (1+\log(p\gamma_j\pi/|\mu_1|))^{\varepsilon(p)}\rmd \mu_1}
{\prod_{i=1}^2\left(1+\gamma_j\left|\left\{{\mu_1}/{\gamma_j}+\xi_i\right\}\right|\right)^{K+\alpha}}\;.
\]
Using that $\delta(p)=\delta_+(p)-\delta_-(p)$ and the
Cauchy--Schwarz inequality, to obtain~(\ref{eq:majo-kappa-generale}), it is
sufficient to show that, for all  $\xi\in(-\pi,\pi]$,
\begin{equation}\label{e:boundK1}
  \int_{-p\gamma_j
    \pi}^{p\gamma_j \pi}
  \frac{|\mu_1|^{-2\delta_+(p)}(1+\log(p\gamma_j\pi/|\mu_1|))^{\varepsilon(p)}\rmd \mu_1}
  {(1+\left|\gamma_j\left\{{\mu_1}/{\gamma_j}+\xi\right\}\right|)^{2(K+\alpha)}}\leq
  C\; p\log p\; (1+\gamma_j|\xi|)^{-2\delta_+(p)}\, (\log \gamma_j)^{\varepsilon(p)} \;,
\end{equation}
where $C$ is a positive constant.

If $\delta(p)>0$ the rest of the proof is similar to that of Lemma~10.1 in
\cite{clausel-roueff-taqqu-tudor-2011b} and is thus omitted.

We now take $\delta_+(p)=0$, so that~(\ref{e:boundK1}) becomes
\begin{equation}\label{e:boundK1bis}
\int_{-p\gamma_j
\pi}^{p\gamma_j \pi}
\frac{(1+\log(p\gamma_j\pi/|\mu_1|))^{\varepsilon(p)}\rmd \mu_1}
{(1+\left|\gamma_j\left\{{\mu_1}/{\gamma_j}+\xi\right\}\right|)^{2(K+\alpha)}}\leq
C\; p\log p\;(\log \gamma_j)^{\varepsilon(p)} \;,
\end{equation}
The denominator in the integral
is a $(2\pi\gamma_j)$-periodic function of $\mu_1$, hence
the integral over
$[-p\gamma_j\pi,p\gamma_j\pi]$ is bounded by the sum of at most  $p+1$ integral
of the form
$$
I(-\gamma_j\xi+2 k\gamma_j \pi)
\quad\text{with}\quad I(y)=
\int_{A(y)} \frac{(1+\log(p\gamma_j\pi/|\mu_1|))^{\varepsilon(p)}\rmd \mu_1}
{(1+\left|\mu_1-y\right|)^{2(K+\alpha)}}\;,
$$
where $k\in\mathbb{Z}$ and $A(y)=[-p\gamma_j\pi,p\gamma_j\pi]\cap
(y-\gamma_j\pi,y+\gamma_j\pi]$. We observe that $I(y)$ is maximal at $y=0$
where it takes value
$$
I(0)=\int_{-\gamma_j\pi}^{\gamma_j\pi}
\frac{(1+\log(p\gamma_j\pi/|\mu_1|))^{\varepsilon(p)}\rmd \mu_1}
{(1+\left|\mu_1\right|)^{2(K+\alpha)}}\leq
(1+\log(p\gamma_j\pi))^{\varepsilon(p)}
\int_{-\infty}^{\infty}
\frac{(1+|\log(|\mu_1|)|)^{\varepsilon(p)}\rmd \mu_1}
{(1+\left|\mu_1\right|)^{2(K+\alpha)}} \;.
$$
Since the last integral in the previous display is finite for
$\varepsilon(p)=0,1$, we finally obtain~(\ref{e:boundK1bis}).
\end{proof}
The following Lemma will be used when identifying the leading terms of the three sums $\Sigma_{n,j}^{(0)}+\Sigma_{n,j}^{(1)}$, $\Sigma_{n,j}^{(2)}$ and $\Sigma_{n,j}^{(3)}$~:
\begin{lemma}\label{lem:c_q-cond}
Condition~(\ref{e:condcv}) implies that for any $C>0$,
\begin{equation}
  \label{eq:constants-q0is1}
\sum_{q,q'\geq 1}\sum_{p=0}^{q\wedge q'}\frac{|c_q|}{q!}
\frac{|c_{q'}|}{q'!} \; p!\;
{{q}\choose{p}}{{q'}\choose{p}}C^{\frac{q+q'}{2}}\Lambda_2(q-p,p)^{\frac{1}{2}}\Lambda_2(q'-p,p)^{\frac{1}{2}}<\infty\;,
\end{equation}
where $\Lambda_2$ is defined by~(\ref{e:KC}).
\end{lemma}
\begin{proof}
Let $C>0$. By definition of $\Lambda_2$ in~(\ref{e:KC}), we have
\begin{align}\nonumber
\sum_{(q,q')}\sum_{p=0}^{q\wedge q'}\frac{|c_q|}{q!}&
\frac{|c_{q'}|}{q'!} \; p!\;
{{q}\choose{p}}{{q'}\choose{p}}C^{\frac{q+q'}{2}}\Lambda_2(q-p,p)^{1/2}\Lambda_2(q'-p,p)^{1/2}\\
\nonumber
&=
\sum_{p\geq0}\sum_{q,q'\geq p\vee1}|c_q
c_{q'}|\;C^{\frac{q+q'}{2}}\;[p!]^{-2d}\;[(q-p)!(q'-p)!]^{-1/2-d}\\
\nonumber
&=\sum_{p\geq0}[p!]^{-2d}\left(\sum_{q\geq p\vee1}|c_q|\;C^{q/2}\;[(q-p)!]^{-1/2-d}\right)^2\\
\label{eq:1}
&\leq\left(\sum_{q\geq p\geq0}[p!]^{-d}|c_q|\;C^{q/2}\;[(q-p)!]^{-1/2-d}\right)^2
\end{align}
Using that $x\mapsto x^{d}$ is concave, we have, for any $q\geq0$,
$$
\sum_{p=0}^{q}(q!)^{d}[p!\,(q-p)!]^{-d}=q\cdot\left(\frac{1}{q}\sum_{p=0}^{q}{{q}\choose{p}}^{d}\right)
\leq q\left(\frac{1}{q}\sum_{p=0}^{q}{{q}\choose{p}}\right)^{d}
=q^{1-d}2^{dq}\;.
$$
We deduce that, for any $q\geq0$,
$$
\sum_{p=0}^q [p!]^{-d}\;[(q-p)!]^{-1/2-d}\leq \sum_{p=0}^q
[p!\,(q-p)!]^{-d}\leq (q!)^{-d}\; q^{1-d}2^{dq} \;.
$$
Using this to bound the sum in $p$ in~(\ref{eq:1}) and then
Condition~(\ref{e:condcv}) , we get~(\ref{eq:constants-q0is1}), which concludes
the proof of Lemma~\ref{lem:c_q-cond}.
\end{proof}
\appendix
\section{Proof of Proposition~\ref{pro:SD}}\label{sec:postponed-proof}

As in the proof of Proposition~\ref{pro:decompSnjLD}, we can take $m=1$ without
loss of generality. In what follows, $C_1,C_2,\cdots$ denote positive constants
which do not depend on $n$, $j$, $q,q',p$. The
following function $\varepsilon'$ defined on $\mathbb{R}_+$ is used in the
sequel,
\begin{equation}\label{e:varepsilonp}
\varepsilon'(a)=\1_{\{1\}}(a)=
\begin{cases}
\varepsilon'(a)=1& \mbox{ if }a=1\;,\\
\varepsilon'(a)=0 & \mbox{ otherwise.}
\end{cases}
\end{equation}
% We first observe that if $p=0$, the bound~(\ref{e:boundL2P51a}) is a
% consequence of~(\ref{e:boundL2P51b}). This follows by observing that
% $\Lambda_2(q,0)=\Lambda_1(q)$, $\alpha(q,q',0)=1/2$, $\beta(q,0)=\delta_+(q)$
% and that the log exponents in~(\ref{e:boundL2P51a}) are all larger than that
% in~(\ref{e:boundL2P51b}). Hence, to prove the result, we show
% that~\ref{item:pro:SD1i} holds for $p>0$ and then
We shall prove ~\ref{item:pro:SD1i} and~\ref{item:pro:SD1ii},
 successively.

\noindent{\it Proof of~\ref{item:pro:SD1i}.}
Set $r=q-p$ and
$r'=q'-p$. The starting point of the proof is the integral expression of
$S_{n,j}^{(q,q',p)}$ given by~(\ref{e:Snjrrp}). Thereafter we follow the same
approach as in the proof of Proposition~7.1
of~\cite{clausel-roueff-taqqu-tudor-2011b}, using Lemma~\ref{lem:suphjkp} to
bound the kernel $\widehat{\kappa}_j^{(p)}$ involved in the integral expression
of $S_{n,j}^{(q,q',p)}$ instead of Lemma~10.1
of~\cite{clausel-roueff-taqqu-tudor-2011b}, replacing $2r$, $(r,r)$,
$\delta(p)$ with $r+r'$, $(r,r')$,
$\delta_+(p)$ and adding if necessary a logarithmic correction.

We obtain the following inequality, similar to~(7.2) and~(7.3) in~\cite{clausel-roueff-taqqu-tudor-2011b},
\begin{equation}\label{e:boundSnj51}
\mathbb{E}\left[\left|S_{n,j}^{(q,q',p)}\right|^2\right]\leq C_1^{2p}
(p!)^{2(1-2d)}\gamma_j^{-2+2\delta(r)+2\delta(r')+4\delta_+(p)}\gamma_j^{4K}(\log\gamma_j)^{2\varepsilon(p)}I_{n,j}\;,
\end{equation}
where, for any $j,n$,
\[
I_{n,j}= \int_{-\gamma_j\pi r}^{\gamma_j\pi
r}\int_{-\gamma_j\pi r'}^{\gamma_j\pi
r'}\frac{J_{r,\gamma_j\pi}(u_1;2d 1_r)J_{r',\gamma_j\pi}(v_1;2d 1_{r'})\rmd
u_1\rmd v_1} {(1+n\left|\{u_1+v_1\}\right|)^{2}
\left(1+\gamma_j\left|\{\frac{u_1}{\gamma_j}\}\right|\right)^{2\delta_{+}(p)}
\left(1+\gamma_j\left|\{\frac{v_1}{\gamma_j}\}\right|\right)^{2\delta_{+}(p)}}\;,
\]
and where $J_{r,\gamma_j\pi}(u_1;2d 1_r)$ and $J_{r',\gamma_j\pi}(v_1;2d
1_{r'})$ are defined in Lemma~\ref{lem:Ja}.

We now use~(\ref{e:Jmaunified}) of Lemma~\ref{lem:Ja} successively with $p=r$, $a=\gamma_j\pi$, $s_1=u_1$
and $p=r'$, $a=\gamma_j\pi$, $s_1=v_1$. We get that
\begin{eqnarray*}
I_{n,j}&\leq&
C_2^{r+r'}(r!r'!)^{1-2d}\gamma_j^{2\delta_-(r)+2\delta_-(r')}
(\log\gamma_j)^{2\varepsilon'_0}\\
&&\times\int_{\mathbb{R}^2}\frac{\1_{(-\pi r,\pi
r)}(\frac{u_1}{\gamma_j})\1_{(-\pi r',\pi
r')}(\frac{v_1}{\gamma_j})|u_1|^{-2\delta_+(r)}|v_1|^{-2\delta_+(r')}\rmd
u_1\;\rmd v_1}
{(1+n\left|\{u_1+v_1\}\right|)^2\left(1+\gamma_j\left|\{\frac{u_1}{\gamma_j}\}\right|\right)^{2\delta_{+}(p)}
\left(1+\gamma_j\left|\{\frac{v_1}{\gamma_j}\}\right|\right)^{2\delta_{+}(p)}}\;,
\end{eqnarray*}
where
$\varepsilon'_0=\frac12[\varepsilon(r)+\varepsilon(r')+2\varepsilon(p)]$.

The next step relies on the inequality $|\{x\}|\leq |x|$ on $\mathbb{R}$ and on
the $2\pi$--periodicity of $x\mapsto \{x\}$. We then get that
\begin{equation}\label{e:eq1}
I_{n,j}\leq
C_3^{r+r'}(r!r'!)^{1-2d}\gamma_j^{2\delta_{-}(r)+2\delta_{-}(r')}(\log\gamma_j)^{2\varepsilon'_0} \tilde{I}_{n,j}\;,
\end{equation}
with
\begin{align*}
\tilde{I}_{n,j}&=\int_{(-\gamma_j\pi,\gamma_j\pi)^2}
\frac{|u_1|^{-2\delta_+(r)}|v_1|^{-2\delta_+(r')}\rmd u_1\;\rmd
v_1} {(1+n\left|\{u_1+v_1\}\right|)^2(1+|u_1|)^{2\delta_{+}(p)}
(1+|v_1|)^{2\delta_{+}(p)}}\;.
\end{align*}

The bound of $\tilde{I}_{n,j}$ is obtained using the decomposition
 $\tilde{I}_{n,j}=A+2B$ with
\[
A=\int_{\Delta_{j}^{(0)}
}\frac{|u_1|^{-2\delta_+(r)}|v_1|^{-2\delta_+(r')}\rmd u_1\;\rmd
v_1} {(1+n\left|\{u_1+v_1\}\right|)^2(1+|u_1|)^{2\delta_{+}(p)}
(1+|v_1|)^{2\delta_{+}(p)}}\;,
\]
and
\[
B=\sum_{s=1}^{\gamma_j}\int_{\Delta_{j}^{(s)}
}\frac{|u_1|^{-2\delta_+(r)}|v_1|^{-2\delta_+(r')}\rmd u_1\;\rmd
v_1} {(1+n\left|\{u_1+v_1\}\right|)^2(1+|u_1|)^{2\delta_{+}(p)}
(1+|v_1|)^{2\delta_{+}(p)}}\;,
\]
where
\[
\Delta_{j}^{(s)}=\{(u_1,v_1)\in
(-\gamma_j\pi,\gamma_j\pi)^2,|u_1+v_1-2\pi s|\leq \pi\}\;,
\]
with $s\in\{-\gamma_j,\cdots,\gamma_j\}$.
This decomposition is similar to the one used in the
proof of Proposition~7.1 in~\cite{clausel-roueff-taqqu-tudor-2011b} and is
obtained by partitioning $(-\gamma_j\pi,\gamma_j\pi)^2$ using the
domains $\Delta_j^{(s)}$.

In the proof of \cite[Proposition~7.1]{clausel-roueff-taqqu-tudor-2011b},
bounds of $A$ and $B$ are provided in the case where $r=r'$ and
$\delta(r)>0$. It turns out that the same arguments apply in the present case
and yield
\begin{align*}
A&\leq
\begin{cases}
 C n^{-2+2\delta(r)+2\delta(r')}&\text{ if  $2\delta_+(r)+2\delta_+(r')>1$}\;,\\
Cn^{-1}(\log n)^{2\varepsilon'_1}\gamma_j^{\max(1-2\delta_+(r)-2\delta_+(r')-4\delta_+(p),0)}
(\log \gamma_j)^{2\epsilon'_2}& \text{otherwise,}
\end{cases}\\
B&\leq
Cn^{-1}\gamma_j^{\max(1-2\delta_+(r)-2\delta_+(p),0)+\max(1-2\delta_+(r')-2\delta_+(p),0)}
(\log \gamma_j)^{2\varepsilon'_3}\;,
\end{align*}
where
\begin{gather*}
\varepsilon'_1=\frac12\varepsilon'(2\delta_+(q-p)+2\delta_+(q'-p)),\,\varepsilon'_2=\frac12\varepsilon'(2(\delta_+(r)+\delta_+(r')+2\delta_+(p)))\;,\\
  \varepsilon'_3=\frac12\varepsilon'(2\delta_+(r)+2\delta_+(p))+\varepsilon'(2\delta_+(r')+2\delta_+(p))\;.
\end{gather*}
Hence we obtain that there exists some
$C>0$ depending only on $\delta_+(r),\delta_+(r'),\delta_+(p),d$ such that
\begin{multline*}
\tilde{I}_{n,j}\leq C
\left(n^{-\min(2(1-\delta_+(r)-\delta_+(r')),1)}(\log n)^{2\varepsilon'_1}\gamma_j^{\max(1-2\delta_+(r)-2\delta_+(r')-4\delta_+(p),0)}
(\log \gamma_j)^{2\varepsilon'_2}\right.\\
+\left.n^{-1}\gamma_j^{\max(1-2(\delta_+(r)+\delta_{+}(p)),0)+\max(1-2(\delta_+(r')+\delta_{+}(p)),0)}
(\log\gamma_j)^{2\varepsilon'_3}\right)\;.
\end{multline*}
Observe now that for any fixed $d$, there exists only a finite
number of possible values for
$\delta_+(r),\delta_+(r'),\delta_+(p)$ and then a finite number of
possible values for $C$. Then, provided we replace $C$ by its
maximum possible value, we can assume that $C$ only depends on $d$.

The bound on $\tilde{I}_{n,j}$ and~(\ref{e:boundSnj51}),(\ref{e:eq1}) then yields
\begin{align*}
\|S_{n,j}^{(q,q',p)}\|_{2}\leq
&C_4^{\frac{q+q'}{2}}\Lambda_2(q-p,p)^{1/2}\Lambda_2(q'-p,p)^{1/2}\,\gamma_j^{2K} \gamma_j^{-1+\delta_+(q-p)+\delta+(q'-p)+2\delta_+(p)}(\log\gamma_j)^{\varepsilon'_0}\\
&\times [n^{-\min(1-\delta_+(q-p)-\delta_+(q'-p),1/2)}(\log n)^{\varepsilon'_1}\gamma_j^{\max(\frac{1}{2}-\delta_+(q-p)-\delta_+(q'-p)-2\delta_+(p),0)}
(\log \gamma_j)^{\varepsilon'_2}\\
&+n^{-1/2}\gamma_j^{\max(1/2-\delta_+(q-p)-\delta_+(p),0)+\max(1/2-\delta_+(q'-p)-\delta_+(p),0)}(\log\gamma_j)^{\varepsilon'_3}] \;.
\end{align*}
Inequality~(\ref{e:boundL2P51a}) corresponds to this bound with
exponents of $\gamma_j$, $\log n$ and $\log \gamma_j$ simplified as follows.

The exponent of $\gamma_j$ is obtained by observing that
$-1+\delta_+(q-p)+\delta_+(q'-p)+2\delta_+(p)
=(-1/2+\delta_+(q-p)+\delta_+(p))+(-1/2+\delta_+(q'-p)+\delta_+(p))$ and using
$\max(-a,0)+a=\max(a,0)$ with $a=-1/2+\delta_+(q-p)+\delta_+(p)$ and
$a=-1/2+\delta_+(q'-p)+\delta_+(p)$ successively.

The log exponents are obtained by observing that, since $r\leq r'$,
$\varepsilon(r)+\varepsilon(r')+2\varepsilon(p)\leq
2(\varepsilon(r')+\varepsilon(p))\leq 4\varepsilon(r'\vee p)$.  In addition
$\varepsilon'(2\delta_+(m)+2\delta_+(m'))= 0$ iff $m+m'\neq 1/(1-2d)$ and
equals $1$ otherwise. Thus
$\varepsilon'(2\delta_+(m)+2\delta_+(m'))\leq\varepsilon(m+m')$ and we get
\[
\varepsilon'(2\delta_+(r)+2\delta_+(p))\leq \varepsilon(q),
\quad\text{and}\quad\varepsilon'(2\delta_+(r')+2\delta_+(p))\leq \varepsilon(q')\;.
\]
Finally, since
$\varepsilon$ is non--decreasing and $q\leq q'$, $r'\vee
p\leq q'$,
\[
\varepsilon'(2\delta_+(r)+2\delta_+(p))+\varepsilon'(2\delta_+(r')+2\delta_+(p))+4\varepsilon(r'\vee
p)\leq \varepsilon(q)+\varepsilon(q')+4\varepsilon(q')\leq 6\varepsilon(q')\;.
\]

\noindent{\it Proof of~\ref{item:pro:SD1ii}.} Here,
$p=0$ and thus $\widehat{\kappa}_j^{(p)}=\widehat{h}_j^{(K)\otimes
2}$. The same approach as in the proof of Proposition~7.2
in~\cite{clausel-roueff-taqqu-tudor-2011b} leads to the following inequality
which corresponds to~(7.12) in~\cite{clausel-roueff-taqqu-tudor-2011b}~:
\begin{equation}\label{e:boundSnj1}
\mathbb{E}[|S_{n,j}^{(q,q',0)}|^2]\leq
C_5\,\gamma_j^{-(q+q')(1-2d)}\gamma_j^{2(2K+1)}I_{n,j}
=C_5\gamma_j^{2(\delta(q)+\delta(q')+2K)}I_{n,j}\;,
\end{equation}
where
$$
I_{n,j}=\int_{u=-q\gamma_j\pi}^{q\gamma_j\pi}\int_{v=-q'\gamma_j\pi}^{q'\gamma_j\pi}g(u,v)J_{q,\gamma_j\pi}(u;d,\cdots,d)J_{q',\gamma_j\pi}(v;d,\cdots,d)\rmd
u\rmd v_{1}\;,
$$
with $J_{m,a}$ defined as in Lemma~\ref{lem:Ja} and with $g(u,v)$
defined for all $(u,v)\in\mathbb{R}^{2}$ by,
\begin{equation}\label{e:exprg}
g(u,v)= (1+|n\{u+v\}|)^{-2}
\frac{\left|\gamma_j\{u/\gamma_j\}\right|^{2(M-K)}\cdot\left|
\gamma_j\{v/\gamma_j \}\right|^{2(M-K)}}
{\left[(1+|\gamma_j\{u/\gamma_j \}|)
(1+|\gamma_j\{v/\gamma_j\}|)\right]^{2(M+\alpha)}} \;.
\end{equation}

As in the case $p\neq 0$, we can use the bound~(\ref{e:Jmaunified})
of $J_{m,a}$ and the inequality $|\{u\}|\leq |u|$. We get that
\begin{eqnarray*}
I_{n,j}&\leq& C_6^{q+q'} (q!q'!)^{1-2d}\gamma_j^{2\delta_-(q)
-\delta_-(q'))}\\
&&\times\int_{u=-q\gamma_j\pi}^{q\gamma_j\pi}\int_{v=-q'\gamma_j\pi}^{q'\gamma_j\pi}
\frac{ \left|\gamma_j\{\frac{u}{\gamma_j}\}\right|^{2M-2K-2\delta_+(q)}
\left|\gamma_j\{\frac{v}{\gamma_j}\}\right|^{2M-2K-2\delta_+(q')}\rmd
u\rmd
v}{(1+n|\{u+v\}|)^2\left(1+\left|\gamma_j\{\frac{u}{\gamma_j}\}\right|\right)^{2(M+\alpha)}
\left(1+\left|\gamma_j\{\frac{v}{\gamma_j}\}\right|\right)^{2(M+\alpha)}}\;.
\end{eqnarray*}

As in the proof of Proposition~7.2 of~\cite{clausel-roueff-taqqu-tudor-2011b}, we then
obtain that
\begin{equation}\label{e:boundSnj2}
I_{n,j}\leq C_7^{q+q'} (q!q'!)^{1-2d}
n^{-1}\gamma_j^{-2(\delta_{-}(q) +\delta_{-}(q'))}\;.
\end{equation}
The conclusion follows
from~(\ref{e:boundSnj1}) and~(\ref{e:boundSnj2}).

\section{Integral representations}\label{s:appendixA}
It is convenient to use an integral representation in the spectral
domain to represent the random processes (see for
example~\cite{major:1984,nualart:2006}). The stationary Gaussian
process $\{X_k,k\in\mathbb{Z}\}$ with spectral
density~(\ref{e:sdf}) can be written as
\begin{equation}\label{e:intrepX}
X_\ell=\int_{-\pi}^{\pi}\rme^{\rmi\lambda
  \ell}f^{1/2}(\lambda)\rmd\widehat{W}(\lambda)=\int_{-\pi}^{\pi}\frac{\rme^{\rmi\lambda
    \ell}f^{*1/2}(\lambda)}{|1-\rme^{-{\rmi}\lambda}|^d}\rmd\widehat{W}(\lambda),\quad\ell\in\mathbb{N}\;.
\end{equation}
This is a special case of
\begin{equation}\label{e:int}
\widehat{I}(g)=\int_{\mathbb{R}}g(x)\rmd\widehat{W}(x),
\end{equation}
where $\widehat{W}(\cdot)$ is a complex--valued Gaussian random
measure satisfying, for any Borel sets $A$ and $B$ in
$\mathbb{R}$, $\mathbb{E}(\widehat{W}(A))=0$,
$\mathbb{E}(\widehat{W}(A)\overline{\widehat{W}(B)})=|A\cap B|$
and
$$
\widehat{W}(A)=\overline{\widehat{W}(-A)}\;.
$$
The integral~(\ref{e:int}) is defined for any function $g\in
L^2(\mathbb{R})$ and one has the isometry
\[
\mathbb{E}(|\widehat{I}(g)|^2)=\int_{\mathbb{R}}|g(x)|^2\rmd x\;.
\]
The integral $\widehat{I}(g)$, moreover, is real--valued if
\[
g(x)=\overline{g(-x)}\;.
\]

We shall also consider multiple It\^{o}--Wiener integrals
\[
\widehat{I}_q(g)=\int^{''}_{\mathbb{R}^q}g(\lambda_1,\cdots,\lambda_q)\rmd
\widehat{W}(\lambda_1)\cdots\rmd\widehat{W}(\lambda_q)
\]
where the double prime indicates that one does not integrate on
hyperdiagonals $\lambda_i=\pm \lambda_j,i\neq j$. The integrals
$\widehat{I}_q(g)$ are handy because we will be able to expand our
non--linear functions $G(X_k)$ introduced in Section~\ref{s:intro}
in multiple integrals of this type.

These multiples integrals are defined for
$g\in\overline{L^2}(\mathbb{R}^{q},\mathbb{C})$, the space of
complex valued functions defined on $\mathbb{R}^{q}$ satisfying
\begin{gather}\label{e:antisym}
g(-x_{1},\cdots,-x_{q})= \overline{g(x_{1},\cdots, x_{q})}\mbox{ for }(x_{1},\cdots, x_{q}) \in \mathbb{R}^q\;,\\
\label{e:fL2} \Vert g\Vert ^{2} _{L^{2}}:= \int_{\mathbb{R}^{q}}
\left| g(x_{1},\cdots, x_{q}) \right| ^{2} \rmd x_{1}\cdots \rmd
x_{q} <\infty\;.
\end{gather}
Hermite polynomials are related to multiple integrals as follows :
if $X=\int_{\mathbb{R}}g(x)\rmd\widehat{W}(x)$ with
$\mathbb{E}(X^2)=\int_{\mathbb{R}}|g(x)|^2\rmd x=1$ and
$g(x)=\overline{g(-x)}$ so that $X$ has unit variance and is
real--valued, then
\begin{equation}\label{e:herm-integ}
H_q(X)=\widehat{I}_q(g^{\otimes
q})=\int_{\mathbb{R}^q}^{''}g(x_1)\cdots g(x_q)\rmd
\widehat{W}(x_1)\cdots\rmd\widehat{W}(x_q)\;.
\end{equation}
Since $X$ has unit variance, one has for any $\ell\in\mathbb{Z}$,
\begin{align*}
H_q(X_\ell)&= H_q\left(\int_{-\pi}^{\pi}\rme^{\rmi\xi
\ell}f^{1/2}(\xi)
\rmd\widehat{W}(\xi)\right)\\
&=\int_{(-\pi,\pi]^q}^{''}
\rme^{\rmi\ell(\xi_1+\cdots+\xi_q)}\times
\left(f^{1/2}(\xi_1)\times\cdots\times f^{1/2}(\xi_q)\right)
\;\rmd\widehat{W}(\xi_1)\cdots \rmd\widehat{W}(\xi_q)\;.
\end{align*}
Then by~(\ref{e:Wjk}), we have
\begin{equation}\label{eq:Wjk_representation}
\bW_{j,k}^{(q)}=\sum_{\ell\in\mathbb{Z}}\bh_j^{(K)}(\gamma_j
k-\ell)H_q(X_\ell)=\widehat{I}_q(\mathbf{f}_{j,k}^{(q)})
\end{equation}
with
\begin{equation}\label{e:fjk_representation1}
\mathbf{f}_{j,k}^{(q)}(\xi_1,\cdots,\xi_q)= \rme^{\rmi
k\gamma_j(\xi_1+\cdots+\xi_q)}\times
\mathbf{\widehat{h}}_{j}^{(K)}(\xi_1+\cdots+\xi_q)
f^{1/2}(\xi_1)\cdots f^{1/2}(\xi_q) \1_{(-\pi,\pi)}^{\otimes
q}(\xi) \;,
\end{equation}
because
\begin{eqnarray*}
\sum_{\ell\in\mathbb{Z}}\rme^{\rmi\ell(\xi_1+\cdots+\xi_q)}\bh_j^{(K)}(\gamma_j
k-\ell)&=&\rme^{\rmi\gamma_j
k(\xi_1+\cdots+\xi_q)}\sum_{u\in\mathbb{Z}}\rme^{-\rmi
u(\xi_1+\cdots+\xi_q)}\bh_j^{(K)}(u)\\
&=&\rme^{\rmi\gamma_j
k(\xi_1+\cdots+\xi_q)}\mathbf{\widehat{h}}_j^{(K)}(\xi_1+\cdots+\xi_q)\;,
\end{eqnarray*}
by~(\ref{e:dF}).

The following proposition can be found
in~\cite{peccati:taqqu:2011}, Formula~(9.7.32). It is an extension to our
complex--valued setting of a corresponding result
in~\cite{nualart:2006} for multiple integrals in a real--valued
setting.%\frnote{$(2\pi)^p$ ??? PAPIER 2 AUSSI!!}
\begin{proposition}\label{pro:intproduct}
  Let $(q,q')\in \mathbb{N}^2$. Assume that $f,g$ are two symmetric functions
  belonging respectively to $\overline{L^2}(\mathbb{R}^q)$ and
  $\overline{L^2}(\mathbb{R}^{q'})$ then the following product formula holds :
\begin{equation}\label{EqProdIntSto}
\widehat{I_q}(f)\widehat{I_{q'}}(g)=\sum\limits_{p=0}^{q\wedge q'}
 p
!\begin{pmatrix}q\\p\end{pmatrix}\begin{pmatrix}q'\\p\end{pmatrix}\widehat{I_{q+q'-2p}}(f\overline{\otimes}_p
g),
\end{equation}
where for any $p\in\{1,\cdots,q\wedge q'\}$
\begin{equation}\label{e:times-p}
(f\overline{\otimes}_p
g)(t_1,\cdots,t_{q+q'-2p})=\int_{\mathbb{R}^p}f(t_1,\cdots,t_{q-p},s)g(t_{q-p+1},\cdots,t_{q+q'-2p},-s)\rmd^p
s\;.
\end{equation}
\end{proposition}
\section{The wavelet filters}\label{s:appendixB}
The sequence $\{Y_{t}\}_{t\in\mathbb{Z}}$ can be formally
expressed as
$$
Y_{t}=\Delta^{-K}G(X_{t}), \quad t\in\mathbb{Z}\;.
$$
The study of the asymptotic behavior of the scalogram of $\{Y_{t}\}_{t\in\mathbb{Z}}$ at different scales involve multidimensional wavelets coefficients of
$\{G(X_t)\}_{t\in\mathbb{Z}}$ and of $\{Y_t\}_{t\in\mathbb{Z}}$.
To obtain them, one applies a multidimensional linear filter
$\bh_j(\tau),\tau\in\mathbb{Z}=(h_{j,\ell}(\tau))$, at each scale index $j\geq 0$. We shall
characterize below the multidimensional filters $\bh_j(\tau)$ by their discrete
Fourier transform~:
\begin{equation}\label{e:dF}
\widehat{\bh}_j(\lambda)=\sum_{\tau\in\mathbb{Z}}\bh_j(\tau)
\rme^{-\rmi \lambda\tau },\,\lambda\in [-\pi,\pi]\;,\quad
\bh_j(\tau)=\frac{1}{2\pi}\int_{-\pi}^{\pi}\widehat{\bh}_j(\lambda)\rme^{\rmi\lambda\tau}\rmd\lambda,\tau\in\mathbb{Z}\;.
\end{equation}
The resulting wavelet coefficients $\bW_{j,k}$, where $j$ is the
scale index and $k$ the location are defined as
\begin{equation}\label{e:W1}
\bW_{j,k}=\sum_{t\in\mathbb{Z}}\bh_j(\gamma_j
k-t)Y_{t}=\sum_{t\in\mathbb{Z}}\bh_j(\gamma_j
k-t)\Delta^{-K}G(X_{t}),\,j\geq 0, k\in\mathbb{Z},
\end{equation}
where $\gamma_j\uparrow \infty$ as $j\uparrow \infty$ is a
sequence of non--negative scale factors applied at scale index $j$, for
example $\gamma_j=2^j$.  We do not assume that the wavelet
coefficients are orthogonal nor that they are generated by a
multiresolution analysis. Our assumption on the filters
$\bh_j=(h_{j,\ell})$ are as follows~:
\begin{enumerate}[label=(W-\alph*)]
\item\label{ass:w-a} \underline{Finite support}: For each $\ell$ and $j$,
  $\{h_{j,\ell}(\tau)\}_{\tau\in\mathbb{Z}}$ has finite support.
\item\label{ass:w-b} \underline{Uniform smoothness}: There exists
$M\geq K$, $\alpha>1/2$
  and $C>0$ such that for all $j\geq0$ and $\lambda\in [-\pi,\pi]$,
    \begin{equation}\label{e:majoHj}
    |\widehat{\bh}_j(\lambda)|\leq \frac{C\gamma_j^{1/2}|\gamma_j\lambda|^M}{(1+\gamma_j|\lambda|)^{\alpha +M}}\;.
    \end{equation}
By $2\pi$-periodicity of $\widehat{h}_j$ this inequality can be
extended to $\lambda\in\mathbb{R}$ as
\begin{equation}\label{EqMajoHjR}
|\widehat{\bh}_{j}(\lambda)|\leq C
\frac{\gamma_j^{1/2}|\gamma_j\{\lambda\}|^M}
{(1+\gamma_j|\{\lambda\}|)^{\alpha+M}}\;.
\end{equation}
where $\{\lambda\}$ denotes the element of $(-\pi,\pi]$ such that
$\lambda-\{\lambda\}\in2\pi\mathbb{Z}$.
\item\label{ass:w-c}
\underline{Asymptotic behavior}: There exists a sequence of phase functions $\Phi_j :\mathbb{R}\rightarrow (-\pi,\pi]$ and some non identically
  zero function $ \widehat{\bh}_{\infty}$ such that
\begin{equation}\label{EqLimHj}
\lim_{j \to
+\infty}(\gamma_j^{-1/2}\widehat{\bh}_{j}(\gamma_j^{-1}\lambda))=
\widehat{\bh}_{\infty}(\lambda)\;,
\end{equation}
locally uniformly on $\lambda\in\mathbb{R}$.
\end{enumerate}
In~\ref{ass:w-c} {\it locally uniformly} means that for all compact
$K\subset \mathbb{R}$,
\[
\sup_{\lambda\in K}\left|\gamma_j^{-1/2}\widehat{\bh}_j(\gamma_j^{-1}\lambda)\rme^{\rmi \Phi_j(\lambda)}-\widehat{\bh}_{\infty}(\lambda)\right|\to 0\;.
\]

Assumptions (\ref{e:majoHj}) and~(\ref{EqLimHj}) imply that for
any $\lambda\in\mathbb{R}$,
\begin{equation}\label{EqMajoHinf}
|\widehat{\bh}_{\infty}(\lambda)|\leq
C\frac{|\lambda|^M}{(1+|\lambda|)^{\alpha+M}}\;.
\end{equation}
Hence $\widehat{\bh}_{\infty}$ has entries in $L^{2}(\mathbb{R})$. We let
$\bh_{\infty}$ be the vector of $L^2(\mathbb{R})$ inverse Fourier transforms of
$\widehat{h}_{\ell,\infty}$, that is
\begin{equation}\label{eq:TFdef}
\widehat{\bh}_{\infty}(\xi)=\mathfrak{F}(\bh_{\infty})(\xi)=\int_{\mathbb{R}^q}
\bh_{\infty}(t)\rme^{-\rmi t^T\xi}\;\rmd^q t,\quad
\xi\in\mathbb{R}^q\;,
\end{equation}
is defined for any $f\in L^2(\mathbb{R}^{q},\mathbb{C})$.

Observe that while $\widehat{\bh}_j$ is $2\pi$--periodic, the
function $\widehat{\bh}_{\infty}$ has non--periodic
entries on $\mathbb{R}$. For the connection between these
assumptions on $h_j$ and corresponding assumptions on the scaling
function $\varphi$ and the mother wavelet $\psi$ in the classical
wavelet setting see \cite{moulines:roueff:taqqu:2007:jtsa}. In
particular, in that case, one has
$\widehat{h}_{\infty}=\widehat{\varphi}(0)\overline{\widehat{\psi}}$.

A more convenient way to express the wavelet coefficients
$\bW_{j,k}$ defined in~(\ref{e:W1}) is to incorporate the linear
filter $\Delta^{-K}$ in~(\ref{e:W1}) into the filter $\bh_j$ and
denote the resulting filter $\bh_j^{(K)}$. Then
\begin{equation}\label{e:W2}
\bW_{j,k}=\sum_{t\in\mathbb{Z}}\bh_j^{(K)}(\gamma_j k-t)G(X_{t})\;,
\end{equation}
where
\begin{equation}\label{e:HjkVSHj}
\widehat{\bh}_j^{(K)}(\lambda)=(1-\rme^{-{\rmi}\lambda})^{-K}
\widehat{\bh}_j(\lambda)
\end{equation}
is the discrete Fourier transform of $\bh_j^{(K)}$.
Using~(\ref{EqMajoHjR}) we get,
\begin{equation}\label{e:hK}
\left| \widehat{\bh}_j^{(K)}(\lambda) \right| \leq C
  \gamma_j^{1/2+K}\;\frac{|\gamma_j\{\lambda\}|^{M-K}}{(1+\gamma_j|\{\lambda\}|)^{\alpha+M}},\quad\lambda\in\mathbb{R},\,j\geq 1\;.
\end{equation}
In particular, since we assume if $M\geq K$, we get
\begin{equation}
  \label{eq:M_replaced_by_K}
  \left| \widehat{\bh}_j^{(K)}(\lambda) \right| \leq C
  \gamma_j^{1/2+K}\;(1+\gamma_j|\{\lambda\}|)^{-\alpha-K},
  \quad\lambda\in\mathbb{R},\,j\geq 1\;.
\end{equation}

By Assumption~(\ref{e:majoHj}), $\bh_j$ has null moments up to order
$M-1$, that is, for any $m\in\{0,\cdots,M-1\}$,
\begin{equation}\label{e:mom}
\sum_{t\in \mathbb{Z}}\bh_j(t)t^m =0\;.
\end{equation}
Observe that $\Delta^K Y$ is centered by definition. However,
by~(\ref{e:mom}), the definition of $\bW_{j,k}$ only
depends on $\Delta^M Y$. In particular, provided that $M\geq K+1$,
its value is not modified if a constant is added to $\Delta^K Y$,
whenever $M\geq K+1$.

\bibliographystyle{plainnat}
\bibliography{lrd}
\end{document}